\documentclass[11pt,leqno]{amsproc} 

\usepackage{epsfig}
\usepackage[utf8]{inputenc}
\usepackage{framed}

\usepackage{url}
\usepackage{mathtools}
\mathtoolsset{showonlyrefs}

\usepackage[T1]{fontenc}

\usepackage{amssymb}
\usepackage{graphicx, color,blindtext}
\usepackage{caption}
\usepackage{subcaption}
\usepackage[export]{adjustbox}
\usepackage[titletoc,toc]{appendix}

\usepackage{floatrow}
\usepackage{dsfont}
\usepackage{stmaryrd}
\usepackage{mathrsfs}
\usepackage{graphicx}
\usepackage{float}
\usepackage{algorithm}
\usepackage{algpseudocode}
\usepackage{tikz}
\usepackage{epstopdf}

\newtheorem{theorem}{Theorem}

\newtheorem{lemma}{Lemma}
\newtheorem{remark}{Remark}
\newtheorem{corollary}{Corollary}

\newtheorem{definition}{Definition}
\newcommand{\C}{\mathcal{C}}

\newcommand{\D}{\mathcal{D}}

\newcommand{\K}{\mathcal{K}}
\newcommand{\M}{\mathcal{M}}

\newcommand\Lip{\mathrm{Lip}}
\newcommand\Lipu{{\mathrm{Lip}}_1}
\newcommand\Omb{\bar{\Omega}}

\newcommand\wW{{\widetilde{W}}_1}

\DeclareMathOperator{\argmax}{argmax}
\DeclareMathOperator{\argmin}{argmin}

\newcommand\CC{\hbox{C\kern -.58em {\raise .54ex \hbox
			{$\scriptscriptstyle |$}}
		\kern-.55em {\raise .53ex \hbox{$\scriptscriptstyle |$}} }}
\newcommand\qd{\hfill$\sqcap\kern-8.0pt\hbox{$\sqcup$}$}
\newcommand\NN{\hbox{I\kern-.2em\hbox{N}}}
\newcommand\nn{\hbox{I\kern-.2em\hbox{N}}}
\newcommand\RR{I\!\!R}
\newcommand\sRR{{\sl \hbox{I\kern-.2em\hbox{R}}}}
\newcommand\QQ{\hbox{I\kern-.53em\hbox{Q}}}

\everymath{\displaystyle}

\usepackage{stackengine}
\usepackage{scalerel}
\usepackage{xcolor,amssymb}

\newcommand{\cs}{$^\dagger$} \newcommand{\cm}{$^\ddagger$} \newcommand{\cmm}{$^\mp$}

\title{Congested Crossing  Pedestrian Traffic Flow  : Dispersion vs. Transport in Crowded Areas}

\author[M. Al Khatib]{Mariam Al Khatib\cm}
\author[S. Gounane]{Said Gounane\cs}
\author[N. Igbida]{Noureddine Igbida\cm}
\author[G. Jradi]{Ghadir Jradi\cmm }
\thanks{\cm Institut de recherche XLIM-DMI, UMR-CNRS 6172, Facult\'e des Sciences et Techniques, Universit\'e de Limoges, France. Emails:   noureddine.igbida@unilim.fr,  mariam.al\_khatib@unilim.fr}
\thanks{\cs MIMSC Laboratory, Higher School of Technology Essaouira, Morocco. Emails:   s.gounane@uca.ac.ma}
\thanks{\cmm GJ began this research project while completing her doctoral thesis. She is currently working on it independently, unaffiliated with any institution, ghadirjradi@gmail.com}

\date{\today}

\begin{document}
	 
	\maketitle

	\begin{abstract}  This study investigates the complex dynamic interactions between two typed populations coexisting within a shared space. We propose both theoretical and numerical study to analyze scenarios where one population (population $1$) must traverse a territory occupied by another (population $2$), necessitating strategies to mitigate overcrowding caused by spatial limitations.   To capture these interactions, we model population $1$ using a linear transport equation, while population $2$ is described by a granular diffusion model à la sandpile to represent its internal dynamics and tendency to decongest.
		Through numerical simulations, we explore how different movement strategies of the traversing population (population $1$) – including directed motion towards a specific destination, internal dispersion to minimize crowding, and uniform dispersal across the space – affects the behavior of population $2$.  
	\end{abstract}
	
	\textbf{keywords: }Crowed motion; congestion, transport equation, $1-$Wasserstein distance, $W_1-$gradient flow, minimum-flow problem, primal-dual numerical optimization.
	
	\textbf{AMS Subject Classification:} 22E46, 53C35, 57S20.
	
	\section{Introduction}
	\medskip
	Consider two distinct populations $\rho_1$ and $\rho_2$, occupying a common space, denoted by $\Omega.$ Population $\rho_1$ must traverse an area inhabited by $\rho_2$ requiring a strategic interaction to manage congestion.  Given the limited spatial capacity, the dynamic interplay between these populations can easily lead to congestion.
	
	In this paper, we develop a macroscopic mathematical model for these congested pedestrian flows. By treating the crowd as a collective entity, we can apply this approach to large-scale crowd dynamics. Building upon previous works (\cite{Bord,Helbing1,Helbing2}), we use the continuity  equation to describe the macroscopic dynamics of each population:
	\begin{equation}\label{transport1}
		\partial_t \rho_i +\nabla \cdot (\rho_i\: U_i) =0, \quad \hbox{ for }i=1,2,
	\end{equation}
	where  $\rho_i=\rho_i(t,x)$ represents the density of individuals at time $t\geq 0$ and position $x\in \Omega\subset  \RR^N$ $(N=2),$ assumed to be a bounded and regular domain.  The dynamic  of each $\rho_i$  must ensure an admissible global distribution of the population : 
	\begin{equation}
		0\leq \rho_1+\rho_2\leq \rho_{max},
	\end{equation} 
	where $\rho_{max}$  a given maximum density    (assumed to be known).  Without loss of generality, we assume throughout the paper that 
	$$\rho_{max}=1. $$ 
	 
	\medskip
	Determining the flow velocity vector field $U_i$ remains a challenging task, in general. Although various approaches have been explored, a universal solution is elusive. Balancing the general behavior of the crowd with the behavior of the individual pedestrians is critical. 
	Our approach assumes that while $\rho_1$ moves toward a fixed target, $\rho_2$ can adjust its movement based on the instantaneous distribution of both populations. For example, assuming that $U_1=V$ is a known spontaneous velocity field,  we assume that $U_2$  is a vector field that performs the patch for the spontaneous dynamic when the pedestrian $\rho_2$ is hindered by the other one.  We propose to consider a dynamic vector field $U_2$  in congested regions given by :
	\begin{equation}
		U_2 = U_2[\nabla p],
	\end{equation}
	where $p$ is a Lagrange multiplayer associtae with the constraint  $0\leq \rho_1+\rho_2\leq \rho_{max}:$ $p\geq 0$ and $p(\rho_{max}-\rho_1-\rho_2)=0.$ 
	This model ensures that $\rho_1+\rho_2$ does not exceed $\rho_{max}$ and provides a framework for analyzing the dynamics of these two populations in congested environments.

	\medskip
	Inspired by ~\cite{EIJ},  we introduce a grain-like motion velocity field sandpile to model the random microscopic movements of $\rho_2$ agents, facilitating the passage of $\rho_1$ while adhering to spatial distribution constraints. 
	More precisely,  while the population $\rho_1$ follows the transport equation 
	\begin{equation}\label{transport0}
		\partial_t \rho_1 +\nabla \cdot (\rho_1\: V) =0	,\quad \hbox{ in }Q:=(0,T)\times \Omega,\end{equation}
	with $\rho_1(0)=\rho_{01},$   we propose to handle the dynamic  of the second population  $\rho_2,$ by using the nonlinear second order equation  where $\rho_2= \rho_{max}-\rho_1$ workout the utmost distribution of the population in $\Omega$  :    
	\begin{equation}\label{EvolGran1}
		\left\{ 
		\begin{array}{ll} 
			\displaystyle \frac{d \rho_2 }{d t}+  \partial {I\!\! I}_{\Lipu} (p)=  0  \\  \\    0\leq m,\: 	p\geq 0, \: 0\leq \rho_2\leq \rho_{max}-\rho_1,\: p(\rho_{max}-\rho_1-\rho_2)=0
		\end{array} 
		\right.
	\end{equation}
	with $\rho_2(0)=\rho_{02},$ where $\partial {I\!\! I}_{\Lipu}$ is the subdifferential of the indicator faction of $\Lipu,$ the set of $1-$Lipschitz continuous function.

	\medskip 
	While the natural boundary condition for the transport equation  \eqref{transport0}  incorporates the normal trace of the velocity field $V$ (assumed to be known),  we examine the dynamics governed by \eqref{EvolGran1} under Neumann boundary conditions. However, \eqref{transport0}  still requires Dirichlet boundary conditions on portions of the boundary where V points inward.
	
	\medskip
	
	The model's adaptability to various practical scenarios is a key feature. The present paper focuses on a scenario in which only the population $\rho_2$ experiences congestion, with a priority placed on the movement of $\rho_1$. However, the framework can be extended to model situations in which both populations exhibit congestion and transport behaviors. This type of coupled system has been explored theoretically in ~\cite{IgCD}. The model presented here specifically describes the dynamics of two populations: one with a primary objective (e.g., emergency evacuation) and another that adapts to facilitate the first's passage while maintaining acceptable density. We conducted an analysis of various transportation types for  $\rho_1$ : (1) directed movement towards a goal, modeled by a velocity field $V$ derived from an eikonal equation;   (2) motion influenced by local average density, modeled by a nonlocal potential for $V,$ and also  motion   inspired by the Keller-Segel models, where the potential is determined by a diffusion equation associated with $\rho_2.$

	
	\medskip
	To our knowledge, this study is the first to address congestion within this specific context. While research on cross-diffusion systems, where populations exhibit behaviors to alleviate overcrowding (akin to 'overcrowding dispersal' in our model), provides some relevant background, these studies primarily focus on segregation phenomena. In contrast, our work examines aggregation scenarios where both populations interact within the same space. Our model introduces distinct dynamics for each population: transport for $\rho_1$ and granular dynamics for $\rho_2.$ Future work will investigate scenarios where both populations exhibit both transport and congestion behaviors. 
 \medskip
	In the following section, we present the main results for the transport equation of $\rho_1,$ considering a bounded $W^{1,1}$-velocity field $V$ with bounded divergence.
		
	Section $3$ introduces a solution concept for the evolution problem  \eqref{EvolGran1} and establishes its existence using an implicit time discretization scheme within the framework of the $1-$Wasserstein distance, $W_1,$ and the Kantorovich potential. This approach treats problem \eqref{EvolGran1} in the flavor of a $W_1$-gradient flow in the dual space of $\Lipu.$ 
	
	Section $4$ outlines the numerical algorithm used in this study. Adapting the prediction-correction approach from ~\cite{EIJ}, we first transport $\rho_1$ according to the transport equation \eqref{transport0}. Subsequently, an implicit time discretization scheme is applied for the dynamics described in  \eqref{EvolGran1}.   The discrete dynamics for $\rho_2$  are then determined by a projection operator onto the set of admissible densities, ensuring $0\leq \rho_1+\rho_2\leq 1,$ which can be efficiently implemented using a minimal flow approach.
	
	Finally, Section $6$ presents a series of numerical simulations to illustrate the  behavior of the model.

	\section{Transport flow in population  $\rho_1$}
	\setcounter{equation}{0}
	 
	 	\subsection{Reminder on   linear transport equation in a bounded domain}
	
	In the context of directed crowd movement, where individuals follow a predefined path towards a specific destination, the velocity vector field $V$ plays a crucial role. To accurately model this behavior, $V$ is typically specified with a prescribed normal flux on the boundary. This ensures that the crowd movement aligns with the desired dynamics at the domain's edges.

	We consider the transport equation 
	\begin{equation}\label{PDEtransport0}
		\left\{ 
		\begin{array}{ll} 
			\displaystyle \frac{\partial  u }{\partial t}  + \nabla\cdot (   u\: V )= 0\quad & \hbox{ in } Q\\  \\   
			u(0)= u_0 & \hbox{ in }\Omega, 
		\end{array} 
		\right.
	\end{equation} 
	where   $V$  is assumed to satisfy  
	\begin{itemize}
		\item[$(T1)$] $\displaystyle V \in  L^1(0,T;W^{1,1}(\Omega)^N)\cap L^\infty(Q)^N$ 
		\item[$(T2)$]  $  \nabla \cdot V\in L^\infty(Q) .$
	\end{itemize}

	Under the previous assumptions, $V\cdot \nu \in L^\infty((0,T)\times \partial\Omega),$ where $\nu$ is   the outward unit normal vector to the boundary  $\partial \Omega,$ is well defined.  The behavior of the vector field $V$ significantly influences population dynamics. In a closed area, where the population is confined, $V\cdot \nu =0 $  on the boundary $\partial \Omega,$  signifying no population flux across the domain's edges. In an open domain, the sign of $V\cdot \nu  $  determines the direction of population flow: $V\cdot \nu \geq 0  $ allows for population outflow, while $V\cdot \nu  \leq 0 $ indicates population inflow. More complex scenarios arise when the sign of $V\cdot \nu  $ changes along the boundary. For instance, $V\cdot \nu  $ may be negative on some parts of the boundary, allowing inflow, and positive on others, enabling outflow, resulting in a dynamic where population enters and exits the domain through different sections of its boundary. This variability in the normal trace of $V$  significantly influences the overall population dynamics within the domain.

	Following   ~\cite{Donadello1,Donadello2}, to incorporate the boundary condition for \eqref{PDEtransport0}, we define the inflow boundary as :
	$$\Sigma^- := \Big\{(t,x)\in (0,T)\times \partial \Omega \: :\:  V\cdot \nu <0\Big\},$$ 
	and consider the following evolution problem 
	
	\begin{equation} \label{PDEtransportBC}
		\left\{ 
		\begin{array}{ll} 
			\displaystyle \frac{\partial  u }{\partial t}  + \nabla\cdot (   u\: V )= 0 \quad & \hbox{ in } Q\\  \\   
			u   =0 & \hbox{ on }\Sigma^-  \\  \\  
			u(0)= u_0 & \hbox{ in }\Omega.
		\end{array} 
		\right.
	\end{equation}

	\begin{theorem}\label{theorho1V} 
		Assume $V$ satisfies the assumption $(T1)$ and $(T2),$   and  $  u_0\in L^\infty(\Omega)$,   Then,   
		the problem \eqref{PDEtransportBC}    has a unique   weak solution $ u$ in the sense that 	:   $ u\in L^\infty(Q)$,  $( u\: V)\cdot \nu \in L^\infty(\Sigma) $, $( u\: V)\cdot \nu   =0 ,$ $\mathcal H^ {N-1}-$a.e. on $\Sigma^- ,$ and  
		\begin{equation}\label{weakF0}
			\frac{d}{dt} 	\int_\Omega    u\: \xi \: dx -\int_\Omega  u\: V\cdot \xi \: dx = 0 , \hbox{ in } \D([0,T)),
		\end{equation}	 
		for any $\xi \in \C^\infty_c(\Omega).$   In particular, $( u\: V)\cdot \nu   $ is uniquely well defined, and we have    
		\begin{equation} \label{renorm}
			\begin{array}{c} 
				\frac{d}{dt} \int_\Omega  u \: \xi   -\int_\Omega   u\: V\cdot  \nabla \xi    =\int_\Omega    f\:   \xi  -  \int_{\partial\Omega} \xi\:    u  \: V \cdot \nu \: d\mathcal H^{N-1}, \hbox{ in } \D([0,T)).
		\end{array} 	 \end{equation}  
		If moreover,   $V$ satisfies 
		\begin{equation}\label{divpositif}
			\nabla \cdot V \geq 0 \quad \hbox{ a.e. in }Q.
		\end{equation} 
		and 	 $ 0\leq  u_0 ,$  then  
		\begin{equation}
			0\leq   u(t)\leq \Vert  u_0 \Vert_\infty  ,\quad \hbox{ a.e. in }Q. 
		\end{equation} 
	\end{theorem}
	\begin{proof}
		The results of this theorem are more or less well known by now. For the proof we refer the reader to the papers  ~\cite{Donadello1,Donadello2} and also  ~\cite{Ambrosio,AC,ACM}.
	\end{proof}

	\begin{remark}
		Beyond boundary conditions, the divergence of the velocity field, $\nabla \cdot V$,  plays a crucial role in population dynamics, particularly concerning congestion phenomena. Expansive velocity fields, characterized by $\nabla \cdot V\geq 0,$ tend to prevent excessive population densities, while compressive velocity fields, where $\nabla \cdot V<0  $  can potentially lead to congestion and overcrowding.
	\end{remark}

	\begin{remark}
		Taking $\xi\equiv 1$ in the formulation \eqref{renorm}, we obtain the following mass conservation   property 
		\begin{equation}\label{mass conservation}
			\frac{d}{dt} 	\int_\Omega    u\:dx    =   -  \int_{\partial\Omega  }    \underbrace{  u\: V \cdot \nu}_{\geq 0} \: d\mathcal H^{N-1} , \hbox{ in } \D([0,T)).
		\end{equation}	 
		
	\end{remark}

	\begin{corollary} \label{CorVpostive}
		Under the assumptions of Theorem \ref{theorho1V}, if moreover 
		\begin{equation}\label{outwardV}
			V\cdot \nu \geq 0,\quad  \mathcal H^{N-1}-\hbox{ a.e. on }\Sigma, 
		\end{equation}
		then, for any   $    u_0\in L^\infty(\Omega)$,
		the problem 
		\begin{equation} \label{PDEtransportNB1}
			\left\{ 
			\begin{array}{ll} 
				\displaystyle \frac{\partial  u }{\partial t}  + \nabla\cdot (   u\: V )=  0 \quad & \hbox{ in } Q\\  \\    
				u(0)= u_0 & \hbox{ in }\Omega, 
			\end{array} 
			\right.
		\end{equation}    
		has a unique   weak solution $ u$ in the sense that 	:   $ u\in L^\infty(Q)$ and  
		\begin{equation}\label{weakF1}
			\frac{d}{dt} 	\int_\Omega    u\: \xi \: dx -\int_\Omega  u\: V\cdot \xi \: dx = 0 , \hbox{ in } \D([0,T)),
		\end{equation}	 
		for any $\xi \in \C^\infty_c(\Omega).$   In particular, $( u\: V)\cdot \nu \in L^\infty(\Sigma) $ is uniquely well defined. 
		
	\end{corollary}
	\begin{proof}
		This is a   consequence of Theorem \ref{theorho1V} and the fact that, under the assumption \eqref{outwardV}, 	$\Sigma^-=\emptyset.$ 
	\end{proof}

	
	\subsection{Application to directed crowd motion for population   $\rho_1$}

	In this paper, we focus on admissible scenarios where the vector field $V$ is such that the dynamics of population $\rho_1$ do not lead to congestion in the absence of population $\rho_2.$  So,   we assume that $V$ is a given vector field  satisfying the assumptions $(T1)$, $(T2)$ and moreover \eqref{divpositif}.
	Then, let us consider $0\leq \rho_{01}\leq 1$ representing the given initial density of population   $\rho_1.$ Below, we present the key theoretical results which follow directly from the previous section. 
	Notice that the assumptions on $V$  we are dealing with are such the population is not concerned with the congestion ; i.e. $0\leq \rho_1\leq 1$ a.e. in $Q.$ However, the evolution of $\rho_1$  will inevitably compel population $\rho_2$ to react in order to avoid the global congestion (the phenomenon concerning $\rho_2$ is treated in Section $3$).

	\bigskip\noindent 
	\underline{\textbf{No population flux   in a closed area}:} here we assume that no population exchange across certain sections of the boundary, represented by zero normal flux, effectively acting as walls.  In this case the normal trace of $V$ is such that 
	$$V \cdot \nu =0,\quad  \hbox{  on   }\Sigma.$$  
	Then, the evolution problem 	
	\begin{equation}  
		\left\{ 
		\begin{array}{ll} 
			\displaystyle \frac{\partial  \rho_1 }{\partial t}  + \nabla\cdot (   \rho_1\: V )=  0 \quad & \hbox{ in } Q\\  \\    
			\rho_1(0)= \rho_{01} & \hbox{ in }\Omega, 
		\end{array} 
		\right.
	\end{equation}     
	has a unique solution $\rho_1,$ which satisfies  :  
	\begin{itemize} 
		\item  $ \rho_1\in L^\infty(Q)$ 
		\item $ 0\leq \rho_1\leq 1,$ a.e. in $Q$ 
		\item $( \rho_1\: V)\cdot \nu =0,$ $ \mathcal H^{n-1}-$a.e. on $\Sigma$ 
		\item for any $\xi \in \C^\infty (\overline \Omega),$   
		\begin{equation} 
			\frac{d}{dt} 	\int_\Omega    \rho_1\: \xi \: dx -\int_\Omega  \rho_1\: V\cdot \xi \: dx =  0, \hbox{ in } \D([0,T)).
		\end{equation}	  
		\item for a.e. $t\in (0,T),$ we have 
		$$\int_\Omega \rho_1(t,x)\: dx =\int_\Omega \rho_{01}(x)\: dx.$$
	\end{itemize}

	\bigskip\noindent 
	\underline{\textbf{Population outflow in an  open  area:}}  here we assume  that we have  an outflow of population, modeled by an outward normal flux at exit points.  In this case  
	$$V \cdot \nu  \geq 0,\quad  \hbox{  on   }\Sigma.$$  
	Then, the evolution problem 	
	\begin{equation}  
		\left\{ 
		\begin{array}{ll} 
			\displaystyle \frac{\partial  \rho_1 }{\partial t}  + \nabla\cdot (   \rho_1\: V )=  0 \quad & \hbox{ in } Q\\  \\    
			\rho_1(0)= \rho_{01} & \hbox{ in }\Omega, 
		\end{array} 
		\right.
	\end{equation}     
	has a unique solution $\rho_1,$ which satisfies  :  
	\begin{itemize} 
		\item  $ \rho_1\in L^\infty(Q)$ 
		\item $ 0\leq \rho_1\leq 1,$ a.e. in $Q$ 
		\item $0\leq ( \rho_1\: V)\cdot \nu \in L^\infty(\Sigma),$  
		\item for any $\xi \in \C^\infty (\overline \Omega),$   
		\begin{equation} 
			\frac{d}{dt} 	\int_\Omega    \rho_1\: \xi \: dx -\int_\Omega  \rho_1\: V\cdot \xi \: dx =  - \int_{\partial\Omega } ( \rho_1\: V)\cdot \nu  \: \xi \: d \mathcal H^{n-1}  , \hbox{ in } \D([0,T)).
		\end{equation}	  
		\item for a.e. $t\in (0,T),$ we have 
		$$\int_\Omega \rho_1(t,x)\: dx =\int_\Omega \rho_{01}(x)\: dx- \int_0^t\!\! \int_{\partial\Omega } ( \rho_1\: V)\cdot \nu  \:  d \mathcal H^{n-1} .$$
	\end{itemize}

	\bigskip\noindent 
	\underline{\textbf{Mixed inflow-outflow population   in  input-output    area:}}  In this case,  we assume that 
	there exists an non-negligible subset $\Sigma^- \subset (0,T)\times \partial\Omega ,$ such that 	for a.e. $t\in (0,T),$ 		$$V\cdot \nu < 0,\quad  \hbox{ on  } \Sigma^+ .$$  
	The evolution problem 	
	\begin{equation}  
		\left\{ 
		\begin{array}{ll} 
			\displaystyle \frac{\partial  \rho_1 }{\partial t}  + \nabla\cdot (   \rho_1\: V )=  0 \quad & \hbox{ in } Q\\   \\    \rho_1 =0 & \hbox{ on } \Sigma^-  \\  \\
			\rho_1(0)= \rho_{01} & \hbox{ in }\Omega, 
		\end{array} 
		\right.
	\end{equation}     
	has a unique solution $\rho_1,$ which satisfies  :  
	\begin{itemize} 
		\item  $ \rho_1\in L^\infty(Q)$ 
		\item $ 0\leq \rho_1\leq 1,$ a.e. in $Q$ 
		\item $0\leq ( \rho_1\: V)\cdot \nu \in L^\infty(\Sigma),$  
		\item  $\rho_1\: V \cdot \nu =0$  $ \mathcal H^{n-1}-$a.e. on $\Sigma^-$  
		\item for any $\xi \in \C^\infty (\overline \Omega),$   
		\begin{equation} 
			\frac{d}{dt} 	\int_\Omega    \rho_1\: \xi \: dx -\int_\Omega  \rho_1\: V\cdot \xi \: dx = - \int_{\partial\Omega } ( \rho_1\: V)\cdot \nu  \: \xi \: d \mathcal H^{n-1}  , \hbox{ in } \D([0,T)).
		\end{equation}	  
		\item for a.e. $t\in (0,T),$ we have 
		$$\int_\Omega \rho_1(t,x)\: dx   =\int_\Omega \rho_{01}(x)\: dx- \int_0^t\!\! \int_{\partial \Omega } ( \rho_1\: V)\cdot \nu  \:  d \mathcal H^{n-1} .$$
	\end{itemize}

	\begin{remark}\label{RemTransport}
		Despite restricting our theoretical analysis to the assumptions of Theorem 1, the algorithm and numerical computations remain effective in several broader scenarios:  
		
		\begin{enumerate}
			\item Compressive vector fields with small initial data: our approach can be extended to compressive vector fields $\nabla \cdot V < 0$ coupled with small initial data $\rho_{01}.$ 
			
			\item  Non-local potential with Gaussian kernel: The velocity field  $V$ can be determined by a non-local potential, $V=-\nabla (\rho_1\star K_\sigma),$ where $K_\sigma$ 	is a Gaussian kernel.   This captures the influence of neighboring individuals on population movement and ensures a repulsive effect, preventing overcrowding. While primarily capturing repulsive behaviors, this framework can be adapted to study aggregative phenomena.
			
			\item Diffusion-based potential: Alternatively, we can employ the Green's function of the Laplacian as the kernel, effectively solving a diffusion equation at each step to determine the potential for $V.$ This diffusion-based approach provides a more nuanced representation of the influence of neighboring individuals on population movement. 
			
		\end{enumerate}

	\end{remark}
	
	\section{Congestion management  theory } 
	\setcounter{equation}{0} 
	\subsection{Developing a congestion management model}
	As discussed in the introduction, spatial constraints can lead to congestion when the directed movement of the first population interacts with the second population of density $\rho_2.$ Specifically, the total density, $\rho_1+\rho_2,$  may exceed the maximum permissible density, $\rho_{max}=1.$ While population  $\rho_1$  maintains its singular focus on its objective, following the prescribed velocity field $V,$ and disregards any global congestion, we introduce a 'decongestion' mechanism   for the initially static population $\rho_2$  to prevent congestion arising from their interaction with $\rho_1.$ We propose a dynamic model for this scenario, inspired by the behavior of grains of sand settling into a stable heap. The first population can be likened to a continuous influx of sand that exerts pressure on the existing structure, while the second population represents the existing sand grains that are effectively rearranged to accommodate these new additions. 
	
	To this aim, we consider the following evolution problem:
	\begin{equation}\label{Evol2}
		\frac{d\rho_2 }{dt}  +\partial {I\!\! I}_{\Lipu}(p)\ni 0, \quad \hbox{ in }(0,T)
	\end{equation}
	where $p$ is the Lagrange multiplier associated with the constraint 
	$\rho_1+\rho_2\leq 1$ ; i.e. 
	\begin{equation}\label{Lagrangerho2}
		p\geq 0, \: 0\leq \rho_2\leq 1-\rho_1,\: p(1-(\rho_1+\rho_2)) =0 	 ,\quad   \hbox{ in } Q. 
	\end{equation}
	Here $\partial {I\!\! I}_{\Lipu}(p)$ denotes the sub-differential of the indicator function of $\Lipu$, the set of 1-Lipschitz functions on $\Omega$, defined as:
	$$  {I\!\! I}_{\Lipu}(p)=\left\{ 
	\begin{array}{ll}
		0 & \hbox{ if }p\in \Lipu\\ 
		+\infty & \hbox{ otherwise}. 
	\end{array}
	\right. $$
	The precise definition of  $\partial {I\!\! I}_{\Lipu}$ using the duality pairing will be provided  in   Section \ref{section:Theoretical study} (see \eqref{partialLip1}).  Remember that (cf. \cite{AgCaIg,DJ,IgEquiv,IgEvol}), the evolution problem \eqref{Evol2}    can be reformulated in terms of PDE as:
	\begin{equation}\label{PDEcorrection}
		\left\{   
		\begin{array}{ll}
			\displaystyle \frac{\partial \rho_2 }{\partial t}  - \nabla\cdot (  m\: \nabla p ) =  0\\  \\ 
			m\geq 0,\: 	 	\vert \nabla p\vert  \leq 1,\: m(1-\vert \nabla p\vert)=0, 
		\end{array}
		\right\} \quad \hbox{ in }Q,
	\end{equation}
	subject to \eqref{Lagrangerho2} and Neumann boundary condition:
	\begin{equation}\label{BCrho2}
		m\: \nabla p  \cdot \nu =0 \quad \hbox{ on }\Sigma.
	\end{equation}
	Indeed, this formulation is closely connected to the observation that, employing a test function of the form $p-\xi$, where $  \xi \in \Lipu$, in equation \eqref{PDEcorrection}, we formally obtain:
	\begin{equation} \label{weakrho2}
		\int_{\Omega} \partial_t \rho_2 \: (p-\xi) = -	\int_{\Omega} m \nabla p \cdot \nabla (p-\xi) = \int_{\Omega} m (1 - \nabla p \cdot \nabla \xi) \leq 0,
	\end{equation} 
	which is equivalent to 
	\begin{equation} 
		-\partial_t \rho_2 \in \partial {I\!\! I}_{\Lipu}(p). 
	\end{equation} 
	However, as is typical for this equivalence, $m$ is generally a Radon measure. This requires the use of the tangential gradient for the gradient of $p$, a specialized approach (see ~\cite{Bouchitte&al} for further details). To circumvent this complexity, we focus on the variational formulation \eqref{Evol2}, drawing inspiration from studies such as ~\cite{AgCaIg,DJ,IgEquiv,IgEvol} that address similar problems in slightly different contexts. 
	
	The function $\rho_1$ in \eqref{PDEcorrection} and \eqref{Evol2} is given by the solution of the transport equation from the previous section. Specifically, it is assumed that $\rho_1\in L^\infty(Q)$, $0\leq \rho_1\leq 1$ a.e. in $Q$, and satisfies:
	\begin{equation} \label{weakrho1}
		\frac{d}{dt} 	\int_\Omega    \rho_1\: \xi \: dx = \int_\Omega  \rho_1\: V\cdot \xi \: dx    , \hbox{ in } \D([0,T)).
	\end{equation}	   for any $\xi \in \C^\infty_c (\Omega).$   Here $V\in L^\infty(Q)^N \cap L^1(0,T,W^{1,1}(\Omega)^N)$ and  $0\leq \nabla \cdot V \in L^\infty(Q)$. 
	\medskip 
	
	As previously mentioned, the variational approach will play a crucial role in our characterization of the solutions to \eqref{PDEcorrection}. 
	However, before proceeding further with the formulation, it's important to note that the   key term $	\int_{\Omega} \partial_t \rho_2 \: (p-\xi) $ in the  variational formulation associated with   \eqref{Evol2}  necessitates increased regularity assumptions on either $\rho_1$, $\rho_2$ or $p,$ which are not generally expected to hold.   Indeed, 
	using \eqref{Lagrangerho2}, we formally see  that     
	\begin{equation}\label{formal10}
		\begin{array}{lll}	 \int_\Omega \partial_t \rho_2\: (p-\xi) &=&  \int \partial_t (1-\rho_1)  \: p - \int_\Omega \partial_t \rho_2 \: \xi \\  \\  &=& -  \int  \partial_t \rho_1  \: p - \int_\Omega \partial_t \rho_2\: \xi  \\  \\  
			&=&  - \int  \partial_t  \rho_1  \: p -      \frac{d}{dt}\int_\Omega \rho_2 \: \xi   . 
	\end{array} \end{equation}  
	
	Contrary to the findings reported in the aforementioned studies ( \cite{IgUr,AgCaIg}) where $\rho_1$ was essentially zero, the non-zero nature of $\rho_1$ in this study renders the application of the same approach employed in ~\cite{IgUr,AgCaIg} ineffective.

	\bigskip 
	To circumvent this issue, we propose to incorporate the transport equation  \eqref{PDEtransport0} into  \eqref{Evol2}, achieving   the nonlinear dynamic for the total density $\rho:=\rho_1+\rho_2$ :  
	\begin{equation}
		\frac{d}{dt} \rho +\partial {I\!\! I}_{\Lipu}(p)   \ni f
	\end{equation}
	where   $f:=-\nabla \cdot(\rho_1\: V).$  Since  $\rho_1$ is a solution of the transport equation \eqref{PDEtransport0}, $f=-\nabla \cdot(\rho_1\: V),$ should to be understood in distributional sense :  
	\begin{equation}\label{def-f}
		\begin{array}{c} 
			\langle f(t),\xi  \rangle =\int_\Omega \rho_1(t,x)\: V(t,x) \cdot \nabla \xi(x)\: dx    \\ \\  - \int_{\partial\Omega } ( \rho_1(t,x)\: V(t,x))\cdot \nu(x)  \: \xi(x) \: d \mathcal H^{n-1} (x) , \end{array} 
	\end{equation}
	for any $\xi \in \C^\infty(\overline   \Omega),$ which is well defined.  The formulation \eqref{def-f}  can also be interpreted within the framework of the duality bracket between $\Lip ,$ the space of $1-$Lipschitz functions, and $\Lip'$,  its topological dual space, which will be introduced later. 
	
	\medskip 
	
	In the preceding section, the density was characterized by $\rho_1$. In this section, the focus will be on solving the evolution problem \eqref{Evol2} and \eqref{Lagrangerho2}. Following the same line of reasoning as in the formal computation \eqref{formal10}, we observe  that:
	
	\begin{equation}  	 \int_\Omega \partial_t \rho \: (p-\xi) =    -      \frac{d}{dt}\int_\Omega \rho\: \xi   . 
	\end{equation} 
	Combining this observation with the definition of $\partial I\!\! I_{\Lipu},$  we arrive at the following variational formulation for the total density : 
	\begin{equation}
		- \frac{d}{dt}\int_\Omega \rho\: \xi \leq  \langle f,p-\xi\rangle , \quad \hbox{ in }\D'([0,T)).
	\end{equation}
	That is 
	\begin{equation} \label{varformrho2}
		- \frac{d}{dt}\int_\Omega \rho\: \xi - \int_\Omega \rho_1\: V \cdot \nabla ( p-\xi) \: dx  \leq    - \int_{\partial\Omega } ( \rho_1\: V)\cdot \nu  \: (p-\xi) \: d \mathcal H^{n-1}  , \quad \hbox{ in }\D'([0,T)), 
	\end{equation} 
	for any $\xi \in \Lipu.$  
	
	Finlay, the study of the management  of the congested crossing  pedestrian traffic flow  follows by first solving  the transport equation as in the previous section, and then solving the evolution problem \eqref{Evol2} combined with the state equation \eqref{Lagrangerho2} through the variational formulation \eqref{varformrho2}. 
	
	\subsection{Theoretical study}\label{section:Theoretical study}
	
	Thanks to the last section, we now focus on the theoretical study of the evolution problem.
	
	\begin{equation}\label{Evolrho2}
		\left\{  \begin{array}{ll}  
			\displaystyle \frac{\partial \rho }{\partial t}  +  \partial {I\!\! I}_{\Lipu} (p) =  f \\  \\    
			0\leq p ,\:  0\leq \rho\leq 1,\:  p(1-\rho ) =0 , 
			\\  \\  
			\rho (0)=\rho_{01}+\rho_{02}, 
		\end{array}\right. 
	\end{equation}
	where $f=-\nabla \cdot (\rho_1\: V)$ is given by \eqref{def-f} and $\rho_1$   satisfy   \eqref{weakrho1}
	with $\rho_2(0)=\rho_{02}.$

	\bigskip 
	
	Thanks to the discussion above, the notion of solution we have in mind is given as follows :
	\begin{definition}\label{def2}
		A   function $\rho_2$ is said to be a weak solution of \eqref{Evolrho2}, if $0\leq \rho_2\in L^\infty(Q),$  $0\leq \rho:= \rho_1+\rho_2\leq 1 $ a.e. in $Q$ and, there exists $0\leq p\in L^\infty (Q),$ s.t. $p(t)\in\Lip1$, a.e. $t\in (0,T),$  $p(1-\rho)=0,$ a.e. in $Q,$   and 
		\begin{equation}  \label{WeakEvolrho2} 
			- \frac{d}{dt}\int_\Omega (\rho_1+\rho_2) \: \xi \leq \langle f,p-\xi\rangle  , \quad \hbox{ in }\D'([0,T)), 
		\end{equation} 
		for any $\xi \in \Lipu,$ $\rho_2(0)=\rho_{02}.$    
	\end{definition}

	\begin{remark}
		By   definition of $f,$ equation \eqref{WeakEvolrho2} can be rewritten as 
		\begin{equation}  
			- \frac{d}{dt}\int_\Omega (\rho_1+\rho_2) \: \xi \leq  \int_\Omega \rho_1\: V \cdot \nabla ( p-\xi) \: dx       - \int_{\partial\Omega } ( \rho_1\: V)\cdot \nu  \: (p-\xi) \: d \mathcal H^{n-1}  , \quad \hbox{ in }\D'([0,T)), 
		\end{equation} 
		for any $\xi \in \Lipu.$ This, in turn, is equivalent, by virtue of \eqref{weakrho1}, to 
		\begin{equation}  
			- \frac{d}{dt}\int_\Omega   \rho_2  \: \xi \leq  \int_\Omega \rho_1\: V \cdot \nabla  p  \: dx       - \int_{\partial\Omega } ( \rho_1\: V)\cdot \nu  \:  p  \: d \mathcal H^{n-1}  , \quad \hbox{ in }\D'([0,T)), 
		\end{equation}  
		
	\end{remark}

	\bigskip 
	
	We suggest studying   problem \eqref{Evolrho2} using $1$-Wasserstein distance, $W_1,$ and the Kantorovich potential. To do this, we will use the following notations and tools related to $W_1$  (a complete overview of the theory in may be found in the book \cite{Stbook}).
	
	
	\begin{itemize} 
		\item $\Lip:=W^{1, \infty}(\Omega)$ the set of Lipschitz functions on $\Omega$, 
		\item  $\Lipu$ the set of $1$-Lipschitz functions on $\Omega$.  
		\item  $\Lip'_0:=\{h \in \Lip' \; : \; \langle h \rangle :=\langle h, 1 \rangle=0\},$ where  $ \Lip'$ is  the dual (topological) space of $\Lip$ and $\langle.,.\rangle$ denotes the   duality pairing between $\Lip'$ and $\Lip.$ .
		\item  For any $h\in \Lip'_0$,  we denote by $W_1$ the dual semi-norm of $h$:
		$$W_1(h):=\sup_{\theta\in \Lipu}  \langle h, \theta \rangle.$$
		When $h$ is signed measure of the form $h=h^+-h^-$, with $h^{\pm}$ probability measures on $\Omb$, $W_1(h)$ is the $1$-Wasserstein  distance between $h^+$ and $h^{-}$ (see ~\cite{Stbook}).  
		
		\item For a general $h\in \Lip'_0$, $W_1(h)$ may be infinite. To ensure $W_1(h)<\infty,$  one needs to  restrict $h$ to a space $X'$, where $\Lip\subset X$ is such that the embedding of $\Lip$ into $X$     is compact.    A function $\theta\in \Lip_1$ for which $W_1(h)= \langle h, \theta\rangle$ is called a Kantorovich potential of $h$. We restrict our attention to non-negative potentials, as it is always possible to find at least one such potential whenever the existence of a potential is guaranteed. We denote by $\K(h)$ the set of all Kantorovich potentials of $h$, that is: 
		\[\K(h):=\{0\leq \theta \in \Lipu \; : \; \langle h, \theta\rangle\ge \langle h, \psi\rangle, \; \forall \psi \in \Lipu\}.\]
		In particular, this enables to define $ \partial {I\!\! I}_{\Lipu} (p) $ in terms of Kantorovich potential  as follows :  \begin{equation}\label{partialLip1}
			h\in \partial {I\!\! I}_{\Lipu} (p) \quad  \Longleftrightarrow \quad  p\in \Lipu,\: h\in \Lip_0'\hbox{ and  }p\in \K(h).  
		\end{equation}

		\item    Fenchel-Rockafellar duality theorem gives the following dual formula for $W_1(h)$:
		\begin{equation}\label{dualw1}
			W_1(h)=\inf_{\sigma \in(L^{\infty}(\Omega)')^N} \Big\{ \Vert  \sigma\Vert_{(L^{\infty}(\Omega)')^N }     \; : \;  \langle  \sigma , \nabla \xi \rangle\: dx =\langle h,\xi\rangle, \forall\: \xi \in \Lip \Big\}.
		\end{equation} 
		Moreover, one can prove that for any $g\in \Lip'_0$,  \eqref{dualw1} admits solutions (in $(L^{\infty}(\Omega)')^N$ and  not in $L^1$ in general), such solutions are called optimal flows.
		A Kantorovich potential $\theta\in \K(h)$ is related to an optimal flow $\sigma$ in \eqref{dualw1} by the extremality relation 
		\[\langle \sigma, \nabla \theta\rangle_{(L^{\infty}(\Omega)')^N, L^{\infty}(\Omega)^N} =  \Vert \sigma\Vert_{(L^{\infty}(\Omega)')^N} \mbox{ with } \Vert \nabla \theta \Vert_{L^{\infty}} \le 1\]
		which, very informally, means that $\sigma$ is concentrated on the set where $\vert \nabla \theta\vert$ equals $1$ and is collinear to $\nabla \theta$.  
		If $\sigma\in L^1,$ then previous relation reveals that $\sigma =m\:  \nabla \theta$ with $m\ge 0$ and satisfies the complementary   condition $m\: (1-\vert \nabla \theta\vert)=0$.   This finding formally supports the formulation of equation \eqref{Evolrho2} based on PDE \eqref{PDEcorrection}. Furthermore, the flux $\sigma$ exhibits a tangential alignment with the boundary $\partial \Omega$ in a weak sense.
		
		\item  For an arbitrary $h\in \Lip'$ (not necessarily balanced), we can define  
		\[\wW(h):=W_1(\underbrace{h-\frac{1}{\vert \Omega\vert }\langle h \rangle}_{=:\hat  h})+   \frac{1}{\vert \Omega\vert } \vert \langle h \rangle \vert\]
		and observe that $\wW$ is equivalent to the usual norm of $\Lip'$.
		When there is non confusing, we'll  use again the notation $W_1(h)$  instead of $\wW(h)$ even if $h\in \Lip'\setminus \Lip'_0.$
		
	\end{itemize}

	Given that $f(t)$, defined by \eqref{def-f} for a.e. $t\in (0,T)$, belongs to $\Lip'$,  for any $z\in \Lipu,$    we have 
	\begin{equation}  \label{estfz}
		\begin{array}{ll}  	\langle f(t),z\rangle  &= 	\langle\hat  f(t),z \rangle 
			+ 	\langle f(t)\rangle \: \oint_\Omega  z\\ \\ 
			&\leq 	W_1( \hat   f(t)  )
			+ 	\langle f(t)\rangle \: \oint_\Omega  z . \end{array} \end{equation}   
	Furthermore, both   $	\langle f(t)\rangle$ and $W_1( \hat   f(t)  )$ are well defined and belong to the space  $L^\infty(0,T)$. Indeed,  for a.e. $t\in (0,T),$ we have 
	\begin{equation} 
		\langle f(t)\rangle = -\int_{\partial\Omega } ( \rho_1\: V)\cdot \nu   \: d \mathcal H^{n-1}, \quad \hbox{ for a.e. }t\in (0,T), \end{equation}  
	and, moreover we see that 
	\begin{eqnarray*}
		W_1(\widehat    f(t)) &=& \max_{\theta\in \Lipu}  \langle f(t) , \theta -\oint_\Omega \theta \rangle   \\  \\  
		&=& \max_{\theta\in \Lipu} \left\{  \int_\Omega  \rho_1(t)\: V(t)\cdot \nabla \theta \: dx   - \int_{\partial\Omega } ( \rho_1\: V)\cdot \nu  \: (\theta- \oint_\Omega  \theta  ) \: d \mathcal H^{n-1}  \right\}   \\  \\ 
		&\leq&  C(\Omega,N)\: \left( \Vert V(t)\Vert_1   +  \Vert \vert  (\rho_1\: V)\cdot \nu  \Vert_{L^1(\partial \Omega} \right)   . 
	\end{eqnarray*}  
	
	\medskip 
	Now, considering the previous observations, it is natural to interpret the evolution problem  \eqref{Evolrho2}  
	as  the inclusion
	\begin{equation}\label{wf0}
		p(t) \in \K(f(t)-\partial_t \rho(t)), \quad    \: 0\leq \rho(t)\leq 1,\:  0\leq p(t), \: p(t)(1-\rho(t))=0, 
	\end{equation}
	for any $t\in (0,T).$
	In particular, one sees that this formulation inherently imposes a   conservation of mass, represented  by $$\langle f(t)-\partial_t \rho(t) \rangle=0, \quad \hbox{ for any }t\in [0,T),$$   which aligns   with the boundary condition in \eqref{BCrho2} and the weak formulation \eqref{weakrho1}.  Moreover, thanks to the definition of $f$, \eqref{def-f}  implies that 
	\begin{equation}\label{massf}
		\frac{d}{dt}\int_\Omega \rho(t)\: dx = 	\langle f(t)  \rangle = - \int_{\partial\Omega } ( \rho_1(t)\: V(t))\cdot \nu   \: d \mathcal H^{n-1} ,  \quad \hbox{ for a.e. }t\in (0,T), 
	\end{equation}
	which  implies in turns  
	\begin{equation}\label{massfudef}
		\M(t):=   \int_\Omega \rho(t)\: dx =  \int_\Omega \rho_0\: dx  +\int_0^t 	\langle f(t)  \rangle  \: dt. 
	\end{equation}
	Since $0\leq \rho\leq 1$ a.e. in $Q,$ this property implies a necessary condition for the existence of a solution to the problem \eqref{Evolrho2} :  
	\begin{equation}\label{sufcond0}
		0\leq \M(t) \leq \vert \Omega\vert  ,  \quad \hbox{ for a.e. }t\in (0,T) . 
	\end{equation}

	With all this in mind, we can now state the main result of this section. 
	
	\begin{theorem}\label{thrho2}  Let $0\leq \rho_{02}\in L^\infty(\Omega),$ be such that $0\leq \rho_{01}+\rho_{02}\leq 1$ a.e. in $\Omega$, and satisfying moreover 
		\begin{equation}\label{sufcond}
			\sup_{t\in [0,T)}\M(t) =: \overline \M  < \vert \Omega\vert   ,  \quad \hbox{ for a.e. }t\in (0,T). 
		\end{equation}
		Then, the problem \eqref{Evolrho2} has a weak solution $\rho_2$ in the sense of Definition \ref{def2}. Moreover, we have 
		\begin{equation}\label{massconserv}
			\int_{\Omega} (\rho_1(t)+\rho_2(t))\: dx =\M(t),\quad \hbox{ a.e. }t\in [0,T),
		\end{equation}
		and \begin{equation}\label{massconserv2}
			\frac{d}{dt}\int_\Omega \rho_2(t)\: dx =0.
		\end{equation}
		
	\end{theorem}
	
	\medskip 
	\begin{remark}
		\begin{enumerate}
			\item Since $0\leq \rho_{02} \leq 1,$  	instead of working with   condition \eqref{sufcond} in all the time interval $(0,T),$   one can replace  $T,$ by $T_a$ given by 
			\begin{equation} 	\label{Tadm}
				T_a:=  \sup\left\{ t\in [0,T)\: :\:  0\leq \M(s)  < \vert \Omega\vert   ,  \quad \hbox{ for a.e. }s\in (0,t)   \right\}.  
			\end{equation}
			Yet, one needs to assume that $T_a>0.$ 
			
			\item By definition fo $f,$ one sees that  
			\begin{equation}\label{massfu}
				\begin{array}{ll}
					\M(t)  &=  \int_\Omega \rho_0\: dx  -\int_0^t\!\!  \int_{\partial\Omega } ( \rho_1(t)\: V(t))\cdot \nu  \: d \mathcal H^{n-1} ,  \quad \hbox{ for a.e. }t\in (0,T) \\  \\ 
					&=  \int_\Omega \rho_0\: dx  -\int_\Omega \rho_{01}\: dx + \int_\Omega \rho_{1}(t)\: dx    ,  \quad \hbox{ for a.e. }t\in (0,T) \\  \\ 
					&=  \int_\Omega \rho_{02}\: dx   + \int_\Omega \rho_{1}(t)\: dx    ,  \quad \hbox{ for a.e. }t\in (0,T)\\  \\ 
					&=  \int_\Omega (\rho_{1}(t)+\rho_{2}(t))\: dx    ,  \quad \hbox{ for a.e. }t\in (0,T).  
				\end{array} 
			\end{equation} 
			
			Condition \eqref{sufcond}   ensures that the combined density of both populations $\rho_1$ and $\rho_2$  remains strictly below $1$ at any given time. This means there is always a free space within the domain, allowing for the population  to undergo its correction dynamics.  
		\end{enumerate}
	\end{remark}

	With the  equivalence \eqref{wf0} in consideration, we propose solving the evolution problem \eqref{Evolrho2} using an Euler implicit scheme. For a given time-step $\tau>0,$  the implicit time discretization associated with \eqref{wf0} is given by:
	\begin{equation} \label{wf1disc}
		p_{k+1}^\tau \in \K \Big( f_k^\tau  -\frac{\rho_{k+1}+\rho_{k}}{\tau} \Big),  \quad    0\leq \rho_{k+1}^\tau\leq 1,\: 0\leq p_{k+1}^\tau,\: p_{k+1}^\tau(1-\rho_{k+1})=0, 
	\end{equation} 
	for $k=0,... m-1,$ where $m:=[T/\tau]$ and  $f_k^\tau$ is given by 
	\begin{equation}
		f_k^\tau = \frac{1}{\tau}\int_{k \tau}^{(k+1) \tau} f(t)\: dt, \hbox{ for any }k=0,..n-1. 
	\end{equation}
	Here, we know the value of  $\rho_{k}^\tau$, but we need to find the values of $p_{k+1}^\tau$ and $\rho_{k+1}^\tau$. Thanks to the definition of $\K,$, we also have
	\begin{equation}\label{wf0disc}
		p_{k+1} ^\tau\in \K \Big(\rho_{k}^\tau +\tau f_k^\tau-\rho_{k+1} \Big), \quad    0\leq \rho_{k+1}^\tau\leq 1,\: 0\leq p_{k+1}^\tau , \: p_{k+1}^\tau(1-\rho_{k+1}^\tau)=0 . 
	\end{equation} 
	That is $p_{k+1} ^\tau \in \Lipu$ and 
	\begin{equation}\label{eqpk+1}
		\langle \rho_{k}^\tau +\tau f_k^\tau-\rho_{k+1} , p_{k+1}^\tau -\psi \rangle\ge 0,  \; \forall \psi \in \Lipu ,  
	\end{equation}
	with  $$  0\leq \rho_{k+1}^\tau\leq 1,\: 0\leq p_{k+1}^\tau,  \:   p_{k+1}^\tau(1-\rho_{k+1}^\tau)=0.$$

	The following lemma provides a means to determine $\rho_{k+1}$ and $p_{k+1}$, and  to implement the Euler implicit scheme \eqref{wf1disc}.

	\begin{lemma}\label{Lemduality}
		Given a time step $\tau>0$, for any $k=1,... m,$ the sequences $\rho_{k}^{\tau}$ and  $p_k^{\tau}$ in \eqref{wf0disc} are well defined   inductively by 
		\begin{equation}\label{jkoscheme}
			\begin{array}{c}\rho_{k}^\tau \in \argmin_{u } \Big\{W_1(\rho_{k-1}^\tau +\tau f_{k-1}^\tau-u)  \; : \; \langle \rho_{k-1}^\tau +\tau f_{k-1}^\tau-u \rangle=0,\: \\ \\ u\in L^\infty(\Omega),\: 0\leq u\leq 1  \Big\},
		\end{array} 		\end{equation} 
		and 
		\begin{equation}\label{jkoschemep}
			p_{k}^\tau \in  \argmax_{\theta}    \Big\{   \langle \rho_{k-1}^\tau +\tau f_{k-1}^\tau, \theta \rangle  -\int\theta   \: :\: 0\leq \theta\in \Lipu \Big\} 
		\end{equation}  
		by setting $\rho_{0}^\tau=\rho_{01}.$ 
		Moreover, we have 
		\begin{equation}
			W_1(\rho_{k-1}^\tau +\tau f_{k-1}^\tau-\rho_{k}^\tau ) =  \langle \rho_{k-1}^\tau +\tau f_{k-1}^\tau, 	p_{k}^\tau \rangle  -\int	p_{k}^{\tau}  
		\end{equation}
	\end{lemma}
	\begin{proof}
		Existence for  \eqref{jkoscheme} follows from the fact that $W_1$ is lsc for the weak $L^p$-topology, for any $1\leq p\leq \infty.$  As to \eqref{jkoschemep}, it follows by duality.  Indeed, using   the definition of $f,$ we have   
		\begin{equation}
			W_1(\rho_{k-1}^\tau +\tau f_{k-1}^\tau-u) =  \max_{0\leq \theta\in \Lipu}   \langle \rho_{k-1}^\tau +\tau f_{k-1}^\tau-u, \theta \rangle, 
		\end{equation}
		so that 
		\begin{equation}
			\begin{array}{c} 
				I_{k-1} := 	\min_{u \in  L^{\infty}(\Omega)} \Big\{W_1(\rho_{k-1}^\tau +\tau f_{k-1}^\tau-u)  \; : \; \langle \rho_{k-1}^\tau +\tau f_{k-1}^\tau-u \rangle=0,\: 0\leq u\leq 1  \Big\} \\  \\ 
				=  \min_{u \in  L^{\infty}(\Omega)}  \max_{0\leq \theta\in \Lipu} \Big\{   \langle \rho_{k-1}^\tau +\tau f_{k-1}^\tau-u, \theta \rangle  \; : \; \langle \rho_{k-1}^\tau +\tau f_{k-1}^\tau-u \rangle=0,\: 0\leq u\leq 1  \Big\} . 
		\end{array}	\end{equation} 
		
		We can use the Von Neumann-Fan minimax theorem (see for instance ~\cite{Borwein&Zhuang}), to prove this, just like we did for the proof of Proposition 4.5 in the work \cite{EIJ}, we get
		\begin{equation}
			\begin{array}{c}  	I_{k-1} = \max_{0\leq\theta\in \Lipu}    \min_{u \in  L^{\infty}(\Omega)}  \Big\{   \langle \rho_{k-1}^\tau +\tau f_{k-1}^\tau-u, \theta \rangle  \; : \; \langle \rho_{k-1}^\tau +\tau f_{k-1}^\tau-u \rangle=0,\: 0\leq u\leq 1  \Big\} \\  \\ 
				=  \max_{0\leq \theta\in \Lipu}    \Big\{   \langle \rho_{k-1}^\tau +\tau f_{k-1}^\tau, \theta \rangle  -\int\theta   \Big\}. 
		\end{array}	\end{equation} 
		Moreover,   taking $\rho_{k}$ and $p_{k}$ given by \eqref{jkoscheme} and \eqref{jkoschemep},  respectively, one sees the couple $(\rho_{k},p_{k})$ solves the equation \eqref{wf0disc}. 
	\end{proof}

	%
	Clearly, the algorithm \eqref{wf0disc}, for $ k=0,...m-1$, follows the implicit Euler scheme à la Jordan-Kinderlehrer-Otto (cf.  \cite{JKO})   to construct weak solutions, but using $W_1$ instead of the more familiar $ 2-$ Wasserstein distance, $W_2,$ and incorporating the source explicitly in the scheme with $1-$Wasserstein.

	\medskip 
	
	To prove the convergence of the implicit Euler scheme with the solution given in Theorem \ref{thrho2}, we first establish some preliminary results. 
	We define two curves corresponding to piecewise constant and linear interpolation:
	\begin{equation}\label{interpol}
		\rho^\tau (t):=\rho_{k+1}^\tau \quad	  \hbox{ and }\quad  \tilde\rho^\tau(t):=\rho_{k}^{\tau}+ \frac{t-k\tau}{\tau} (\rho_{k+1}^\tau-\rho_{k}^\tau), \quad   \hbox{ for any }t\in (k\tau, (k+1)\tau], 
	\end{equation}
	for  $k=0, \cdots, m$. We also define the piecewise constant approximation of  $f$:
	\begin{equation}\label{interpolf}
		f^\tau(t):=f_{k}^\tau, \; t\in (k\tau, (k+1)\tau], \; k=0, ...  m,  
	\end{equation}
	and the piecewise constant approximation   \begin{equation}\label{interpolp}
		p^\tau(t):= p_{k+1}^\tau, \; t\in (k\tau, (k+1)\tau].
	\end{equation} 
	Thanks to \eqref{sufcond}, we see that   
	\begin{equation}\label{mass0} 
		\langle \rho_{k}^\tau \rangle  =\M(k\: \tau) \leq \overline \M,\quad \hbox{ for any }k=0,...m.  
	\end{equation}
	Indeed,  by construction 
	\begin{equation} \label{massk}
		\begin{array}{ll }	\langle \rho_{k}^\tau \rangle &=  \langle \rho_{k-1}^\tau +\tau\: f_{k-1}^\tau \rangle \\    \\  
			& = \int_\Omega \rho_0(x)\: dx  +  \tau\: \sum_{i=0}^{k }  \langle f_i^\tau  \rangle   
			\\   \\  
			& =\int_\Omega \rho_0(x)\: dx  + \int_0^{k\: \tau }  \langle f(t)  \rangle\: dt  =\M(k\: \tau).   
		\end{array} 
	\end{equation}

	The goal now is to get to the limit in \eqref{wf0disc}, as $\tau \to 0,$ this is the main objective of the sequence of the following lemmas. 
	
	\begin{lemma}\label{lptaubounded}
		$p^\tau$ is bounded in $L^\infty(0,T; \Lipu).$
	\end{lemma}
	\begin{proof}	By definition of $p^\tau,$ we know that 
		\begin{equation}\label{gradpestimate}
			\Vert  \nabla p^\tau \Vert_{L^{\infty}((0,T)\times \Omega)} \le 1,\quad \hbox{ a.e. in }Q.
		\end{equation} This implies that, for any $t\in [0,T),$   $p^\tau(t) -\oint_\Omega  p^\tau(t) $  is bounded in $\Lipu.$ On the other hand,  using  \eqref{eqpk+1}, we have  
		$$\langle \rho_{k+1}^\tau - \rho_{k}^\tau ,p_{k+1}^\tau \rangle \leq \tau\:  \langle f_k^\tau ,p_{k+1}^\tau \rangle  .$$
		Combining this with \eqref{estfz}, we get  
		$$  \int_\Omega (\rho_{k+1}^\tau-\rho_{k}^\tau)\: p_{k+1}^\tau   \leq \tau\:  W_1(  \widehat   { f_k^\tau})  + \tau\: 	\langle f_k^\tau\rangle \: \oint_\Omega  p_{k+1}^\tau.  $$
		Using the fact that  $\rho_{k+1}^\tau \: p_{k+1}^\tau = p_{k+1}^{\tau } ,$ we get  
		$$ \langle p_{k+1}^\tau \rangle \leq    \tau\:   W_1(  \widehat   { f_k^\tau})  + \tau\: 	\langle f_k^\tau\rangle  \: \oint_\Omega  p_{k+1}^\tau   + \int_\Omega  \rho_{k}^\tau\: p_{k+1}^\tau .    $$ 
		This implies that 
		$$ \langle p_{k+1}^\tau \rangle   \leq    \tau\:   W_1(  \widehat   { f_k^\tau}) + \underbrace{( \langle  \rho_{k}^\tau \rangle +  \tau\: 	\langle f_k^\tau\rangle)}_{\M((k+1)\tau)} \: \oint_\Omega  p_{k+1}^\tau  +  \int_\Omega  \rho_{k}^\tau\:( p_{k+1}^\tau - \oint_\Omega  p_{k+1}^\tau     )  ,$$
		and then 
		$$ \oint_\Omega  p_{k+1}^\tau   \:  (\vert \Omega\vert - \M((k+1)\tau)) \leq    \tau\:   W_1(  \widehat   {  f_k^\tau} )  +  \int_\Omega  \rho_{k}^\tau\:( p_{k+1}^\tau - \oint_\Omega  p_{k+1}^\tau     )  . $$
		Remember that   $ \M((k+1)\tau)  \leq \overline \M<\vert \Omega\vert .$   So
		$$  \oint_\Omega  p_{k+1}^\tau   \:    \leq    \frac{ \tau\:  W_1(  \widehat   {  f_k^\tau} ) +  \int_\Omega \vert  p_{k+1}^\tau -   \oint_\Omega p_{k+1}^\tau     \vert  }{ \vert \Omega\vert -  \overline \M } ,$$ 
		and then 
		\begin{equation}\label{averagestt}
			\oint_\Omega   p^\tau(t)     \:    \leq    \frac{ \tau\:   W_1(  \widehat   {  f^\tau} (t)) +  \int_\Omega \vert  p^\tau(t) -  \oint_\Omega    p^\tau(t)     \vert  }{ \vert \Omega\vert -  \overline \M  } ,\quad \hbox{ for any }t\in (k\tau, (k+1)\tau].  
		\end{equation}
		Using moreover, the fact that $p^\tau(t)- \oint_\Omega  p^\tau(t) \: dx$ and $ W_1(  \widehat   {f^\tau}(t))$ are bounded in $L^\infty(0,T),$ we deduce the  result of the lemma.    
	\end{proof}
	
	\begin{lemma}
		There exist a vanishing sequence of time steps, $\tau_{n}\to 0$ as $n\to \infty$,  $0\leq \rho\in L^\infty(Q)$ and $0\leq p\in L^\infty(Q) ,$  such that  $p (t)\in \Lipu$    for a.e. $t\in (0,T)$ and  
		setting $ \rho_{n }:=  \rho^{\tau_n}$, $\tilde\rho_{n}:=\tilde\rho^{\tau_n}$  and $p_n:=p^{\tau_n},$  we have :
		\begin{equation}
			\rho_{n} \to \rho, \quad \mbox{ in  }L^\infty(Q)-\hbox{weak}^*  \label{cvgce1}, 
		\end{equation}
		\begin{equation} \tilde \rho_{n}  \to \rho, \quad  \mbox{ in  }L^\infty(Q)-\hbox{weak}^*  \label{cvgce2},\end{equation} 
		and 	\begin{equation} p_n \to p, \quad  \mbox{ in }L^q(0,T;W^{1,q}(\Omega)-\hbox{weak}, \: \forall q\in (1, +\infty).  \label{cvgce3} 
		\end{equation}
		Moreover, we have $ 0\leq \rho\leq 1$ and $p(1-\rho )=0$ a.e. in $Q$.
	\end{lemma}
	\begin{proof} Clearly \eqref{cvgce1} follows by Lemma \ref{lptaubounded}. We  give here the main estimates concerning  $ \rho^\tau$ and $\tilde\rho^\tau.$ 
		Using $\rho_{k}^\tau  +\frac{\tau}{\vert \Omega\vert}\langle f_k^\tau \rangle$ as a competitor to $\rho_{k+1}^\tau$ in \eqref{jkoscheme}, we see first :  
		\begin{equation}\label {neweq1}
			\begin{split}
				W_1(\rho_{k+1}^\tau-(\rho_{k}^\tau+ \tau f_k^\tau))  &\le  	W_1(\tau\: \widehat   {f_k^\tau}  )  \\  \\  &\leq   \int_{k\tau}^{(k+1)\tau} W_1(\widehat   { f}(t) )  \mbox{d}t 
			\end{split}
		\end{equation}
		Hence,  
		\begin{equation}\label{estim01}
			\sum_{k=0}^{m-1} W_1(\rho_{k+1}^\tau-\rho_{k}^\tau -\tau f_k^\tau) \le \int_{0}^{T}  W_1(\widehat   { f}(t) )  \mbox{d}t .
		\end{equation}
		Next, by definition of $\tilde \rho^\tau,$  we observe that,    
		\begin{eqnarray*}
			\int_{k\tau}^{(k+1)\tau} W_1(\widehat{\partial_t \tilde\rho^\tau(t)}) &=& \tau\: 	 W_1(\widehat{\partial_t \tilde\rho^\tau(t)}) \\  \\  &=&   W_1(\widehat{\rho_{k+1}^\tau-\rho_{k}^\tau}) \\  \\ 
			&\leq& W_1 (\rho_{k+1}^\tau-\rho_{k}^\tau  - \tau f_k^\tau)+\tau   W_1(\widehat   {f_k^\tau}), 
		\end{eqnarray*}
		so that 
		\begin{equation}
			\int_0^T W_1(\widehat{\partial_t \tilde\rho^\tau(t)}) \: dt  \leq \sum_{k=0}^{m-1}  W_1 (\rho_{k+1}^\tau-\rho_{k}^\tau  -\tau f_k^\tau)     + \int_0^T   W_1(\widehat {f(t)})\: dt. 
		\end{equation}
		Together with \eqref{estim01} and the fact that $\langle\partial_t \tilde\rho^\tau(t) \rangle =\langle  f^\tau(t) \rangle   ,$  for any $t\in [0,T),$  we deduce
		\begin{equation}\label{estim2}
			\Vert \partial_t \tilde\rho^\tau\Vert_{L^1(0,T, (\Lip', \wW))} \le C(\Omega) \:  \int_0^T   \wW(  {f(t)})\: dt. 
		\end{equation}
		Combining this with the fact that $ \vert \tilde\rho^\tau(t)-\rho^{\tau}\vert \leq \tau\: \vert \partial_t \tilde \rho^\tau\vert ,  $ in $[k\tau,(k+1)\tau),$  we obtain
		\begin{equation}
			\Vert \rho^\tau -\tilde\rho^\tau\Vert_{L^1( 0,T , (\Lip', \wW))}  \leq   C(\Omega) \:  \tau\:   \int_0^T   \wW(  {f(t)})\: dt.   \end{equation}
		Having in  mind that  $\rho^\tau$ and $\tilde\rho^\tau$ are bounded in $L^\infty(Q)-\hbox{weak}^*, $  we  conclude \eqref{cvgce1} and \eqref{cvgce2}. 
		To prove that $p(1-\rho)=0$ a.e. in $Q,$ we  employ first the theory of weak compensated compactness (cf.  ~\cite{ACM,Moussa}) with  \eqref{estim2}, \eqref{cvgce1}  and \eqref{cvgce3} to deduce first that
		\begin{equation}\label{weakcomp0}
			p_n \tilde \rho_n  \to p\: \rho,\quad \hbox{ in }L^\infty(Q)-\hbox{weak}^* .
		\end{equation}  
		Then, clearly  
		\begin{equation}\label{weakcomp1}
			p_n (1-\tilde \rho_n)  \to p\:(1 -\rho),\quad \hbox{ in }L^\infty(Q)-\hbox{weak}^*. 	\end{equation}   
		On the other, using  the fact that $p_n (1-  \rho_n)=0,$ a.e. $Q,$ and  \eqref{estim2}, we see that 
		\begin{eqnarray*} 
			0\leq 	\int_0^T\!\!\int_\Omega p_n (1-  \tilde \rho_n) \: dtdx &=&  	\int_0^T\!\!\int_\Omega  p_n (\rho_n- \tilde \rho_n )) \: dtdx   \\   
			\\   		 &\leq  &  \tau_n \int_0^T\!\!\int_\Omega  p_n  \: \partial_t \tilde \rho_n dtdx    \\    \\   
			&\leq  &  \tau_n  \left\{ 	\int_0^T W_1( \widehat {\partial_t \tilde \rho_n(t)}   )\: dt 
			+ 	\int_0^T  	\langle \partial_t \tilde \rho_n  \rangle \: \oint_\Omega  p_n \: dt \right\}  
			\\    \\   
			&\leq  & \tau_n  \left\{ 	\int_0^T W_1( \widehat {f(t)}   )\: dt 
			+ 	\int_0^T  	\langle f(t)  \rangle \: \oint_\Omega  p_n (t)\: dt \right\}  
			\\    \\   
			&\leq  & \tau_n \: C(\Omega,N)    \quad \quad \to 0,\quad \hbox{ as }n\to \infty, \end{eqnarray*} 
		where we use, by Lemma \ref{lptaubounded},  the fact that 	$\oint_\Omega  p_n$ is bounded in $L^\infty(0,T).$  Then  letting $n\to\infty$ and using  \eqref{weakcomp1}, we deduce that $p(1-\rho)=0$ a.e. in $Q.$   
	\end{proof}

	\begin{proof}[Proof of Theorem \ref{thrho2}]
		Remember that  
		\begin{equation}\label{eulerjko0}
			p_{k+1}^\tau \in \K(\rho_{k}^\tau+\tau f_k^\tau-\rho_{k+1}^\tau)=\K \Big (f_k^\tau- \frac{\rho_{k+1}^\tau-\rho_{k}^\tau}{\tau} \Big)
		\end{equation}
		which  can be rewritten as,  for a.e. $t\in [0, T),$
		\begin{equation}\label{eulerjko}
			p^{\tau}(t) \in \K(f^{\tau}(t) - \partial_t \tilde\rho^{\tau}(t)). 
		\end{equation} 
		This implies that, for any  $\xi \in \Lipu$ and for every $\xi\in \Lipu$,  we have 
		\begin{equation}\label{eulerjkodeveloppe}
			\langle \partial_t \tilde\rho_n(t), p_n(t) -\xi \rangle \le   \langle f_n(t), p_n(t)-\xi \rangle \quad \hbox{ a.e.    }t\in [0,T).
		\end{equation}
		Using the fact that 
		$p_n  \: \partial_t \tilde \rho_n    \geq 0,$   we obtain 
		\begin{equation}\label{eulerjkodeveloppe1}
			-	\langle \partial_t \tilde\rho_n(t), \xi \rangle \le   \langle f_n(t), p_n(t)-\xi \rangle \quad \hbox{ a.e.    }t\in [0,T).
		\end{equation}
		Finally, using the convergences   \eqref{cvgce2}, and \eqref{cvgce3} and passing to the limit as $\tau\to 0$ in \eqref{eulerjkodeveloppe1}, we deduce that the pair $(\rho,p)$ satisfies \eqref{def2}. Finally, using \eqref{massk}, we see that 
		$$\frac{d}{dt} \langle \tilde \rho_n\rangle =\langle f_n\rangle , \quad \hbox{ in } \D'(0,T),$$
		so that letting  $n\to\infty,$ and using \eqref{cvgce2},  we get \eqref{massconserv}. The conservation of mass for $\rho_2,$ \eqref{massconserv2}, is directly derived from \eqref{massconserv} and the fact that $\frac{d}{dt} \langle \rho_1\rangle =\langle f\rangle,$ in $\D'(0,T).$\end{proof} 
	
	\section{Algorithm and numerical computation}			 
	\setcounter{equation}{0}
	
	\subsection{Algorithms}
	In this section, we focus on the development of a numerical algorithm to calculate the solutions to the congested pedestrian traffic flow model in crossings introduced in the previous sections.  The core of our approach lies in a prediction-correction scheme, adapted here for the case of two distinct populations. This methodology builds on the work of Maury and al.  \cite{MRS1} and Ennaji and al.  \cite{EIJ}, who previously used this approach to model congestion in crowd motion for a single population.
	
	\medskip   
	For a given time-step $\tau>0,$ and a $\tau-$ discretization  $0=t_0<t_1<.... <t_{m}=\tau [T/\tau]\leq T,$ 
	the prediction-correction algorithm divides the dynamics of the two populations, represented by densities $\rho_1$ and $\rho_2,$  into sequential steps on each interval $[t_k,t_{k+1}[:$ a \emph{prediction step} for $\rho_1$  followed by a \emph{correction step} for $\rho_2.$ 
	
	\medskip  
	Assuming that $\rho_1(t_{k})=:\rho_{1,k}$   and $\rho_2(t_{k})=:\rho_{2,k}$  are known at time $t=t_{k},$ with $\rho_{1,0} =\rho_{01}$  and   $\rho_{2,0} =\rho_{02},$ we now describe the focus of each step:

	\medskip
	\noindent \underline{\textbf{Prediction step:}} This step focuses mainly on the computation of $\rho_{1,k+1} $ transporting the population $\rho_1$ through a spontaneous velocity field. To achieve this, we utilize the transport equation   i.e.
	\begin{equation} \label{PDEtransportk}
		\left\{ 
		\begin{array}{ll} 
			\frac{d\rho_1 }{dt}  +\nabla \cdot (\rho_1\: V)=0, \quad  & \hbox{ in }[t_k,t_{k+1})\times \Omega\\  \\ 	\rho_1(t_k)=\rho_{1,k} ,    &\hbox{ in } \Omega,\end{array}\right. 
	\end{equation}
	where $V$ is derived from a potential, $V=-\nabla \varphi.$  Then, we consider  
	\begin{equation}
		\rho_{1,k+1} =  \rho_1 (t_{k+1}). 
	\end{equation} 
	We consider three primary scenarios for determining $V,$  as outlined in Remark 1:

	\begin{itemize}
		\item  Directed movement: the potential $\varphi$  is given by 
		the geodesics towards the exit ; i.e. taking  $V$ definitely given by $V=-\nabla \varphi,$ where   the potential \(\varphi\) is given by the solution of the Eikonal equation:
		\begin{equation}\label{eikonal}
			\left\{
			\begin{array}{l}
				\|\nabla \varphi\| = 1\quad \text{in} \quad \Omega, \\
				\varphi = 0 \quad \text{on} \quad \Gamma_D.
			\end{array}
			\right. 
		\end{equation}
		Recall that the solution of equation \eqref{eikonal} (in the viscosity sense) gives the fast path to the exit \(\Gamma_D\). The potential \(\varphi\) corresponds to the expected travel time to maneuver towards an exit.  
		
		\item Non-local interaction:  the  potential $\varphi$ is non-local  and is  computed by the convolution of the solution with a Gaussian kernel, taking into account the influence of neighboring individuals.  Within the time interval   $[t_k,t_{k+1}[,$ $\varphi $ is given by  
		\begin{equation}\label{convolution}
			\varphi : = \rho_1^k \star K_\sigma.
		\end{equation}
		Here $K_\sigma$  is the  Gaussian kernel, accounting for the influence of neighboring individuals. Here 
		$$K_\sigma^\mu (x)  = \frac{1}{\sigma \sqrt{2\pi}} \exp \left( -\frac{(x - \mu)^2}{2\sigma^2} \right),$$ 
		where $\sigma >0$  is the standard deviation, which measures the spread or dispersion of the distribution (larger standard deviation indicates a wider, flatter curve), and $\mu$  is the mean, or average, of the distribution (it  represents the center of the bell curve). 
		\item Diffusion-based interaction: the  potential $\varphi$ is  obtained by solving a stationary PDE with the Laplacian operator : 
		\begin{equation}
			\left\{ \begin{array}{ll}
				-\Delta \varphi = \rho_1^k \quad & \hbox{ in }\Omega\\  \\ 
				\varphi =0 & \hbox{ on }\partial \Omega.
			\end{array} \right. 
		\end{equation}
		This enables to capture  diffusion-like interactions within the population. This is analogous to considering the vector field potential from equation \eqref{convolution}, where we effectively replace $K_\sigma$ with the Green's function to model a form of Brownian interaction that respects Dirichlet the boundary conditions.
	\end{itemize}

	\medskip
	\noindent \underline{\textbf{Correction step:}} 
	Since the prediction step may result in a density  $  \rho_{1,k+1} $   that violates  the constraint; 
	i.e.  $\hbox{ess sup}_x(\rho_{1,k+1}(x) + \rho_{2,k}(x))> 1$,  we implement a correction step to adjust $\rho_{2,k}$ and ensure that the total population density remains within the permissible bounds.  While,   \eqref{Evol2} may not be ideal for the theoretical study of the dynamics, as discussed earlier, it proves to be quite effective for numerical simulations in  the correction phase, as we'll see below.   More precisely, to achieve  the correction we use  the  Euler-Implicit schema associated with the dynamic. 
	\begin{equation}\label{Evol3}
		\left\{ 
		\begin{array}{ll} 
			\frac{d\rho_2 }{dt}  +\partial {I\!\! I}_{\Lipu}(p)\ni 0, \quad  & \hbox{ in }[t_k,t_{k+1})\\  \\
			p\geq 0, \: 0\leq \rho_2\leq 1-\rho_1,\: p(\underbrace{1- \rho_{1,k+1}}_{=:\kappa_{k}}-\rho_2 ) =0 	 ,    &\hbox{ in } Q\\   	\rho_2(t_k)=\rho_{2,k}&\hbox{ in } \Omega  . \end{array}\right. 
	\end{equation} 
	This formulation has the advantage of aligning the dynamics with the framework employed in the work of ~\cite{EIJ}, which has demonstrated excellent numerical performance in the context of congestion for a single population. However, in contrast to the single-population case where the critical density is fixed at $1$, here we use a spatially dependent maximum density value,  $ \kappa_k(x)$  where $0\leq \kappa\leq 1.$  This function, assumed to be given, accounts for potential variations for the population $\rho_2$ in the maximum allowable density over space (and time).

	\bigskip 
	To compute the solution  \eqref{Evol3}, we use  the associate Euler-Implicit schema associated with \eqref{Evol3}, at the step $k.$ That is 
	\begin{equation}\label{Evol3kappak}
		\left\{ 
		\begin{array}{ll} 
			\rho_{2,k+1} +\partial {I\!\! I}_{\Lipu}(p_{k+1} )= \rho  _{2,k} ,  \\  \\
			p_{2,k+1} \geq 0, \: 0\leq \rho_{2,k+1} \leq \kappa_{k} ,\: p_{2,k+1} (\kappa_{k}  - \rho_{2,k+1} ) =0 ,\end{array}\right. 
	\end{equation} 
	with   $\rho_{2,0} =\rho_{02}$ given.

	\bigskip 
	Using  the $1$-Wasserstein distance, $W_1,$  and the concept of the Kantorovich potential as in the study of the problem \eqref{Evol2}, it is not difficult to see that, the solution of the discrete problem \eqref{Evol3kappak} is also given by 
	\begin{equation}\label{wf0disckappa}
		p_{k+1} \in \K  (\rho_{2,k} -\rho_{2,k+1}  ), \quad    0\leq \rho_{2,k+1}\leq \kappa_{k},\: 0\leq p_{k+1} , \: p_{k+1}( \kappa_{k} -\rho_{2,k+1})=0 . 
	\end{equation} 
	That is $p_{k+1} \in \Lipu$ and 
	\begin{equation}\label{eqpk+1kappa}
		\langle \rho_{2,k}  -\rho_{2,k+1} , p_{k+1} -\psi \rangle\ge 0,  \; \forall \psi \in \Lipu ,  
	\end{equation}
	with  $$  0\leq \rho_{2,k+1}\leq \kappa_{k},\: 0\leq p_{k+1},  \:   p_{k+1}(1-\rho_{2,k+1})=0.$$
	
	\bigskip
	Following the same proof of Lemma \ref{Lemduality}, we   have the following result where we replace the density maximal value   $  1,$ by  space dependent value  $  \kappa_{k}.$  
	
	\begin{lemma}\label{Lemspaceduality}
		Given a time step $\tau>0$, for any $k=1,... m,$ the sequences $\rho_{k}^{2,\tau}$ and  $p_k^{\tau}$ in \eqref{wf0disckappa} are well defined   inductively by 
		\begin{equation}\label{jkoschemekappa}
			\rho_{2,k}^\tau \in \argmin_{u } \Big\{W_1(\rho_{2,k-1}^\tau  -u)  \; : \; \langle \rho_{2,k-1}^\tau  -u \rangle=0,\: u\in L^\infty(\Omega),\: 0\leq u\leq \kappa_{k}  \Big\},
		\end{equation} 
		and 
		\begin{equation}\label{jkoschemepkappa}
			p_{k}^\tau \in  \argmax_{\theta}    \Big\{   \langle \rho_{2,k-1}^\tau  , \theta \rangle  -\int\kappa_{k} \: \theta   \: :\: 0\leq \theta\in \Lipu \Big\} 
		\end{equation}  
		by setting $\rho_{2,0}^\tau =\rho_{02}.$ 
		Moreover, we have 
		\begin{equation}\label{fluxduality}
			\begin{array}{c}
				\rho_{k}^\tau \in \argmin_{u } \inf_{\sigma \in  (L^{\infty}(\Omega)')^N }  \Big\{	\Vert \sigma\Vert_{(L^{\infty}(\Omega)')^N}    \; : \; \langle  \sigma,    \nabla \xi \rangle =\int_\Omega(\rho_{2,k-1}^\tau  -u)\: \xi\: dx, \\   \\ \quad \quad \: \forall\: \xi \in \Lipu  ,\: u\in L^\infty(\Omega),\: 0\leq u\leq \kappa_{k}  \Big\},\end{array}
		\end{equation}  
		where the duality bracket is taken in the sense of    $({(L^{\infty}(\Omega)')^N,L^{\infty}(\Omega)^N}).$
	\end{lemma} 
	\begin{proof} 
		The proof follows by adapting the  techniques employed in Lemma \ref{Lemduality}. In this case, we replace the constraint $0\leq u\leq 1,$ by a spatially varying constraint of the form $0\leq u\leq \kappa_{k}.$   Subsequently, the duality result expressed in \eqref{fluxduality} follows from the fact  
		$$W_1(\rho_{2,k-1}^\tau  -u) = \inf_{\sigma \in  (L^{\infty}(\Omega)')^N }  \Big\{	\Vert \sigma\Vert_{(L^{\infty}(\Omega)')^N}    \; : \; \langle  \sigma,    \nabla \xi \rangle =\int_\Omega(\rho_{2,k-1}^\tau  -u)\: \xi\: dx,   \quad \: \forall\: \xi \in \Lipu\Big\}. $$
	\end{proof}

	\medskip  
	Coming back to the computation of $\rho_2^{k+1}$ in the correction step through the scheme \eqref{Evol3kappak}, we use Lemma \ref{Lemspaceduality} and proceed  with a minimum flow process  in \eqref{fluxduality}. More precisely, we consider \(\rho_2^{k+1}\) to be the solution to the following regularized optimization problem:
	
	\begin{equation}\label{minflowk}
		\begin{array}{c}
			\argmin_{u} \inf_{\sigma}   \Big\{	\int_\Omega \vert \sigma\vert(x) \: dx   \; : \;     \int_\Omega \sigma\cdot \nabla \xi \: dx  =\int_\Omega(\rho_{2,k-1}^\tau  -u)\: \xi\: dx, \\   \\ \quad \quad \: \forall\: \xi \in \Lipu  ,\: u\in L^\infty(\Omega),\: 0\leq u\leq \kappa_{k}  \Big\}	\end{array}
	\end{equation}
	This formulation allows to update the density \(\rho_2\) while respecting the constraint \(\rho_1 + \rho_2 \leq 1\), and at the same time to optimize the velocity field \(\sigma\) to adjust the dynamics of the system. This   vector field, \(\sigma\),  represents the flow rate used by population  \(\rho_2\) to relieve congestion.
	
	\subsection{Numerical approximation:} 
	
	As mentioned in the previous section, the approximation of the densities \(\rho_1\) and \(\rho_2\) is performed using a prediction correction strategy. The first step, called prediction, consists of solving the transport equation \eqref{PDEtransportk} using an Euler scheme for time discretization, while the term \(\text{div}(\rho_1 V_1)\) is discretized using the finite volume method. The second step, correction or projection, involves a minimum flow problem, which will be solved using a primal-dual (PD) algorithm. To begin, we will detail the discretization of the problems \eqref{PDEtransportk} and \eqref{minflowk}.
	
	\vspace{0.5cm} 
	
	\noindent \textbf{Domain Discretization:} 
	In this section, we numerically solve equations \eqref{PDEtransportk} and \eqref{eikonal} over the domain \(\Omega\). This domain represents a room surrounded by walls, denoted as \(\Gamma_N\), and featuring an exit door \(\Gamma_D\). The domain is divided into \(m \times n\) control volumes, each with a length and width of \(h\). We define \(C_{i,j}\) as the cell located at position \((i,j)\), and \(\Psi_{i,j}\) as the average value of the quantity \(\Psi\) within cell \(C_{i,j}\). At the boundaries of cell \(C_{i,j}\), \(\omega_{i+1,j}\) and \(\omega_{i,j-1}\) represent the in-flow and out-flow quantities, respectively.
	
	\medskip 
	
	\noindent \textbf{Discretization of the Operators:} The space \( X = \mathbb{R}^{m \times n} \) is equipped with a scalar product and an associated norm as follows:
	\[
	\langle u,v \rangle = h^2 \sum_{i=1}^{m} \sum_{j=1}^{n} u_{i,j} v_{i,j}, \quad \| u \|^2 = \langle u, u \rangle.
	\]
	For \( 1 \leq i \leq m \) and \( 1 \leq j \leq n \), we define the components of the discrete gradient operator via finite differences:
	\[
	(\nabla_h u)^1_{i,j} =
	\left\{\begin{array}{ll}
		\frac{u_{i+1,j} - u_{i,j}}{h} & \text{if} \quad i < m, \\
		0 & \text{if} \quad i = m.
	\end{array}\right.
	\]
	
	\[
	(\nabla_h u)^2_{i,j}  =
	\left\{\begin{array}{ll}
		\frac{u_{i,j+1} - u_{i,j}}{h} & \text{if} \quad j < n, \\
		0 & \text{if} \quad j = n.
	\end{array}\right.
	\]
	Then, the discrete gradient \( \nabla_h : X \rightarrow Y = \mathbb{R}^{m \times n \times 2} \) is given by:
	\[
	(\nabla_h u)_{i,j} = \left( (\nabla_h u)_{1,i,j}, (\nabla_h u)_{2,i,j} \right).
	\]
	Similar to the continuous setting, we define a discrete divergence operator \( \text{div}_h : Y \rightarrow X \), which is the minus of the adjoint of \( \nabla_h \), given by \( \text{div}_h = - \nabla_h^* \). That is,
	\[
	\langle -\text{div}_h \phi, u \rangle_X = \langle \phi, \nabla_h u \rangle_Y \quad \text{for any} \quad \phi = ({\phi_1}, {\phi_2}) \in Y \quad \text{and} \quad u \in X.
	\]
	It follows that the divergence with homogeneous Neuman boundary condition is explicitly given by:
	\[
	\text{div}_h \phi_{i,j} = 
	\left\{\begin{array}{ll}
		\frac{{\phi_1}_{i,j}}{h} & \text{if} \quad i = 1, \\
		\frac{{\phi_1}_{i,j} - {\phi_1}_{i-1,j}}{h} & \text{if} \quad 1 < i < m, \\
		-\frac{{\phi_1}_{m-1,j}}{h} & \text{if} \quad i = m,
	\end{array}\right.
	+ 
	\left\{\begin{array}{ll}
		\frac{{\phi_2}_{i,j}}{h} & \text{if} \quad j = 1, \\
		\frac{{\phi_2}_{i,j} - {\phi_2}_{i,j-1}}{h} & \text{if} \quad 1 < j < n, \\
		-\frac{{\phi_2}_{i,n-1}}{h} & \text{if} \quad j = n.
	\end{array}\right.
	\]
	
	\medskip  
	\noindent \textbf{Discretization of the transport equation \eqref{PDEtransport0}:} We apply a splitting method as follows. Given an end time \( T > 0 \) and a timestep \( \tau > 0 \), we decompose the interval \([0, T]\) into subintervals \([t_k, t_{k+1}]\) and \([t_{k+1}, t_{k+1}]\), with \( k = 0, \dots, n-1 \). On each interval \([t_k, t_{k+1}]\), we solve the  transport equation \eqref{PDEtransportk}  to obtain \( \rho_{1,k+1} \), where \( V = (V_x, V_y) \) is the velocity field given by \( V = -\nabla D \), and \( D \) is given by one of the scenarios we describe above. In the case where   \( D \)  is the distance to the boundary \( \Gamma_D \), we solve the Eikonal equation\eqref{eikonal} using the variational method as described in Appendix-B of the work \cite{EIJ}. Solving \eqref{PDEtransportk}  is done by combining a finite difference method in time with a 2D finite volume method in space.  We approximate the term \( \text{div}(V_1 \rho_1) \) in the cell \( C_{i,j} = [x_{i-1/2}, x_{i+1/2}] \times [y_{i,j-1}, y_{i,j+1/2}] \) as follows:
	
	$$\begin{aligned}
		\left( \text{div}(V_1 \rho_1) \right)_{i,j} &= \frac{1}{\Delta x} \left[ F_{\text{up}} \left( V_{i+\frac{1}{2},j}^x, \rho_{i+\frac{1}{2},j}^k , -\rho_{i+\frac{1}{2},j}^k \right) - F_{\text{up}} \left( V_{i-\frac{1}{2},j}^x, \rho_{i-\frac{1}{2},j}^k , -\rho_{i-\frac{1}{2},j}^k \right) \right] \\
		& \quad + \frac{1}{\Delta y} \left[ F_{\text{up}} \left( V_{i,j+\frac{1}{2}}^y, \rho_{i,j+\frac{1}{2}}^k , -\rho_{i,j+\frac{1}{2}}^k \right) - F_{\text{up}} \left( V_{i,j-\frac{1}{2}}^x, \rho_{i,j-\frac{1}{2}}^k , -\rho_{i,j-\frac{1}{2}}^k \right) \right].
	\end{aligned}$$
	We consider in this discretisation \(\rho = \rho_1\), and  
	\(\left( \text{div}(V  \rho_1) \right)_{i,j}\) is the value of \(\text{div}(V_1 \rho_1)\) in the cell \(C_{i,j}\), \((\Delta x, \Delta y)\) are the spatial discretizations.
	$$F_{\text{up}}(x,y,z) =
	\left\{\begin{array}{ll}
		xy & \text{if} \quad x \geq 0, \\
		xz & \text{if} \quad x < 0.
	\end{array}\right.$$
	For the time discretization, we use the Euler explicit method to approximate the time derivative of the density. The overall scheme can be written as:
	\begin{align}
		{\frac{\rho_{i,j}^{k+1} - \rho_{i,j}^k}{\tau}}
		{+}  {\frac{1}{\Delta x} 
			\left( 
			F_{\text{up}}(V_{i+\frac{1}{2},j}^x, \rho_{i+\frac{1}{2},j}^k, \rho_{i+\frac{1}{2},j}^k, \rho_{i-\frac{1}{2},j}^k)
			- F_{\text{up}}V_{i-\frac{1}{2},j}^x, \rho_{i-\frac{1}{2},j}^k, \rho_{i,j}^k) 
			\right)} \notag \\
		{+}  {\frac{1}{\Delta y} 
			\left( 
			F_{\text{up}}(V_{i,j+\frac{1}{2}}^y, \rho_{i,j+\frac{1}{2}}^k, \rho_{i,j+1}^k, \rho_{i,j-\frac{1}{2}}^k)
			- F_{\text{up}}(V_{i,j-\frac{1}{2}}^y, \rho_{i,j-\frac{1}{2}}^k, \rho_{i,j}^k) 
			\right)}  {=}  {0}.
		\label{eq:upwind}
	\end{align}
	where \( \rho_{i,j}^{k+1} \) is the average value of \( \rho_1 \) in the cell \( C_{i,j} = [x_{i-\frac{1}{2},j}, x_{i+\frac{1}{2},j}] \times [y_{i,j-\frac{1}{2}}, y_{i,j+\frac{1}{2}}] \) at time \( (k+1)\tau \), and \( \rho_{a,b}^k, V_{a,b}^x \) (respectively \( V_{a,b}^y \)) are the values of \( \rho_1 \) and \( V \) at the interface \( x_{a,b} \) (respectively \( y_{a,b} \)) at time \( k\tau \). Notice that in practice, we take \( \Delta x = \Delta y = h \), where \( h \) is the mesh size introduced above.
	
	Using the upwind scheme we have and substituting in (\ref{eq:upwind}), the density \( \rho_{i,j}^{k+1} \) can be written as ( \cite{MRS1}):
	\begin{align}
		\rho_{i,j}^{k+1} 
		& = 
		\rho_{i,j}^k 
		-  \frac{\tau}{h} \left( 
		F_{\text{up}}(V_{i+\frac{1}{2},j}^x, \rho_{i,j}^k,\rho_{i-1,j}^k) 
		- F_{\text{up}}(V_{i-\frac{1}{2},j}^x, \rho_{i-1,j}^k, \rho_{i,j}^k) 
		\right)   \notag \\
		& -  \frac{\tau}{h} \left( 
		F_{\text{up}}(V_{i,j+\frac{1}{2}}^y, \rho_{i,j}^k,\rho_{i,j-1}^k) 
		- F_{\text{up}}(V_{i,j-\frac{1}{2}}^y, \rho_{i,j-1}^k, \rho_{i,j}^k) 
		\right) .
		\label{eq:upwindscheme}
	\end{align}
	We consider that no flux is entering the room from the boundary at \( \partial \Omega \). This is equivalent to impose \( \rho_{i-\frac{1}{2},j}^k V_{i-\frac{1}{2},j}^x = 0 \) and \( \rho_{i,j-\frac{1}{2}}^k V_{i,j-\frac{1}{2}}^y = 0 \) at \( i = 1 \) and \( j = 1 \) respectively. The scheme (\ref{eq:upwindscheme}) is conservative, stable under the CFL condition \( \| V \|_\infty \frac{\tau}{h} \leq \frac{1}{2} \). This condition can be obtained using von Neumann stability analysis (cf. \cite{Charney,Crank}). We summarize this in the following algorithm:
	
	\vspace{0.5cm} 

	\begin{framed}
		\noindent
		\textbf{Algorithm 1: Prediction step} \\
		
		\textbf{1st step. Initialization:} Compute the velocity $\mathbf{V} = (V_x, V_y)$. Choose $\Delta x = \Delta y = h$ and $\tau$ such that 
		\[
		\|\mathbf{V}\|_\infty \frac{\tau}{h} \leq \frac{1}{2},
		\] 
		and take an initial density given by $\rho_{1,(i,j)}^k$ at time $k\tau$. \\
		
		\textbf{2nd step. Update the density:} At time $\left(k+1\right)\tau$, update $\rho_{i,j}$ using:
		\begin{align}
			{\rho_{i,j}^{k+1}} 
			& {=} 
			{\rho_{i,j}^k} 
			{-} 
			{\frac{\tau}{h} \left( 
				F_{\text{up}}(V_{i+\frac{1}{2},j}^x, \rho_{i,j}^k,\rho_{i-1,j}^k) 
				- F_{\text{up}}(V_{i-\frac{1}{2},j}^x, \rho_{i-1,j}^k, \rho_{i,j}^k) 
				\right)} \notag \\
			& {-} 
			{\frac{\tau}{h} \left( 
				F_{\text{up}}(V_{i,j+\frac{1}{2}}^y, \rho_{i,j}^k,\rho_{i,j-1}^k) 
				- F_{\text{up}}(V_{i,j-\frac{1}{2}}^y, \rho_{i,j-1}^k, \rho_{i,j}^k) 
				\right)}.
		\end{align}
	\end{framed}
	
	\textbf{Notice:} In the case where $V_x > 0$ and $V_y > 0$, the update formula reduces to:
	$$\rho_{i,j}^{k+1} = \rho_{i,j}^k 
	- \frac{\tau}{\Delta x} \left[ 
	\rho_{i,j}^k V_{i+\frac{1}{2},j}^x - \rho_{i-1,j}^k V_{i-\frac{1}{2},j}^x 
	\right] 
	- \frac{\tau}{\Delta y} \left[ 
	\rho_{i,j}^k V_{i,j+\frac{1}{2}}^y - \rho_{i,j-1}^k V_{i,j-\frac{1}{2}}^y 
	\right].$$
	\vspace{0.5cm} 
	
	\noindent \textbf{Discretization of the correction :} 
	When calculating $\rho_1^{k+1}$, which moves and may cause a sum with $\rho_2^k$ to be greater than 1, we need to adjust $\rho_2^k$ to ensure that the constraint $\rho_1 + \rho_2 \leq 1$ is met. Thus, the next step is to deal with congestion by solving the following minimum flow problem \eqref{minflowk}. 
	where $\kappa_k=1-\rho_1^{k+1}.$
	
	First, let us rewrite \eqref{minflowk}  in the form : 
	$$\textbf{(M)} : \min_{\rho_2, \Phi_2} A(\rho_2, \Phi_2) + \mathbb{I}_C(\Lambda(\rho_2, \Phi_2)),$$
	where (we omit the variable \(\tau\) to lighten the notation)
	$$A(\rho_2, \Phi_2) = \int_{\Omega} \tau \:  |\Phi_2(x)| \, dx + \mathbb{I}_{[0,\kappa]}(\rho_2),$$
	$$\Lambda(\rho_2, \Phi_2) = \rho_2 - \tau \, \text{div}(\Phi_2),$$
	and
	$$ B = \mathbb{I}_{{\rho_2^k}}.$$
	Here, \(\mathbb{I}_C\) is the indicator function of the set \( C \), and it is given by:
	$$\mathbb{I}_C(a) = 
	\left\{\begin{array}{ll} 
		0 & \text{if } a \in C, \\
		+\infty & \text{if } a \notin C.
	\end{array}\right.$$
	This problem can be solved efficiently using the Chambolle-Pock’s primal-dual (PD) algorithm (cf. \cite{Chambolle2004}).
	Based on the discrete gradient and divergence operators, we propose a discrete version of (M) as follows:
	
	$$(\text{M})_d : \min_{\rho_2, \Phi_2} \left\{ 
	h^2 \sum_{i=1}^{m+1} \sum_{j=1}^{n+1} \tau   \|\Phi_{2,i,j}\| + \mathbb{I}_{[0,\kappa]}(\rho_2) + \mathbb{I}_{\mathcal{C}}(\Lambda_h(\rho_2, \Phi_2))
	\right\},$$
	where 
	$$\mathcal{C} := \left\{ (a_{i,j}) : a_{i,j} = \rho_{(i,j)}^{k}, \ \forall (i,j) \in [\![1,m]\!] \times [\![1,n]\!], \ \Lambda_h(\rho_2, {\phi_2}) = \rho_2- \tau \, \text{div}_h \Phi_2 \right \}.$$ 
	In other words, the discrete version \( (\text{M})_d \) can be written as 
	$$\min_{\rho_2, \Phi_2} \mathcal{A}_h(\rho_2, \Phi_2) + \mathcal{B}_h(\Lambda_h(\rho_2, \Phi_2)),$$
	or in a primal-dual form as
	$$\min_{\rho_2, \Phi_2} \max_p \mathcal{A}_h(\rho_2, \Phi_2) + \langle u, \Lambda_h(\rho_2, \Phi_2) \rangle - \mathcal{B}_h^*(p),$$
	where 
	$$\mathcal{A}_h(\rho_2, \Phi_2) = h^2 \sum_{i=1}^{m+1} \sum_{j=1}^{n+1} \tau   \|\phi_{2,i,j}\| + \mathbb{I}_{[0,\kappa]}(\rho_2) 
	\quad \text{and} \quad 
	\mathcal{B}_h = \mathbb{I}_{\mathcal{C}}.$$ 
	Notice that in this case, (2.29) has a dual problem that reads:
	$$\min_{\rho_2} \max_{p} h^2 
	\left\{ 
	\sum_{i=1}^m \sum_{j=1}^n p_{i,j} \left( \rho_{i,j}^{k+1} - \rho_{i,j} \right) 
	: \|\nabla_h p_{i,j}\| \leq 1, \ 0 \leq   \rho_2 \leq \kappa
	\right\}.$$
	Then the (PD) algorithm ( \cite{CP})  can be applied to \( (\text{M})_d \) as follows:
	
	\begin{framed}
		\textbf{Algorithm 2: (PD) iterations} \\
		\textbf{1st step. Initialization:} Choose \(\alpha, \beta > 0\), \(\theta \in [0, 1]\), \(\rho_2^0\), \(\Phi_2^0\), and take \(u_0 = \Lambda_h(\rho_2^0, \Phi_2^0)\), \(\bar{p}_0\).
		
		\textbf{2nd step. For} \(l \leq \text{Itermax}\), do:
		
		$$	(\rho_2^{l+1}, \Phi_2^{l+1}) = \text{Prox}_{\beta A_h} \left( \left( \rho_2^l, \Phi_2^l \right) - \beta \Lambda_h^*(\bar{p}_l) \right);$$
		$$	p_{l+1} = \text{Prox}_{\alpha B_h^*} \left( p_l + \alpha \Lambda_h(\rho_2^{l+1}, \Phi_2^{l+1}) \right);$$
		$$	\bar{p}_{l+1} = p_{l+1} + \theta \left( p_{l+1} - p_l \right).$$
	\end{framed}
	Let us recall that the proximal operator is defined as follows: \\
	
	$$\text{Prox}_{\alpha E}(p) = argmin_{q} \left( \frac{1}{2} \|p - q\|^2 + \eta E(q) \right).$$
	
	\noindent \textbf{Computation of the proximal operators}\\
	
	Note that the functionals \( \mathcal{A}_h \) and \( \mathcal{B}_h^* \) can be computed explicitly. Specifically, the functional \( \mathcal{A}_h \) is separable in terms of the variables \( \rho_2 \) and \( \Phi_2 \):
	$$\mathcal{A}_h(\rho_2, \Phi_2) = \mathbb{I}_{[0,1-\rho_1]}(\rho_2) + \|\Phi_2\|_1.$$
	Therefore, \( \text{Prox}_{\eta \mathcal{A}_h} \) consists of the sum of a projection onto the first component and the so-called soft-thresholding. More explicitly:
	$$\left( \text{Prox}_{\mathcal{A}_h}(\rho_2, \Phi_2) \right)_{i,j} = 
	\left( 
	\max\left(0, \min\left(1 - \rho_{1,(i,j)}, \rho_{(i,j)}\right)\right), 
	\max\left(0, 1 - \frac{1}{|\Phi_{2,(i,j)}|} \right) \Phi_{2,(i,j)}
	\right).$$
	
	For \( \mathcal{B}_h^* \), in order to compute \( \text{Prox}_{\alpha \mathcal{B}_h^*} \), we apply Moreau’s identity:
	$$p = \text{Prox}_{\alpha \mathcal{B}_h^*}(p) + \alpha \text{Prox}_{\alpha^{-1} \mathcal{B}_h}(p / \eta),$$
	and note that \( \text{Prox}_{\alpha^{-1} \mathcal{B}_h}(a, b) \) is simply the projection onto \( \mathcal{C} \). Hence, we have:
	$$\left( \text{Prox}_{\alpha \mathcal{B}_h^*}(p) \right)_{i,j} = 
	\left( 
	p_{i,j} - \alpha \text{Proj}_{\mathcal{C}_{i,j}}\left(p_{i,j} / \alpha\right) 
	\right).$$
	
	In summary, the steps in \textbf{Algorithm 3} to solve \( (\text{M})_d \) are as follows:\\
	
	\begin{framed}
		\textbf{Algorithm 3 (PD) iterations for \( (\text{M})_d \)} \\
		
		\textbf{Initialization:} Let \( k = 0 \), choose \( \alpha, \beta > 0 \) such that \( \alpha \beta \|\Lambda_h\|^2 < 1 \). Choose \( \rho_1^0, \rho_2^0, \Phi_2^0 \), and \( p^0 = \bar{p}^0 = p_0 \). \\
		
		\textbf{Primal step:}
		$$\left( \rho_2^{l+1}  \right)_{i,j} =  
		\max \left( 0, \min \left( 1 - \rho_{1,i,j}^l, \rho_{i,j}^l - \beta \bar{p}_{i,j}^l \right) \right),  $$ 
		and $$
		\left(   \Phi_2^{l+1} \right)_{i,j}  =   \max \left( 0, 1 - \frac{1}{|\Phi_{2,i,j}^l - \beta \nabla_h \bar{p}_{i,j}^l|} \right) 
		\left( \Phi_{2,i,j}^l - \beta \nabla_h \bar{p}_{i,j}^l \right)
		.$$
		\textbf{Dual step:}
		\[
		v^{l+1} = p^l + \alpha \rho_2^{l+1} - \alpha \, \text{div}_h\left(\Phi_2^{l+1}\right),
		\]
		\[
		p_{i,j}^{l+1} = v_{i,j}^{l+1} - \alpha \, \text{Proj}{\mathcal{C}{i,j}}\left(v_{i,j}^{l+1} / \alpha\right), 
		\quad 1 \leq i \leq m, \ 1 \leq j \leq n.
		\]
		
		\textbf{Extragradient:}
		\[
		\bar{p}^{l+1} = 2p^{l+1} - p^l.
		\]
		
	\end{framed}
	
	\subsection{Numerical computation:} 
	
	This section presents several illustrative examples. The first three cases consider a transport equation where the vector field $V$ is defined by a specific destination, inducing directed motion in population $1$ and compelling population $2$ to adapt its distribution to accommodate spatial constraints. Examples $4-7$ explore scenarios where the transport vector field is determined by a non-local potential, enabling population $1$ to disperse based on local neighborhood influences through Gaussian convolution. In contrast to previous cases where boundary conditions were inherent to the vector field $V,$ we now examine two types of boundary conditions: Dirichlet, allowing population members to exit the domain, and reflective, where individuals may or may not re-enter. Finally, in Examples $8-9$ we numerically analyze a scenario where the potential field $V$ is directly linked to the diffusion of population $1.$  
	
	\medskip 
	All computational codes were implemented in Python\footnote{Demonstration videos are available at https://www.unilim.fr/pages\_perso/noureddine.igbida/ }. 
	
	\medskip 
	\subsubsection{Example 1:} A population $\rho_1$ is moving rightward and encounters a stationary population $\rho_2$ : 
	\begin{figure}[H]
		\centering
		\includegraphics[width=0.9\textwidth,height=0.14\textheight,center]{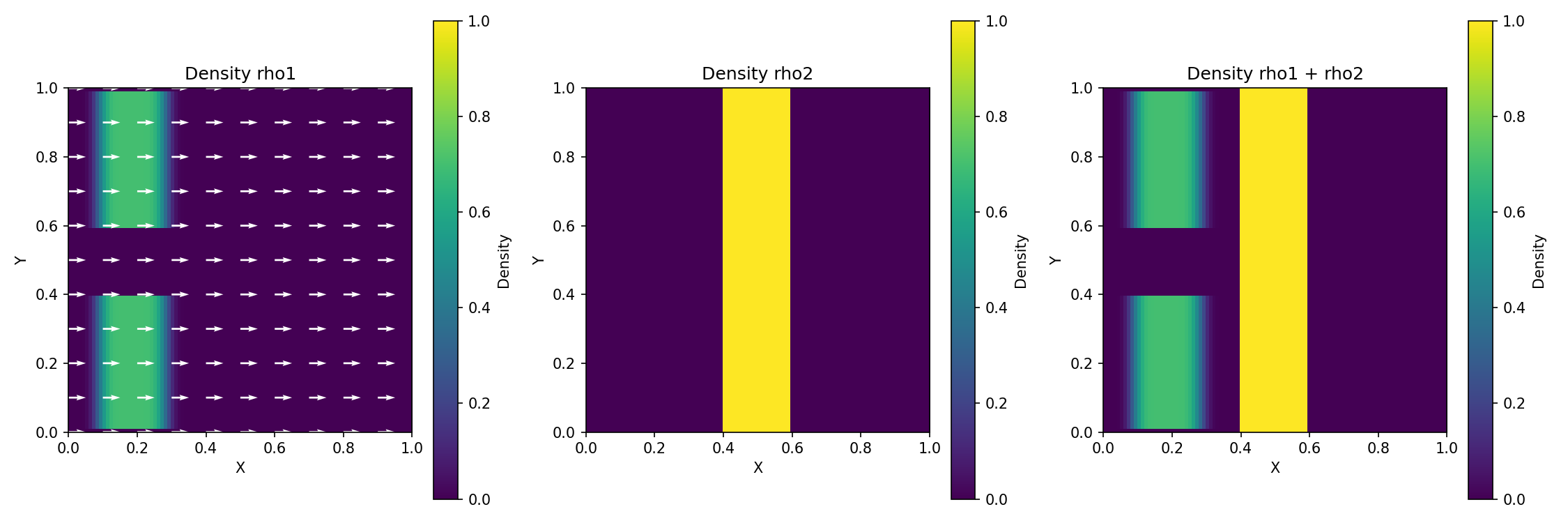}
		\includegraphics[width=0.9\textwidth,height=0.14\textheight,center]{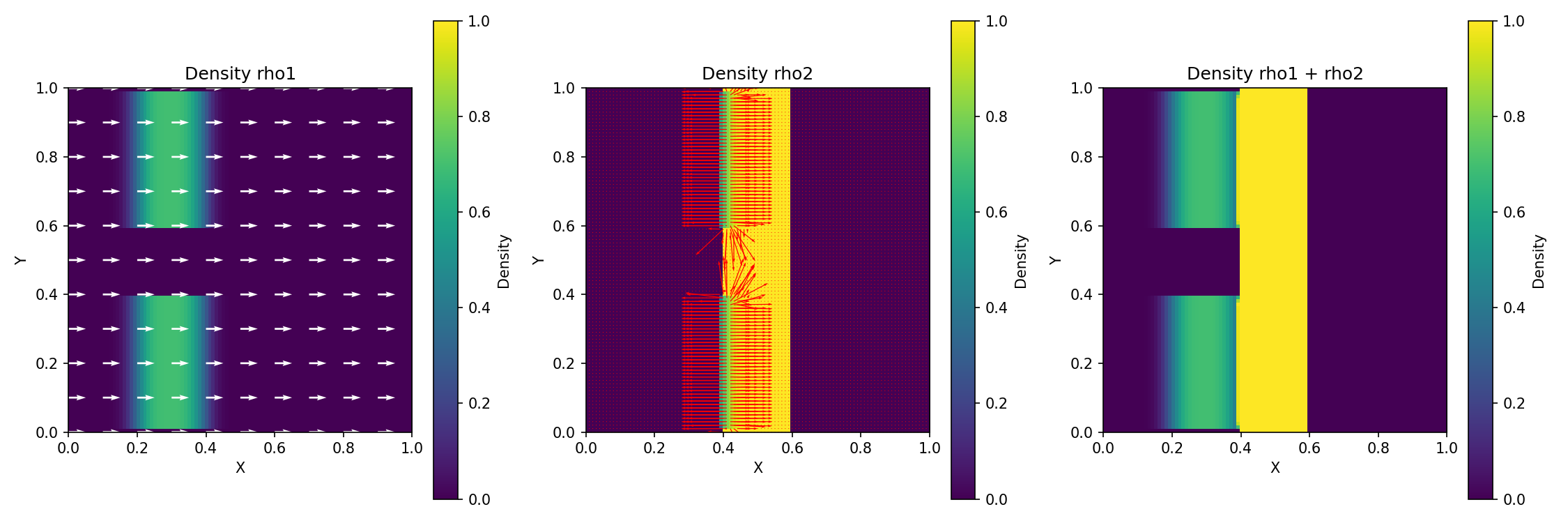}
		\includegraphics[width=0.9\textwidth,height=0.14\textheight,center]{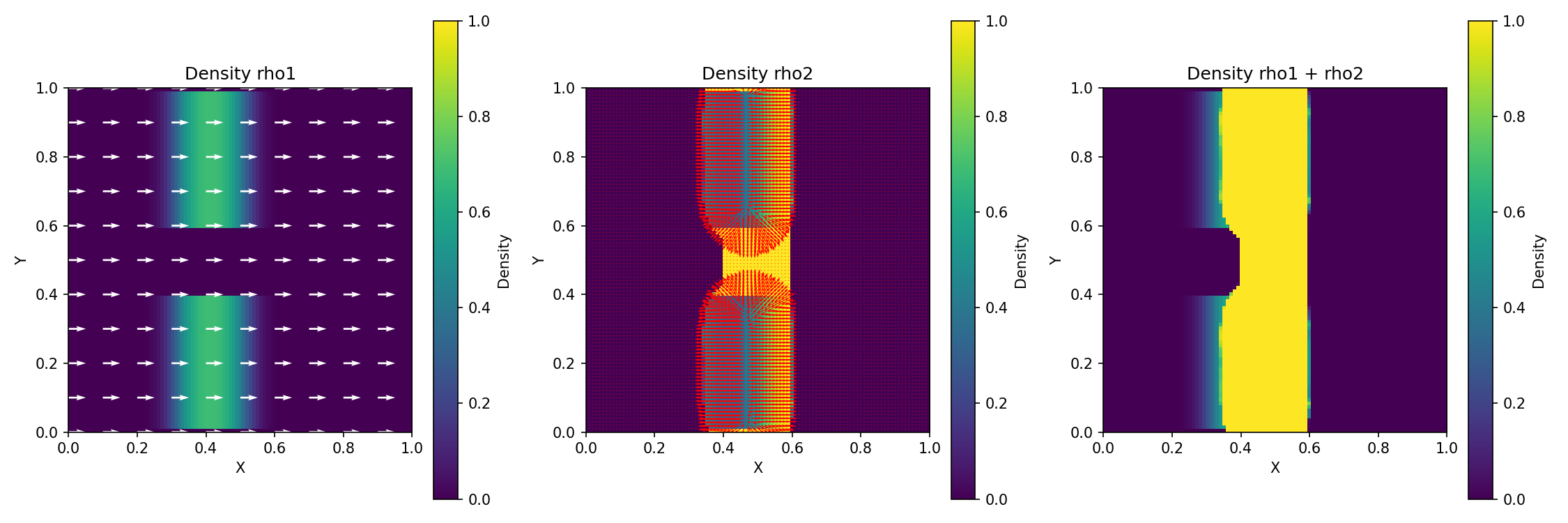}
		\includegraphics[width=0.9\textwidth,height=0.14\textheight,center]{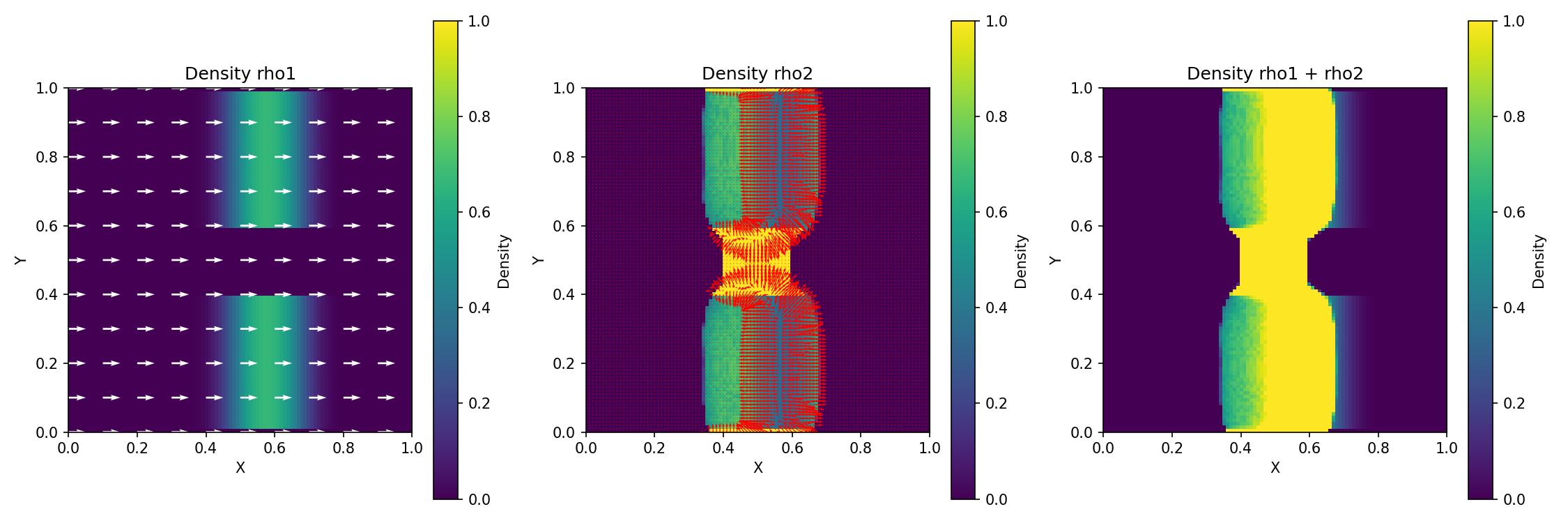}
		\includegraphics[width=0.9\textwidth,height=0.14\textheight,center]{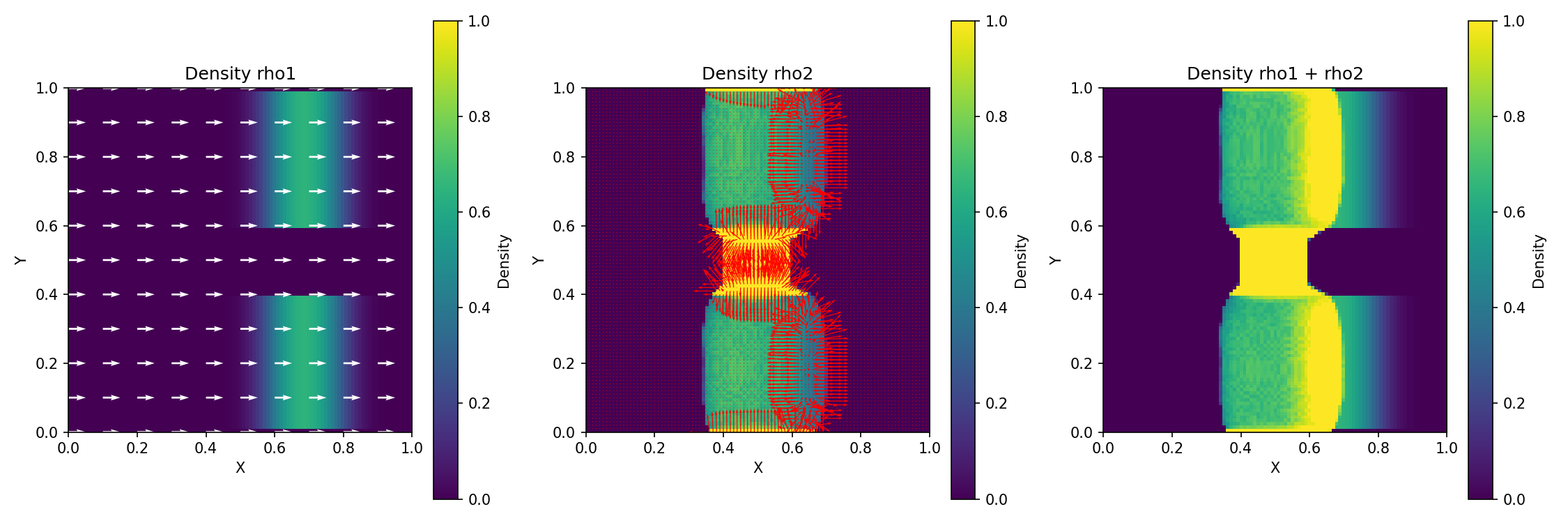}
		\includegraphics[width=0.9\textwidth,height=0.14\textheight,center]{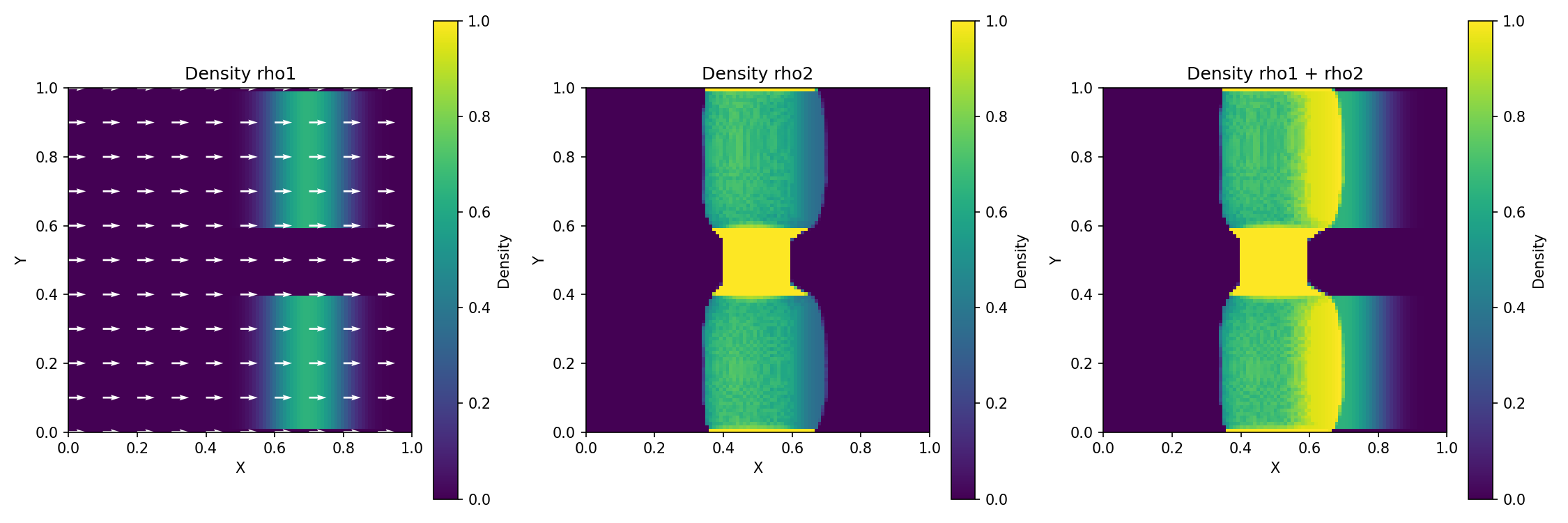}
		\caption{Snapshots of  $\rho_2$ (left) moving along the red-patterned field, interacting with $\rho_1$ (moving along the white field, center). The right image visualizes the combined density $\rho_1+ \rho_2$. The red movement of population $\rho_2$ is only triggered when it encounters population $\rho_1$ and $\rho_1+ \rho_2= 1$}
		\label{fig:img1-1}
	\end{figure}

	\medskip 
	
	\subsubsection{Example 2:} As in Example $1,$ the population $\rho_1$ is moving rightward and encounters a stationary population $\rho_2$: 
	\begin{figure}[H]
		\centering
		\includegraphics[width=0.9\textwidth,height=0.14\textheight,center]{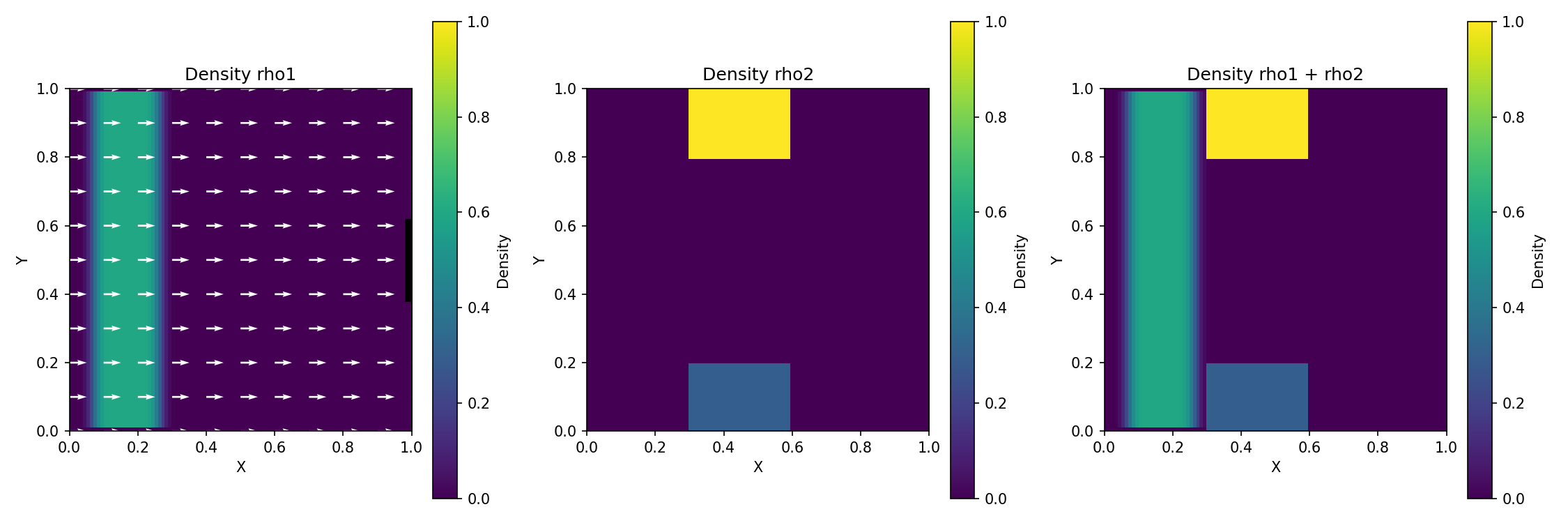}
		\includegraphics[width=0.9\textwidth,height=0.14\textheight,center]{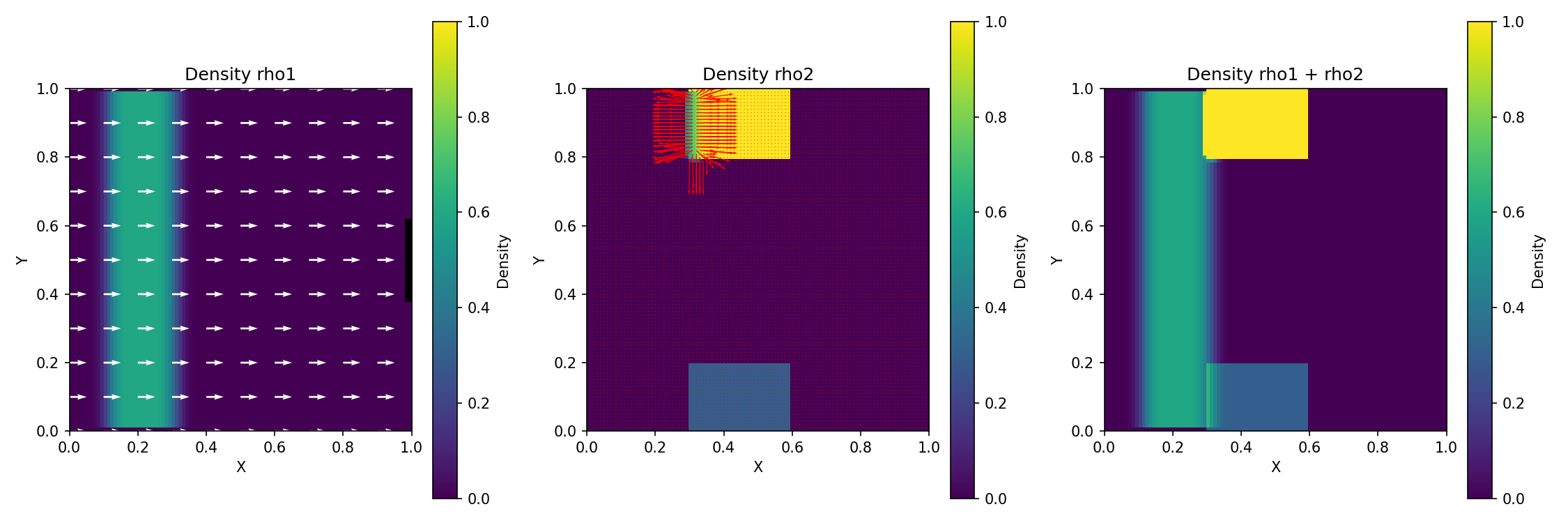}
		\includegraphics[width=0.9\textwidth,height=0.14\textheight,center]{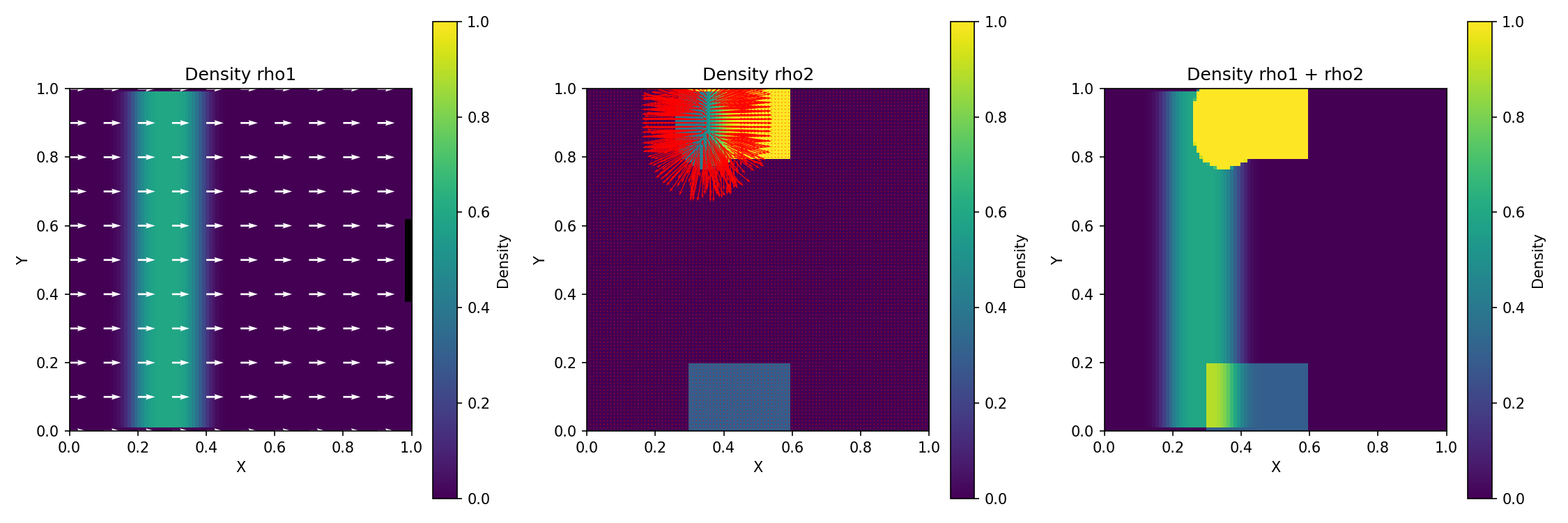}
		\includegraphics[width=0.9\textwidth,height=0.14\textheight,center]{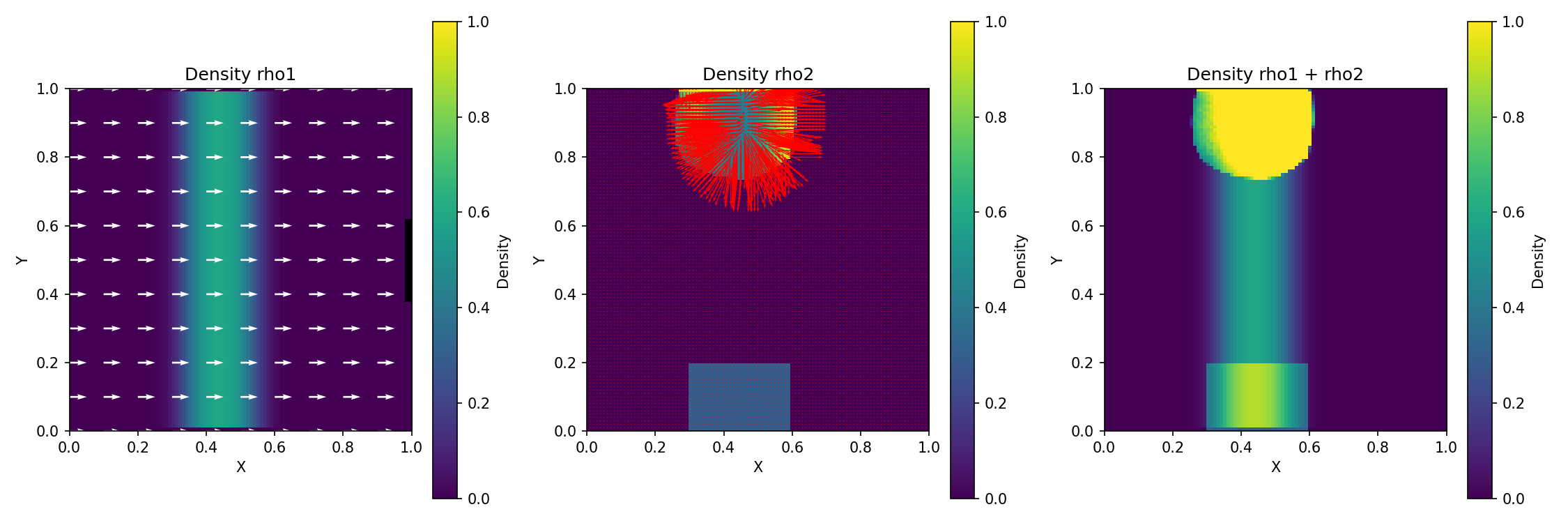}
		\includegraphics[width=0.9\textwidth,height=0.14\textheight,center]{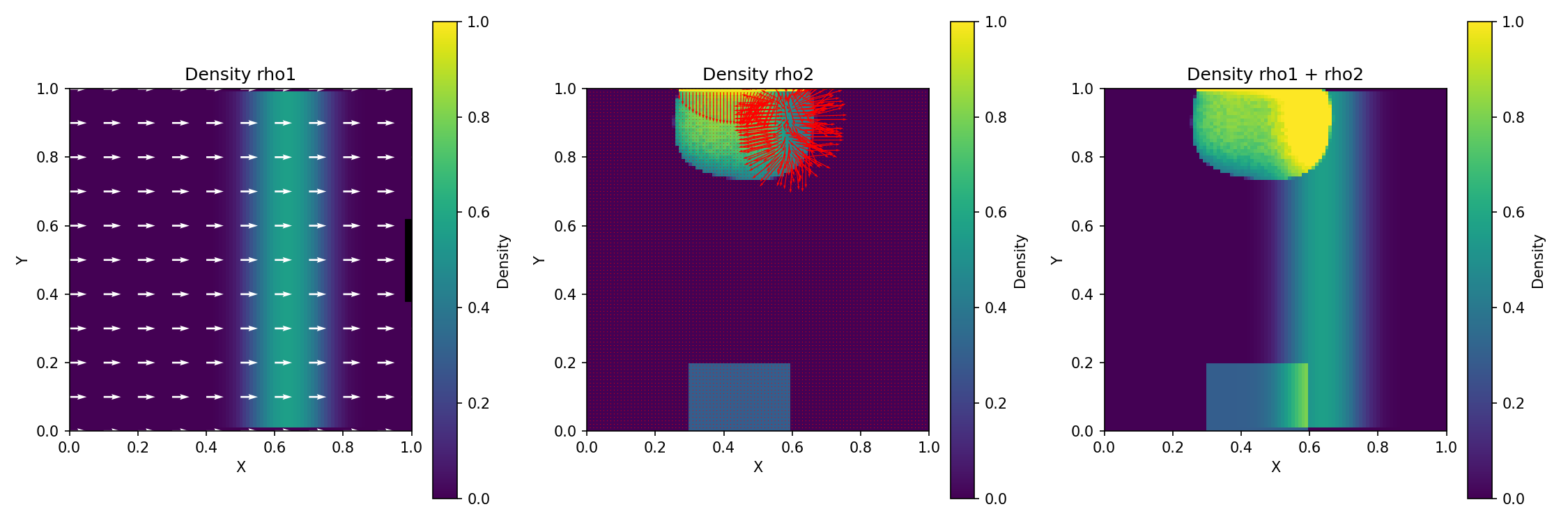}
		\includegraphics[width=0.9\textwidth,height=0.14\textheight,center]{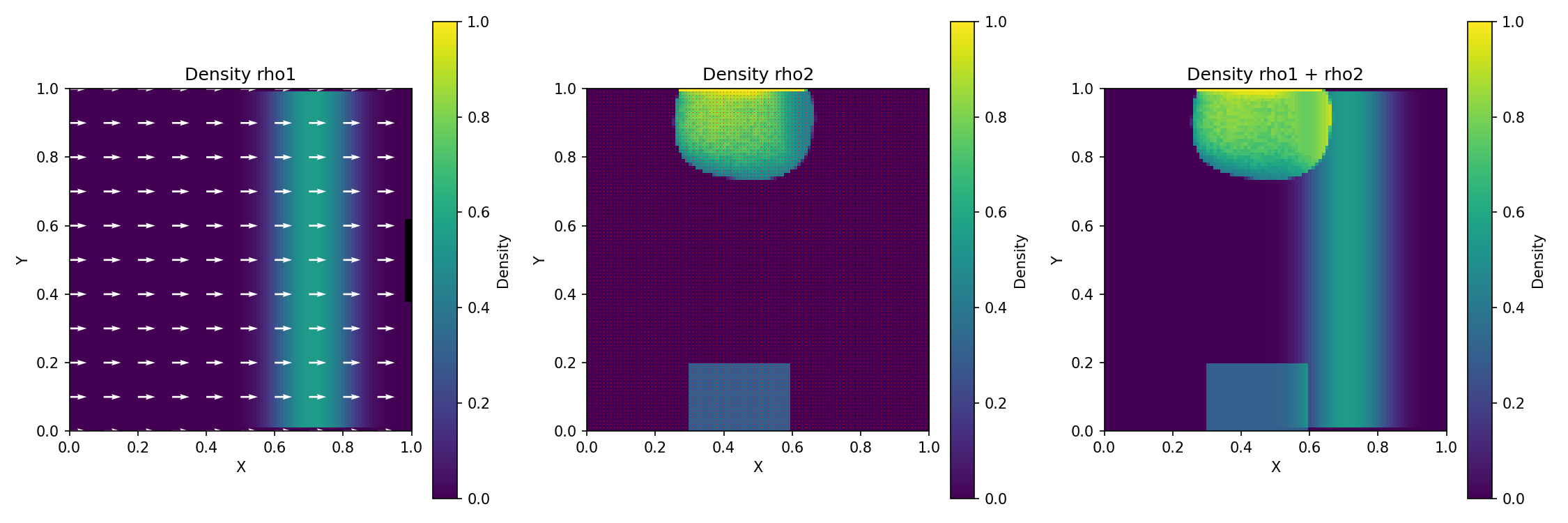}
		\caption{The absence of $\rho_2$ flux (red arrows) in the bottom block can be attributed to the enforcement of the maximum density constraint throughout the dynamical evolution.}
		\label{fig:img1-2}
	\end{figure}
	
	\subsubsection{Example 3:} A rotating population $\rho_1$ meets a static population $\rho_2$ along its trajectory:
	\begin{figure}[H]
		\centering
		\includegraphics[width=0.9\textwidth,height=0.14\textheight,center]{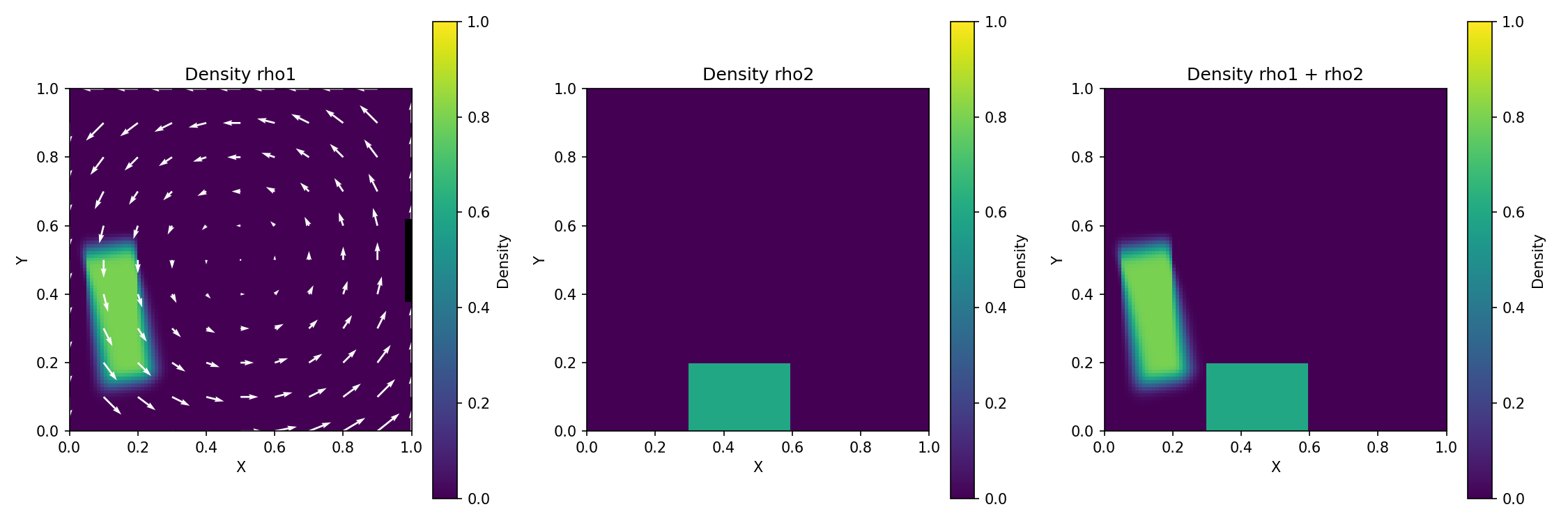}
		\includegraphics[width=0.9\textwidth,height=0.14\textheight,center]{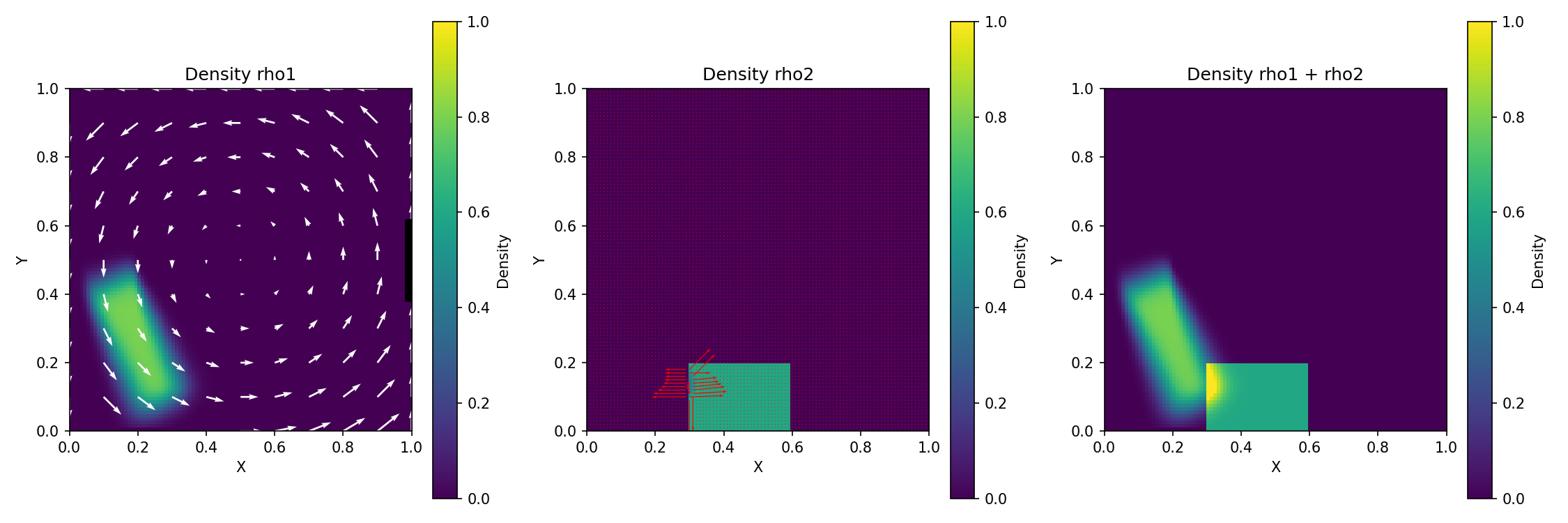}
		\includegraphics[width=0.9\textwidth,height=0.14\textheight,center]{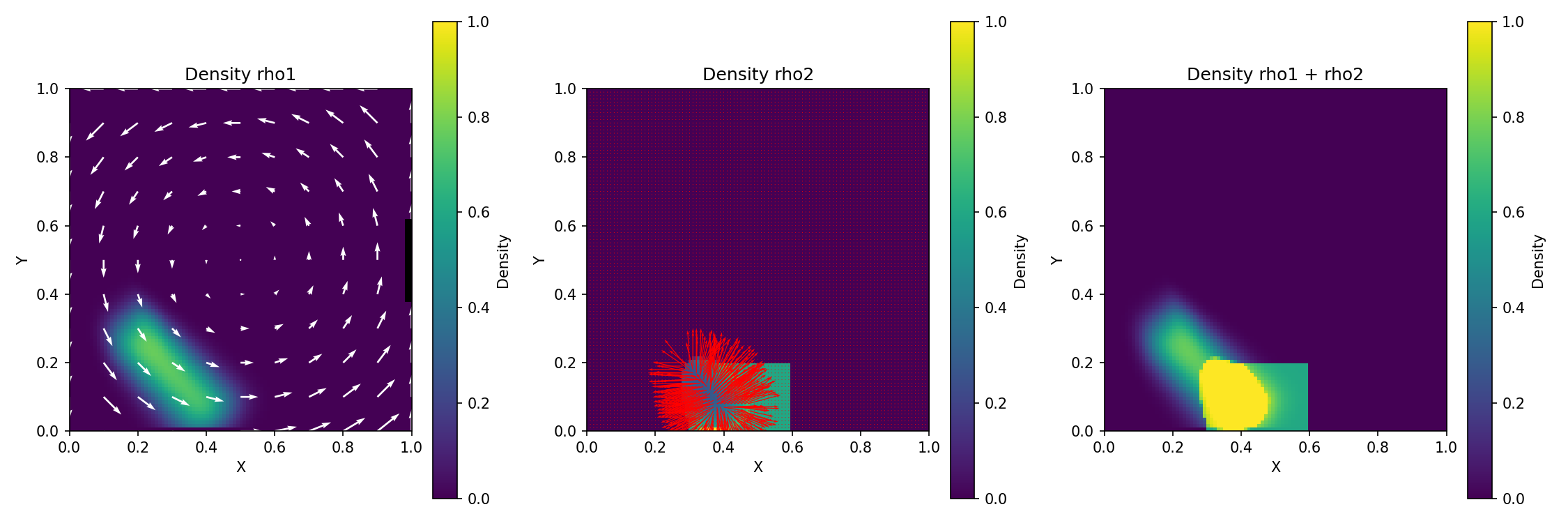}
		\includegraphics[width=0.9\textwidth,height=0.14\textheight,center]{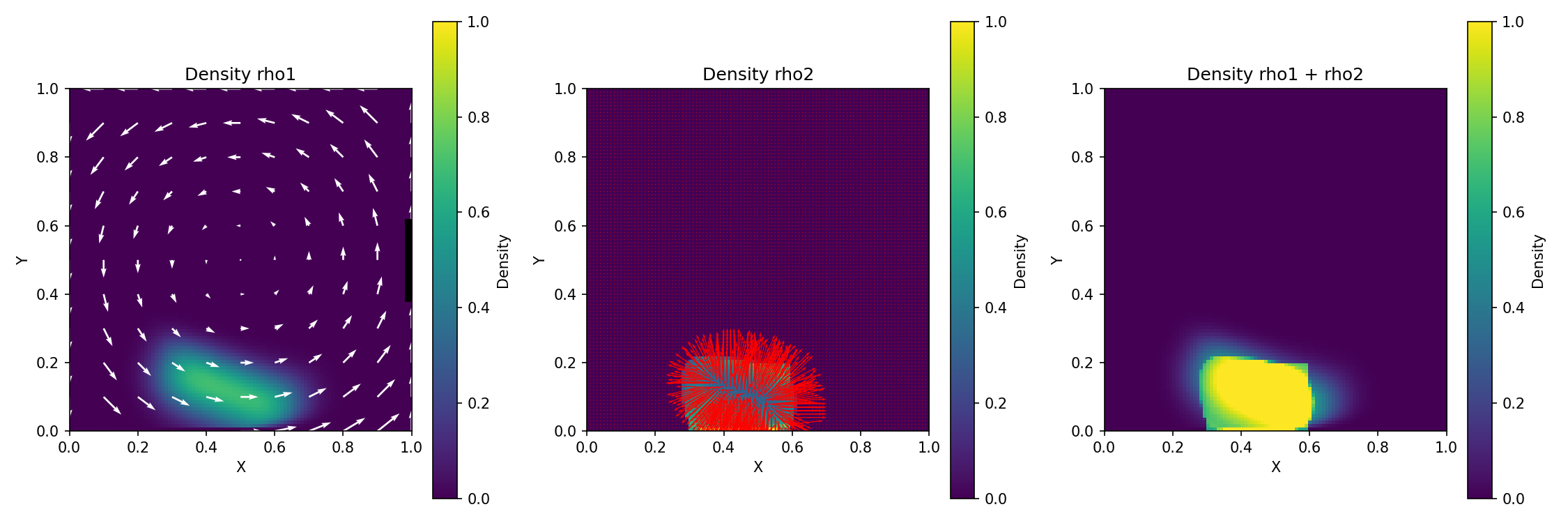}
		\includegraphics[width=0.9\textwidth,height=0.14\textheight,center]{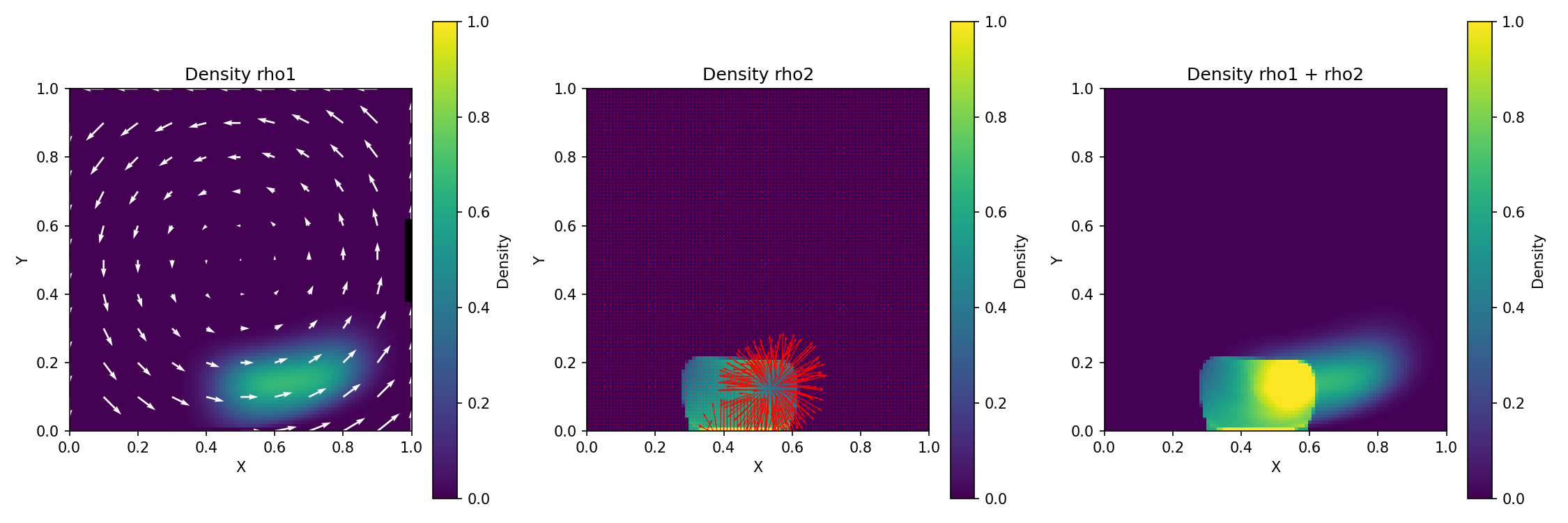}
		\includegraphics[width=0.9\textwidth,height=0.14\textheight,center]{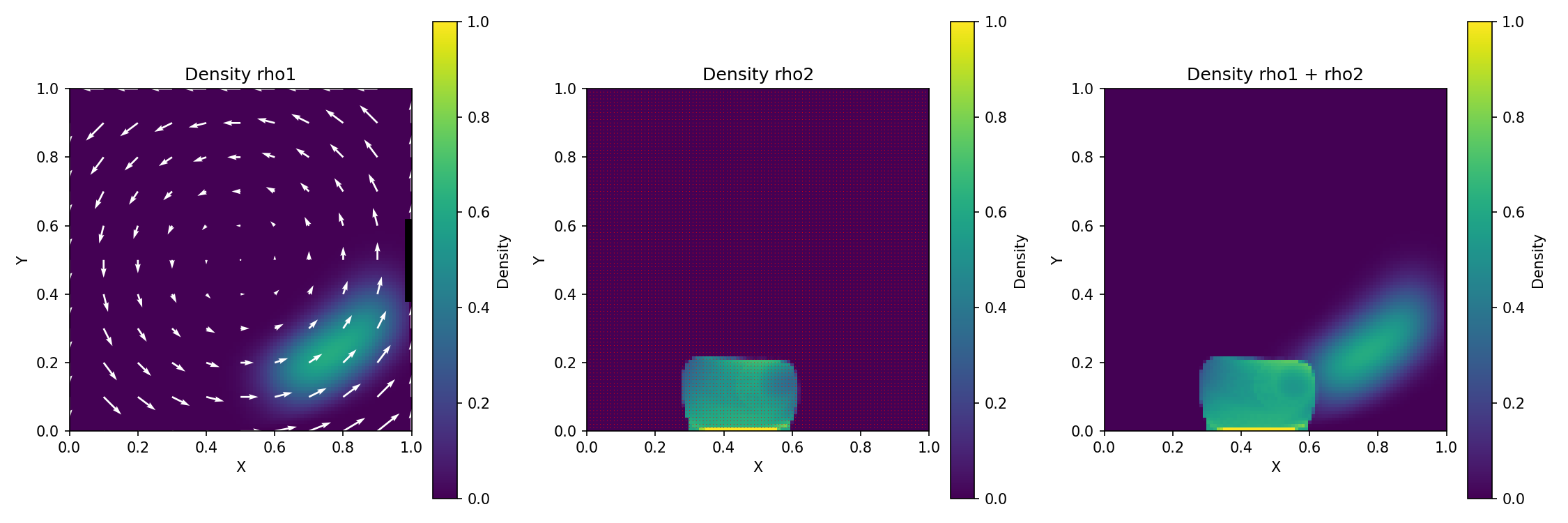}
		\caption{Snapshots of  $\rho_2$ (left) totating  along the red-patterned field, interacting with $\rho_1$ (moving along the white field, center).}
		\label{fig:img1-3}
	\end{figure}
	
	\subsubsection{Example 4:} A population $\rho_1$  spreads out to reduce density variations in its immediate surroundings through a Gaussian smoothing process. This expanding population eventually meets a stationary population $\rho_2$ surrounding it. The initial distribution of $\rho_1$ is a combination of four Gaussian peaks, and $\rho_2$ is simply the complement of $\rho_1.$ The convolution process effectively extends the domain of the population $\rho_1$ with zeros, forcing its members to leave the domain. 
	\begin{itemize}
			\item[]\textbf{Data set} 
		\item $\rho_{01}=0.9 (K_\lambda^{\mu_1} +K_\lambda^{\mu_2} + K_\lambda^{\mu_3}+K_\lambda^{\mu_4} )$ in $\Omega:=[0,1)\times [0,1]$
		with $\lambda =0.1,$ $\mu_1=(0.3, 0.7),$ $\mu_2(0.7, 0.3),$ $\mu_3=(0.3, 0.3)$ and $\mu_4=(0.7, 0.7)$
		
		\item $\rho_{02}=0.9(1-\rho_{01}).$
		
		\item The vector field potential here is given by $\varphi = \rho_1\star K_\sigma^\mu,$ where $\sigma =5$ and $\mu=(0,0).$  
	\end{itemize}

	\begin{figure}[H]
		\centering
		\includegraphics[width=0.9\textwidth,height=0.14\textheight,center]{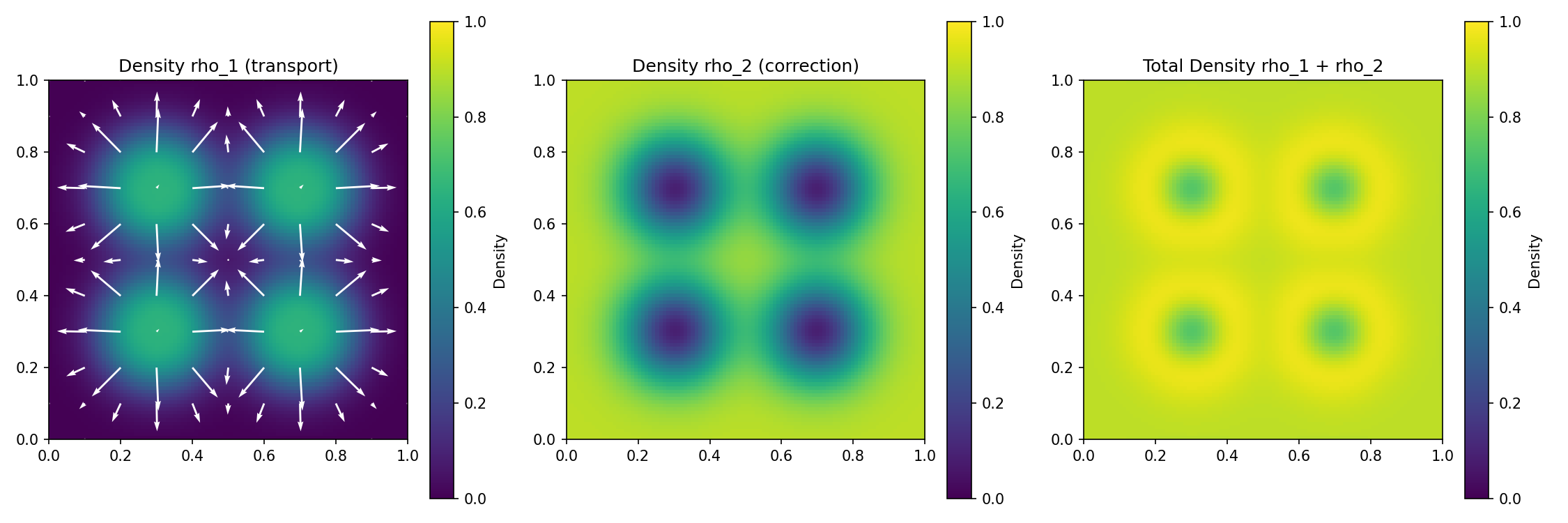}
		\includegraphics[width=0.9\textwidth,height=0.14\textheight,center]{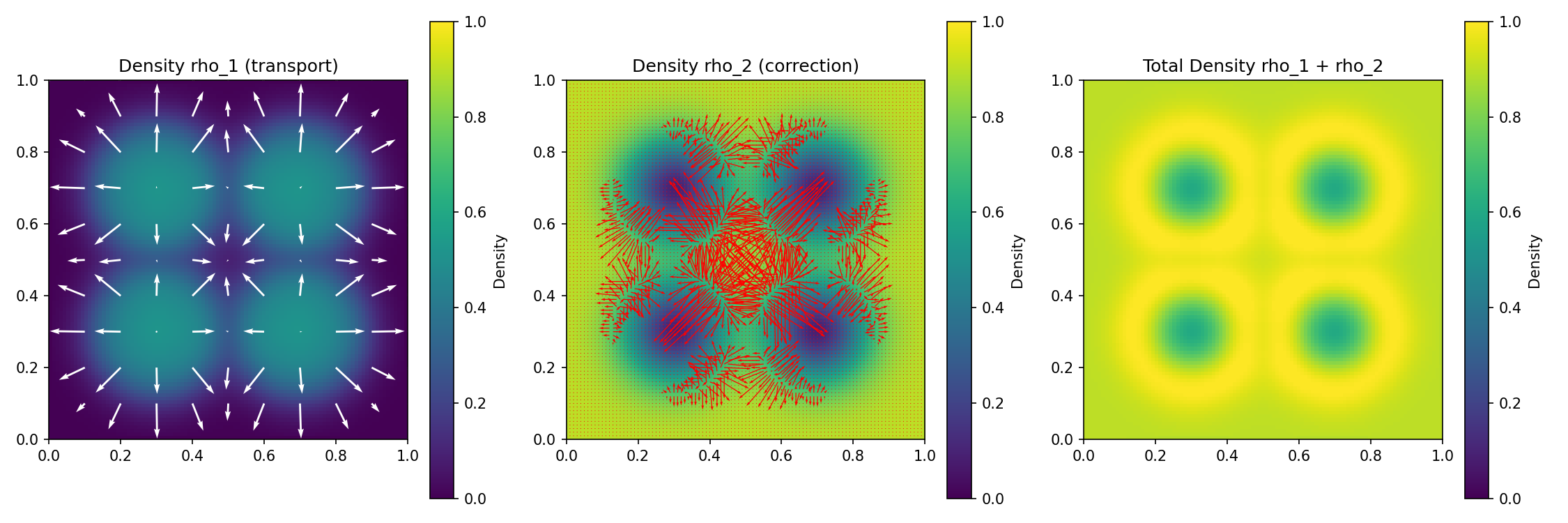}
		\includegraphics[width=0.9\textwidth,height=0.14\textheight,center]{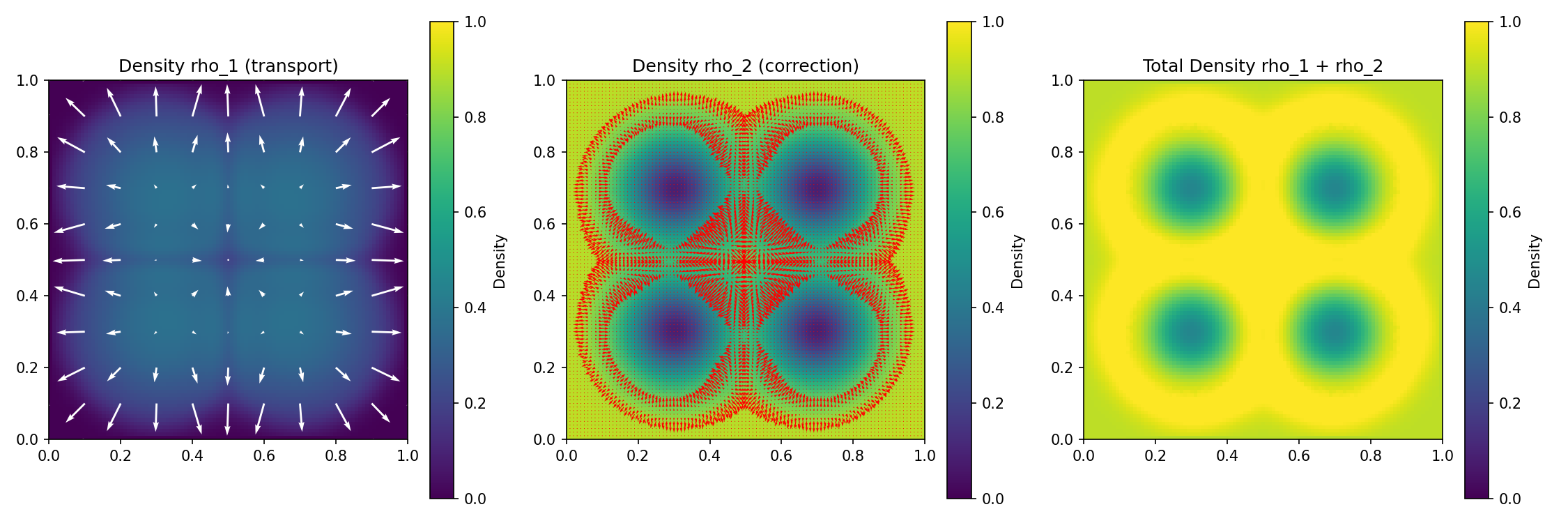}
		\includegraphics[width=0.9\textwidth,height=0.14\textheight,center]{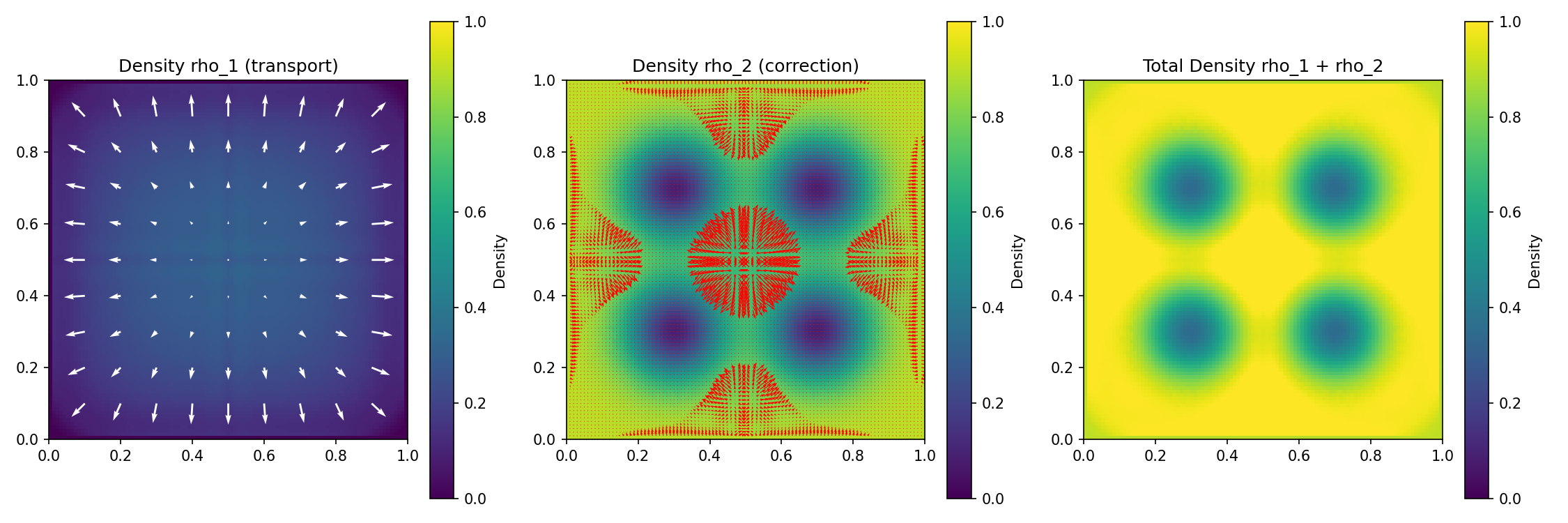}
		\includegraphics[width=0.9\textwidth,height=0.14\textheight,center]{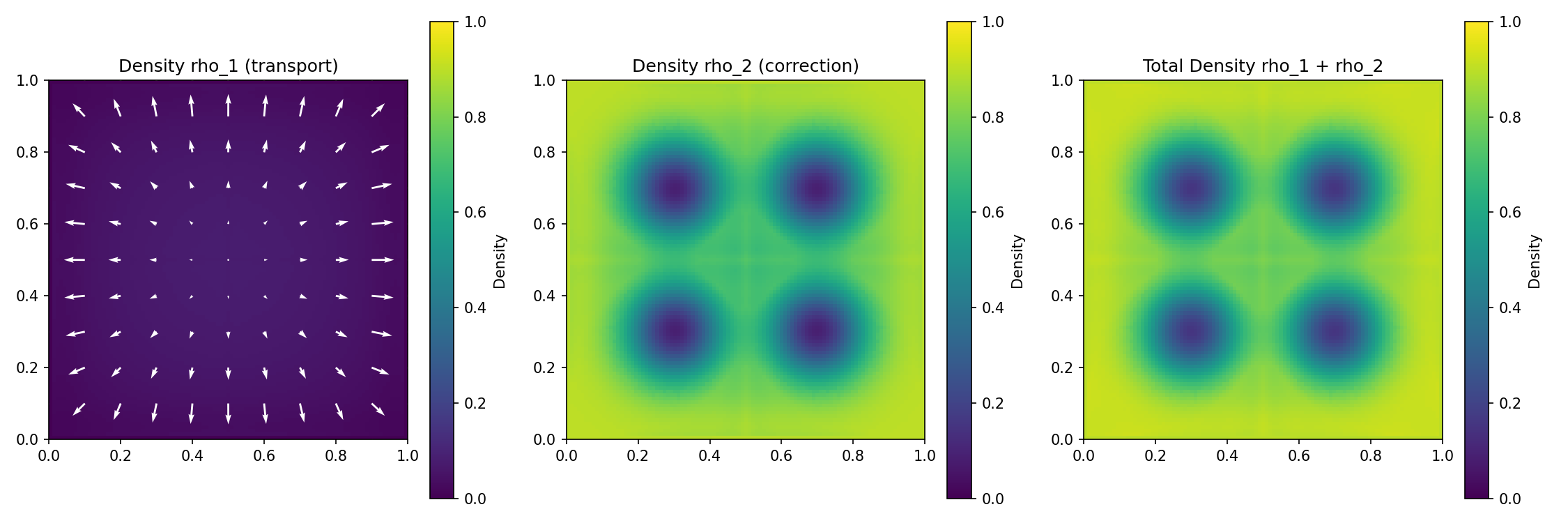}
		\caption{Snapshots of  $\rho_2$ (left) moving along the red-patterned field, interacting with $\rho_1$ (moving along the white field, center). The right image visualizes the combined density $\rho_1+ \rho_2$. Population 2’s red-indicated movement is triggered only when it encounters population $\rho_1$ and$\rho_1+ \rho_2= 1.$ By the end of the simulation, the entire population $\rho_1$ continues to leave the domain without congestion, and population $\rho_2$ is left with its final distribution from the previous time step.}
		\label{fig:img1-4}
	\end{figure} 
	
	\subsubsection{Example 5:} In this case, we use the same data set as in Example 4, but we modify the boundary condition. We implement a "reflect" type boundary condition that reflects the value closest to the boundary. This technique can be effectively used  when a pedestrian reaches the boundary, instead of simply stopping or disappearing, the "reflect" condition simulates a scenario where the pedestrian "bounces back" into the domain. This is achieved by mirroring the pedestrian density across the boundary.
	
	\begin{figure}[H]
		\centering
		\includegraphics[width=0.9\textwidth,height=0.14\textheight,center]{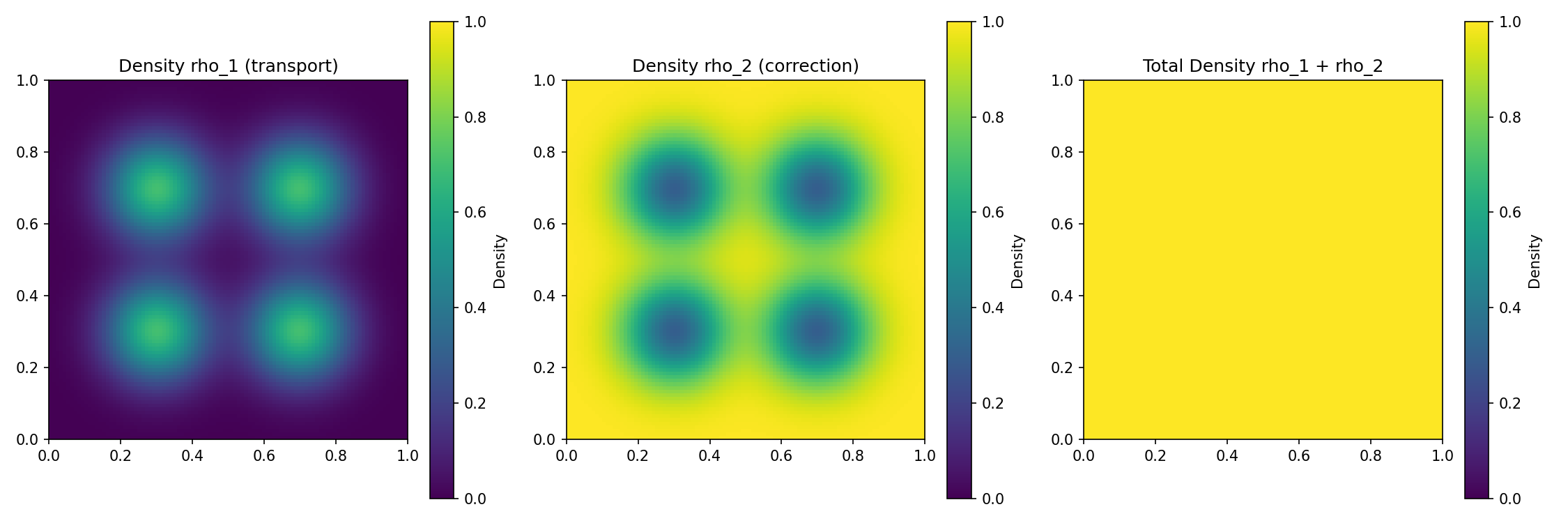}
		\includegraphics[width=0.9\textwidth,height=0.14\textheight,center]{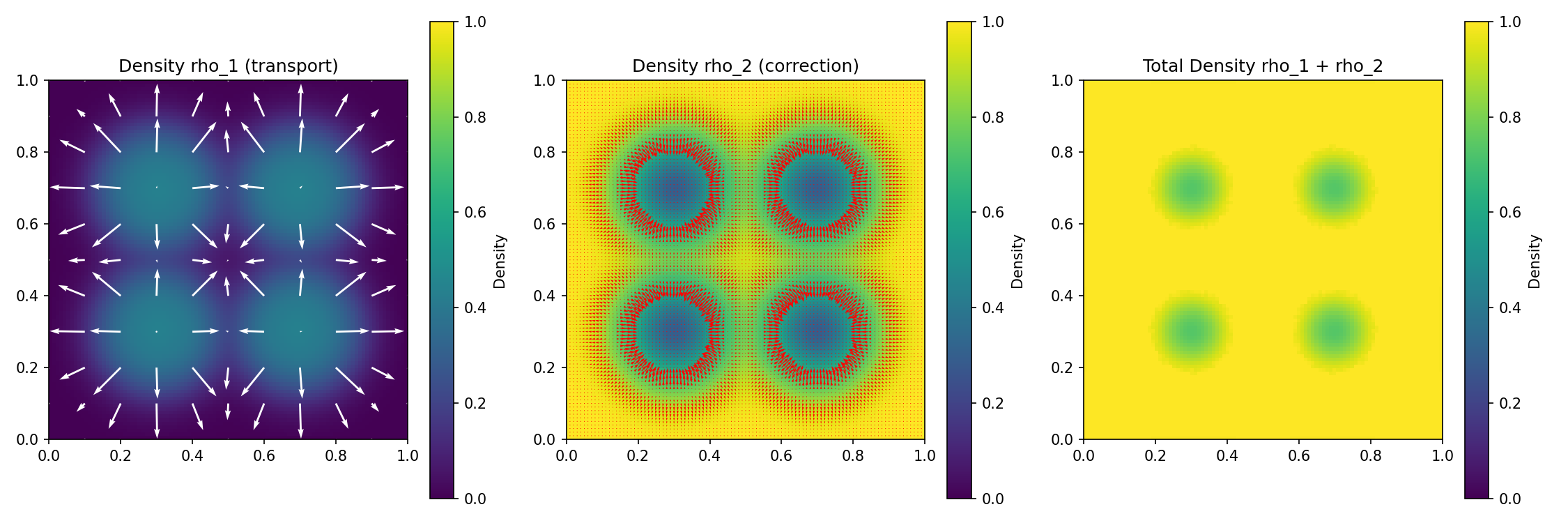}
		\includegraphics[width=0.9\textwidth,height=0.14\textheight,center]{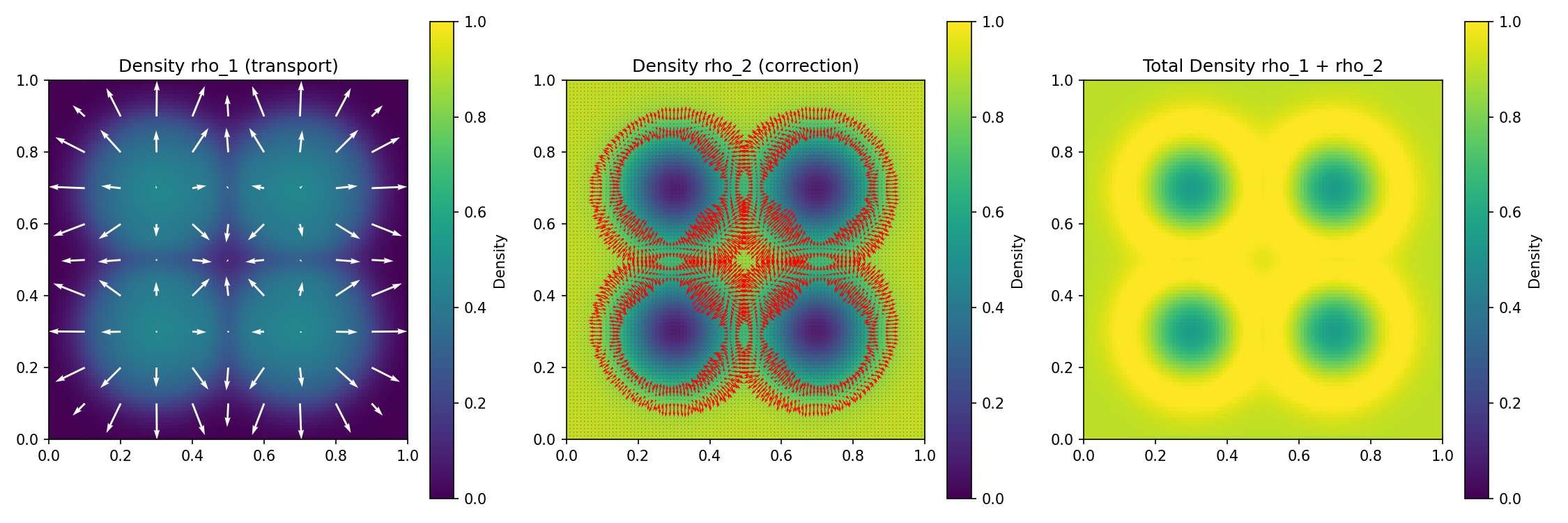}
		\includegraphics[width=0.9\textwidth,height=0.14\textheight,center]{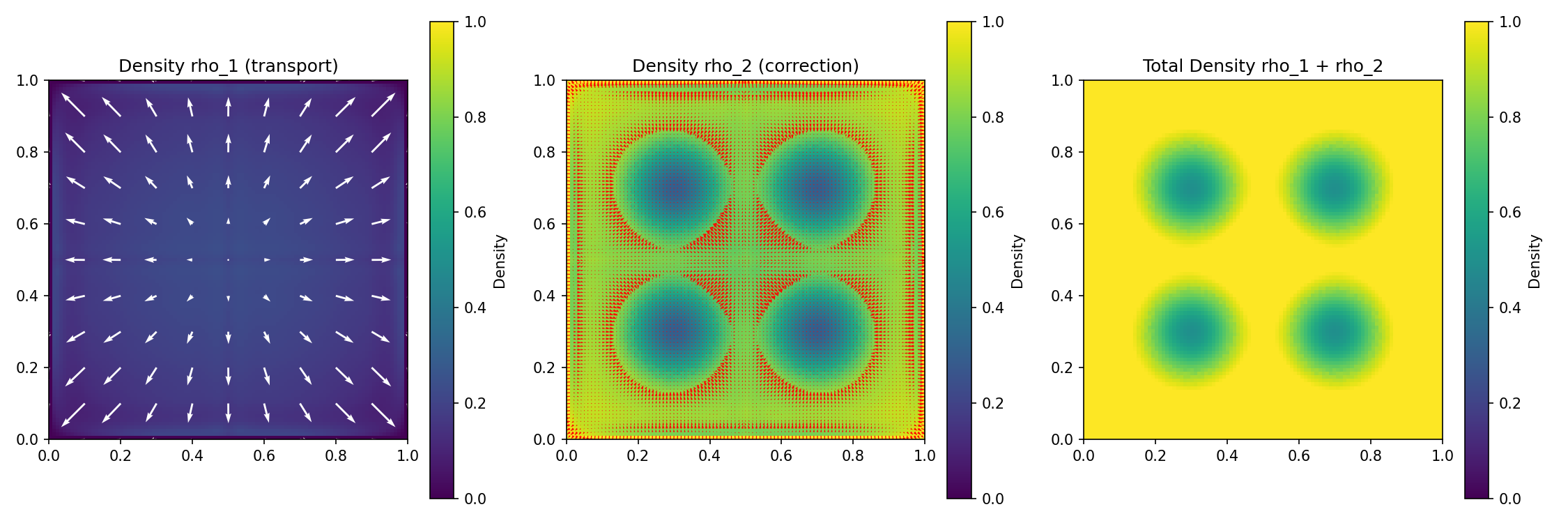}
		\includegraphics[width=0.9\textwidth,height=0.14\textheight,center]{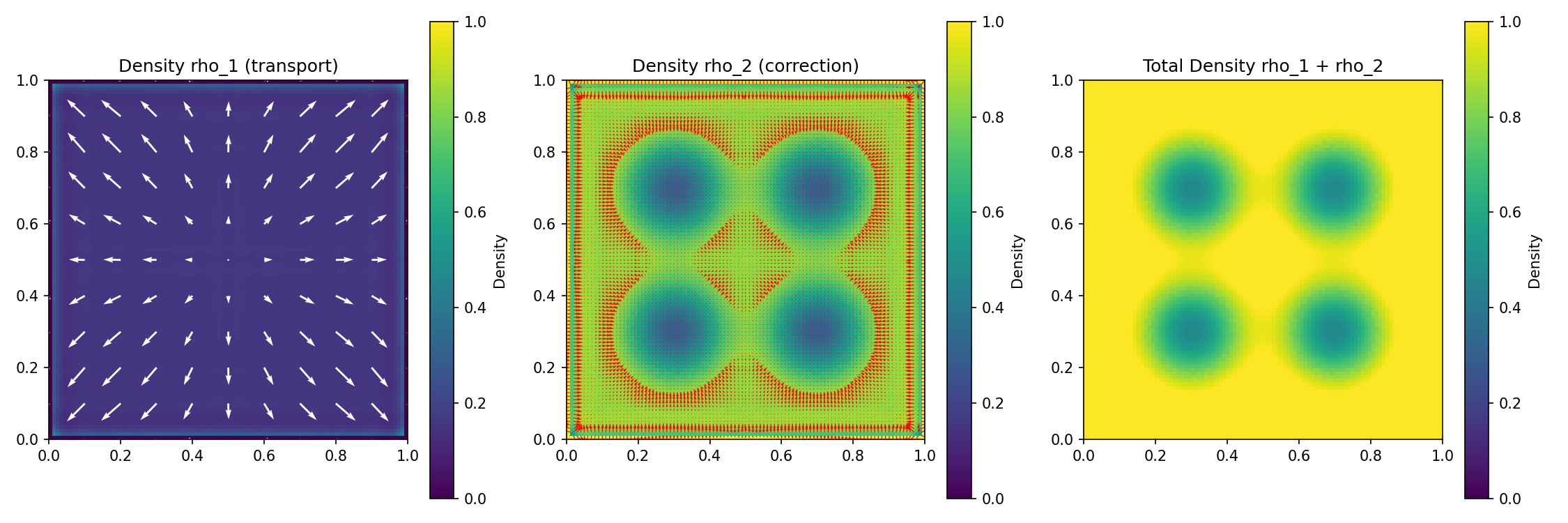}
		\caption{In contrast to the earlier scenario in Example $4$, the reflective nature of the boundary conditions serves to preserve congestion within the domain upon population $rho_1$ reaching the boundary.}
		\label{fig:img1-5}
	\end{figure}

	\subsubsection{Example 6:} A scenario similar to Example $4$ with three initial Gaussian distributions and zero-padding convolution which  effectively forces population $\rho_1$ to exit the domain. 
 
	\begin{itemize}
	\item[]	\textbf{Data set}  \item $\rho_{01}=0.5 (K_\lambda^{\mu_1} +K_\lambda^{\mu_2} + K_\lambda^{\mu_3}+K_\lambda^{\mu_4} )$ in $\Omega:=[0,1)\times [0,1]$
		with $\lambda =0.1,$ $\mu_1=(0.2, 0.5),$ $\mu_2(0.5, 0.2),$ $\mu_3=(0.5, 0.8)$

		\item $\rho_{02}=0.9(1-\rho_{01}).$
		
		\item The vector field potential here is given by $\varphi = \rho_1\star K_\sigma^\mu,$ where $\sigma =5$ and $\mu=(0,0).$  
	\end{itemize}

	\begin{figure}[H]
		\centering
		\includegraphics[width=0.9\textwidth,height=0.14\textheight,center]{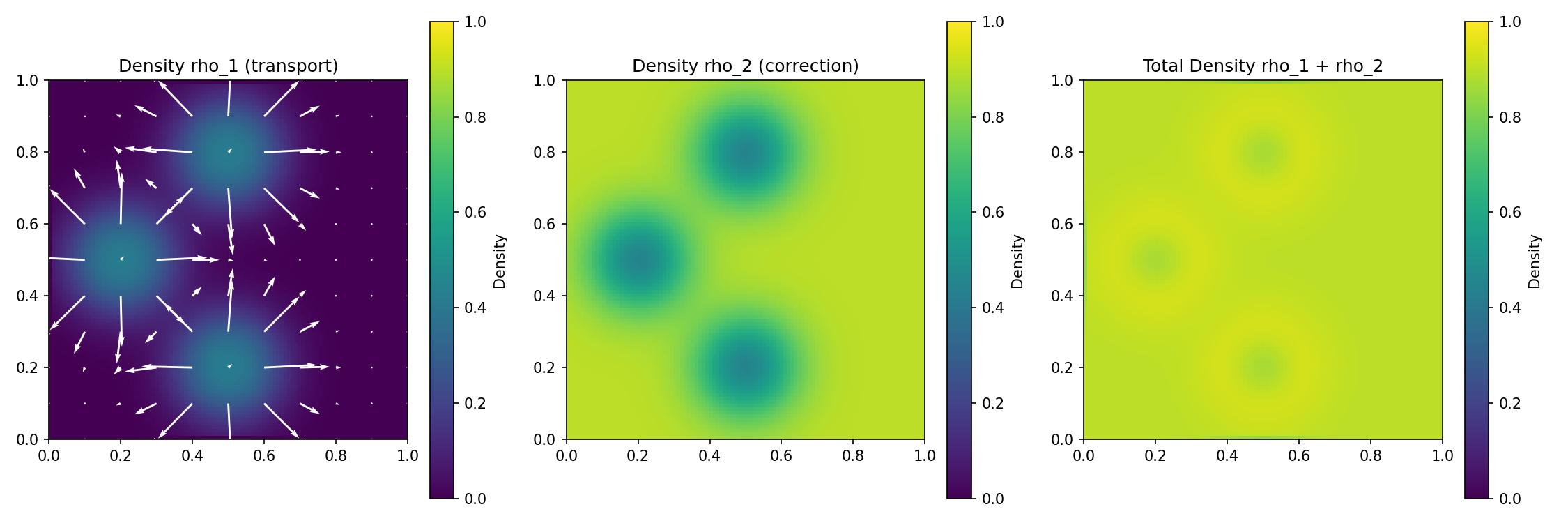}
		\includegraphics[width=0.9\textwidth,height=0.14\textheight,center]{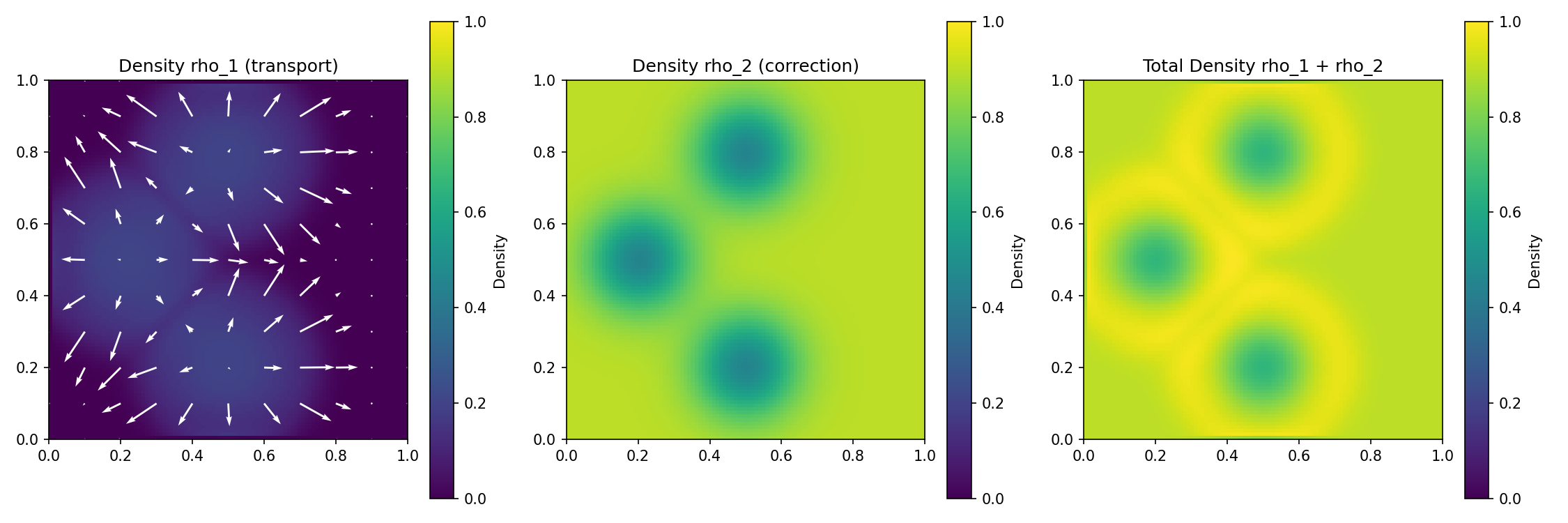}
		\includegraphics[width=0.9\textwidth,height=0.14\textheight,center]{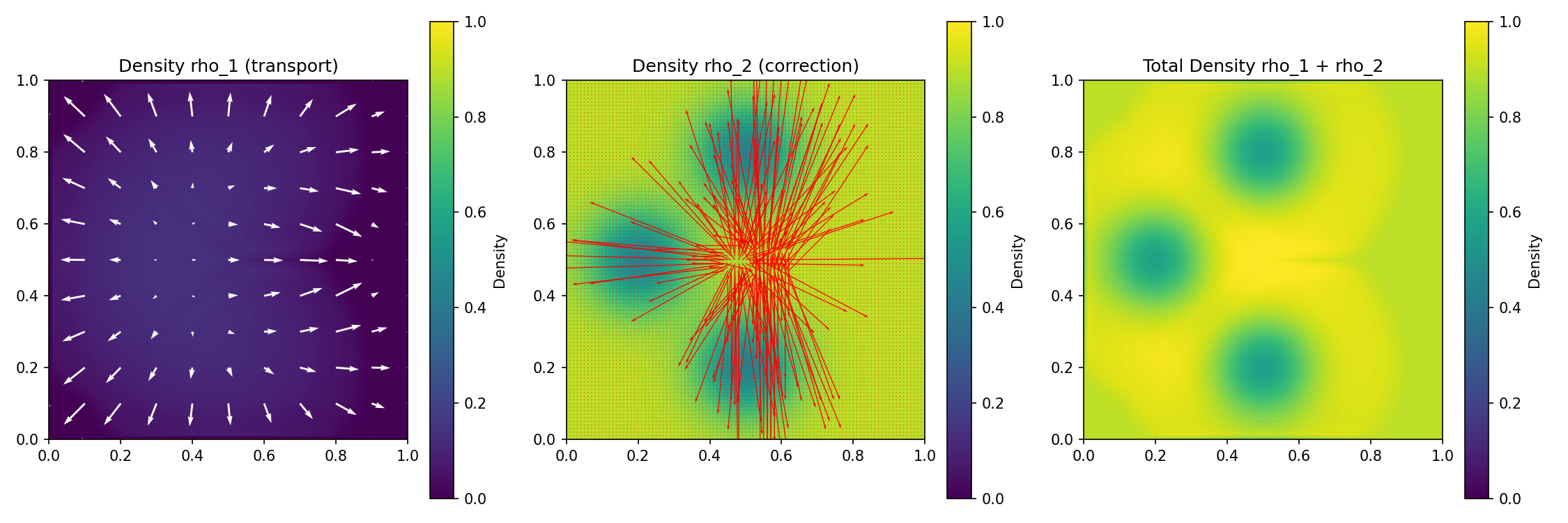} 
		\includegraphics[width=0.9\textwidth,height=0.14\textheight,center]{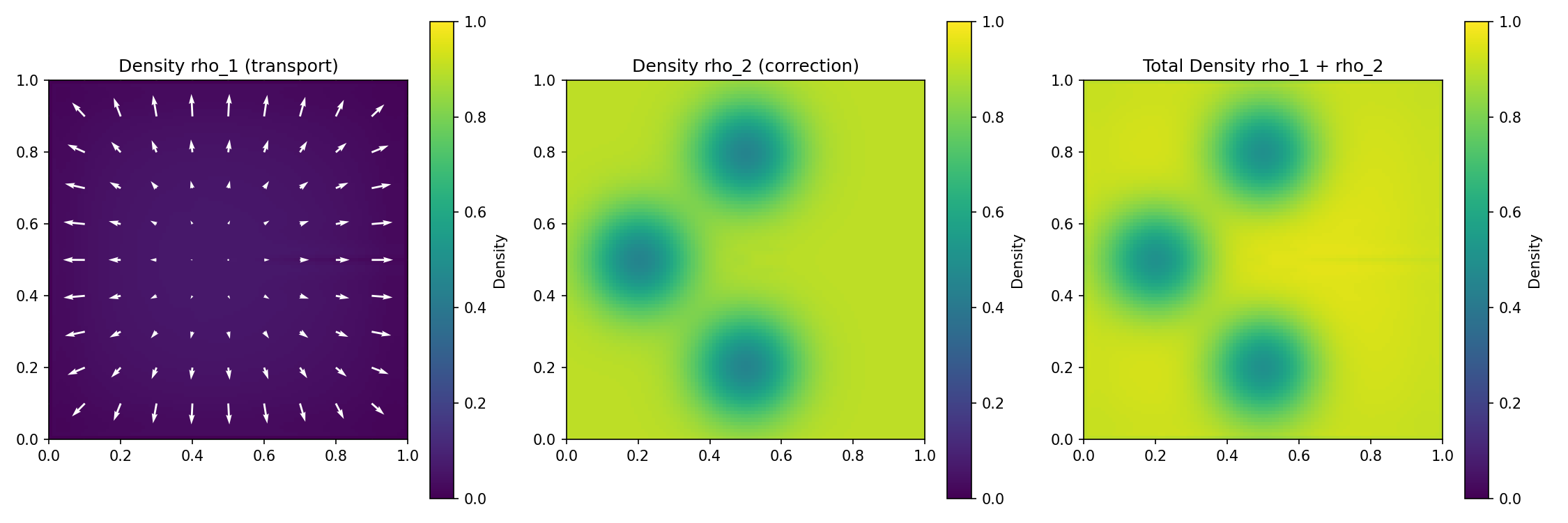}
		\caption{As in Fig. $4,$, by the end of the simulation, the entire population $\rho_1$ continues to leave the domain without any congestion.}
		\label{fig:img1-6}
	\end{figure} 
	
	\subsubsection{Example 7:} A scenario similar to Example $6,$  with the same data set, but we modify the boundary condition  by implementing "reflect" boundary conditions. 
	\begin{figure}[H]
		\centering
		\includegraphics[width=0.9\textwidth,height=0.14\textheight,center]{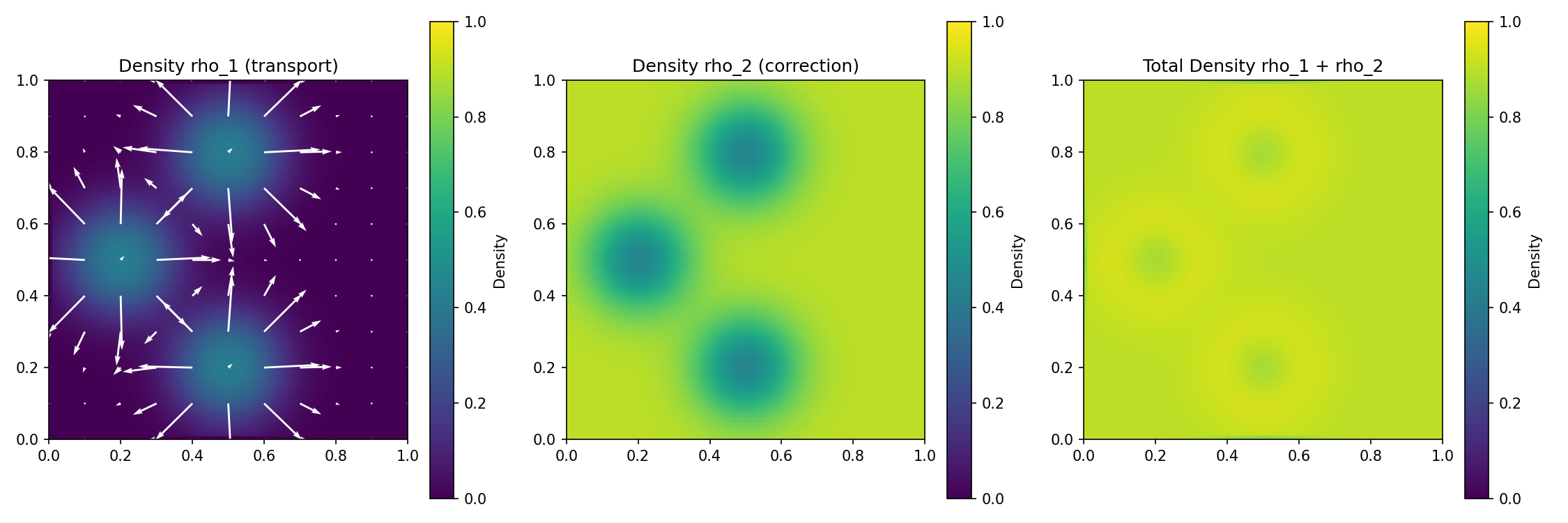}
		\includegraphics[width=0.9\textwidth,height=0.14\textheight,center]{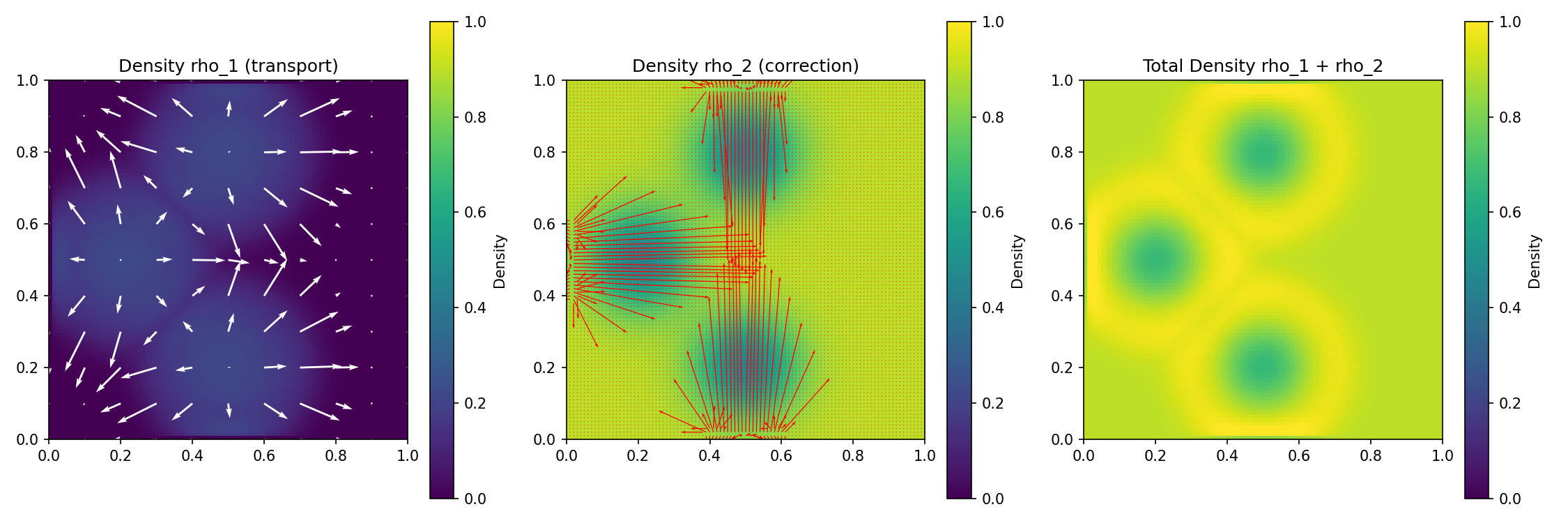}
		\includegraphics[width=0.9\textwidth,height=0.15\textheight,center]{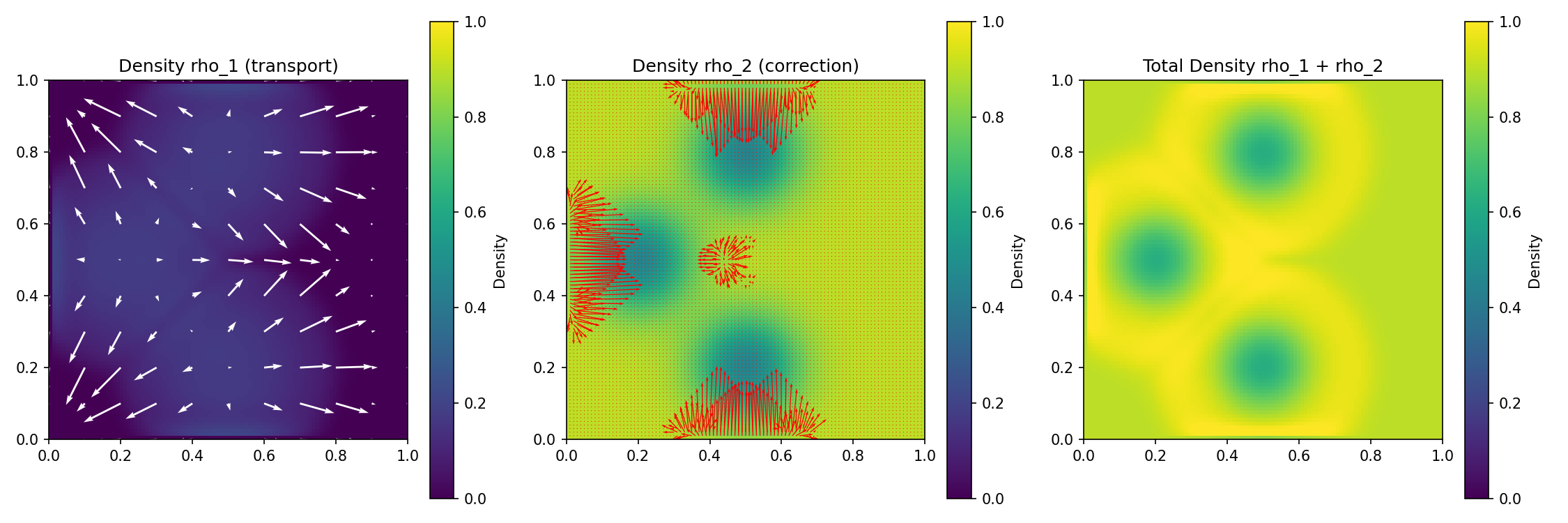}
		\includegraphics[width=0.9\textwidth,height=0.14\textheight,center]{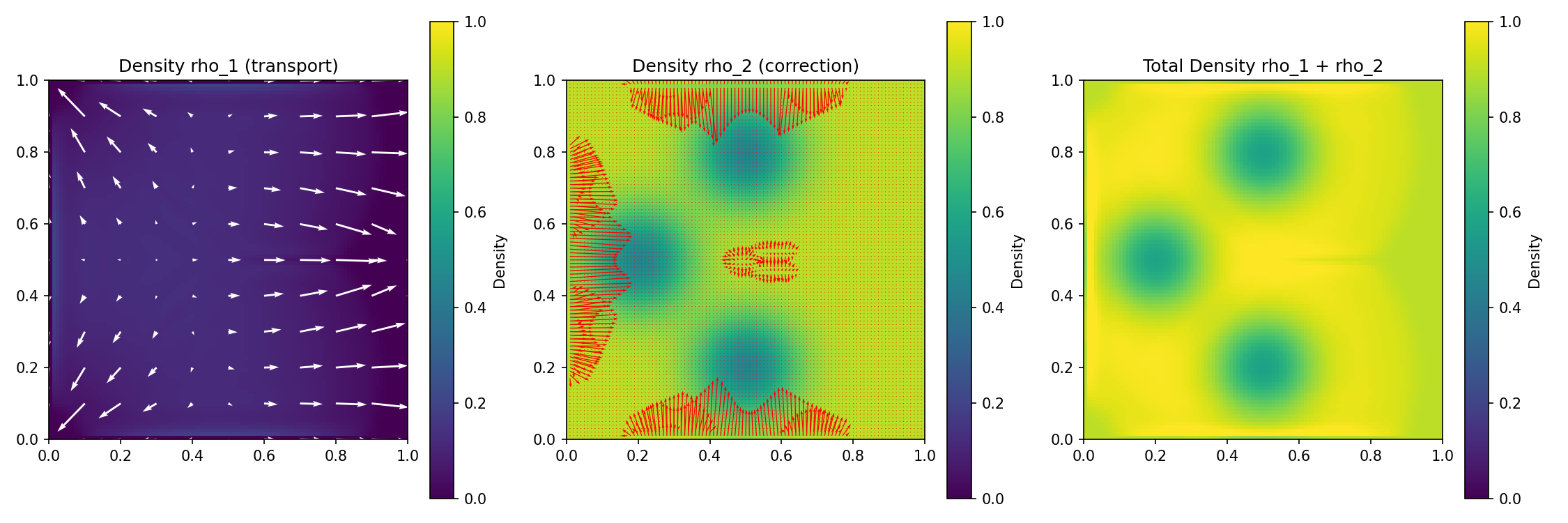}
		\includegraphics[width=0.9\textwidth,height=0.14\textheight,center]{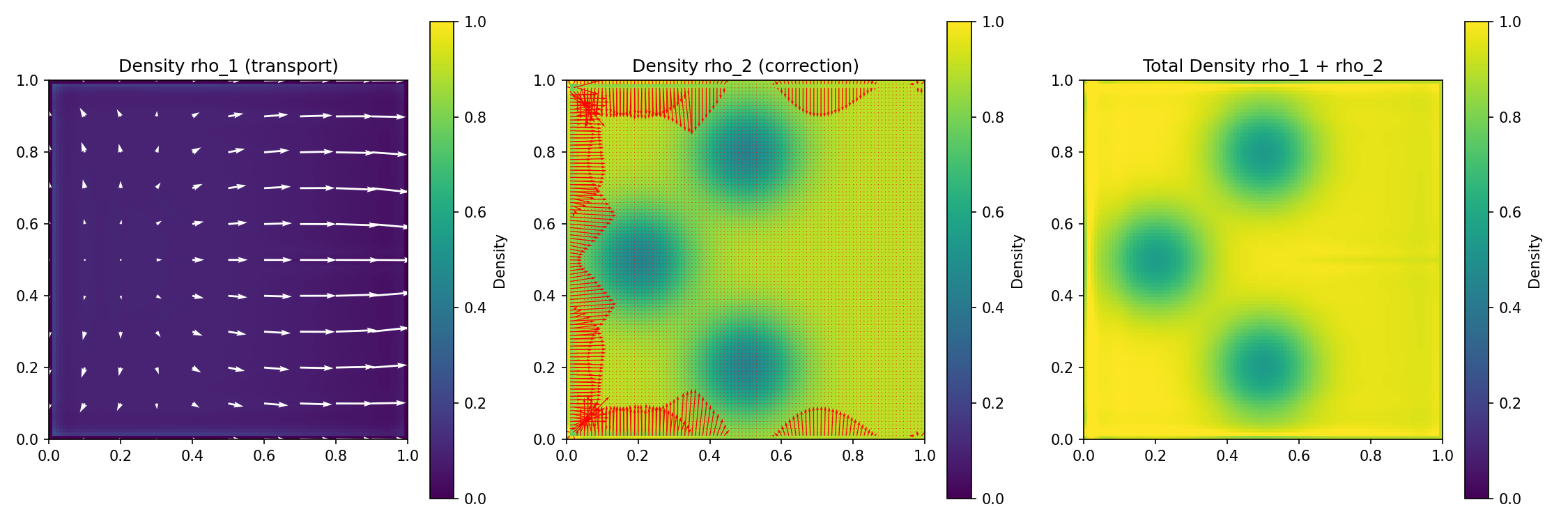}
		\includegraphics[width=0.9\textwidth,height=0.14\textheight,center]{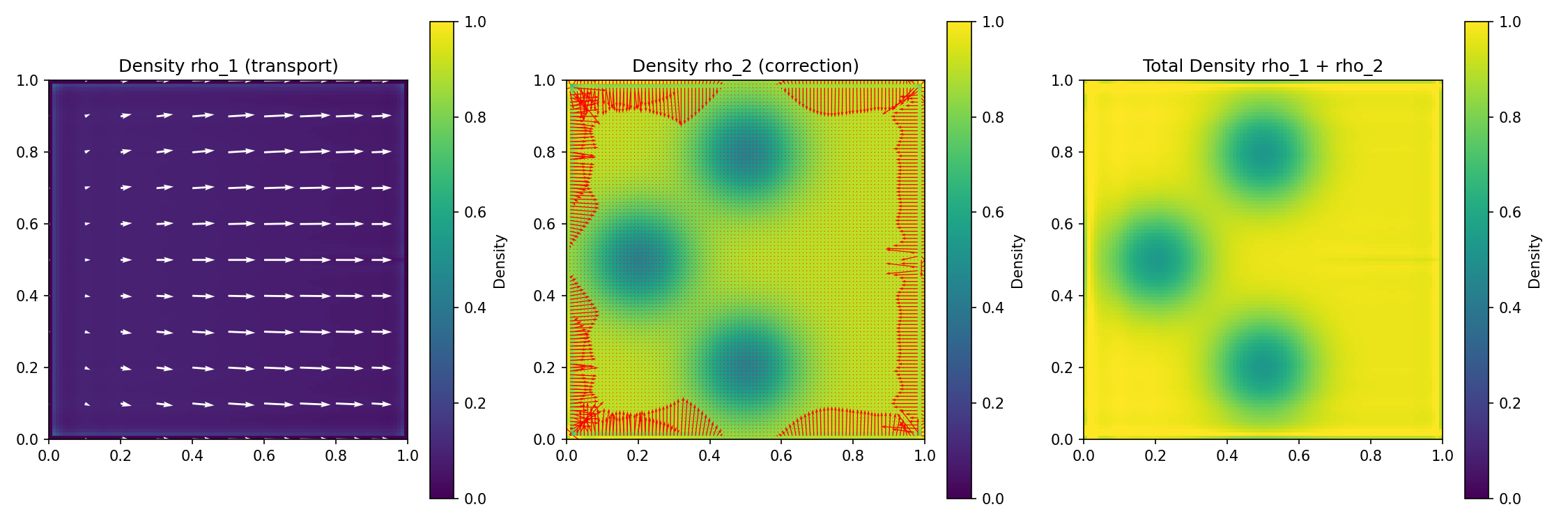}
		\caption{Unlike the previous case, the boundary conditions lead to a persistent congestion near the boundary due to the inward reflection of the population.}
		\label{fig:img1-7}
	\end{figure}

	\subsubsection{Example 8:} Population $\rho_1,$ with an initial distribution of a Gaussian function, undergoes diffusion with Dirichlet boundary conditions. This leads to population expansion and eventual exit from the domain, while population $\rho_2,$ initially surrounding $\rho_1,$ remains confined.
	
	\begin{figure}[H]
		\centering
		\includegraphics[width=0.9\textwidth,height=0.14\textheight,center]{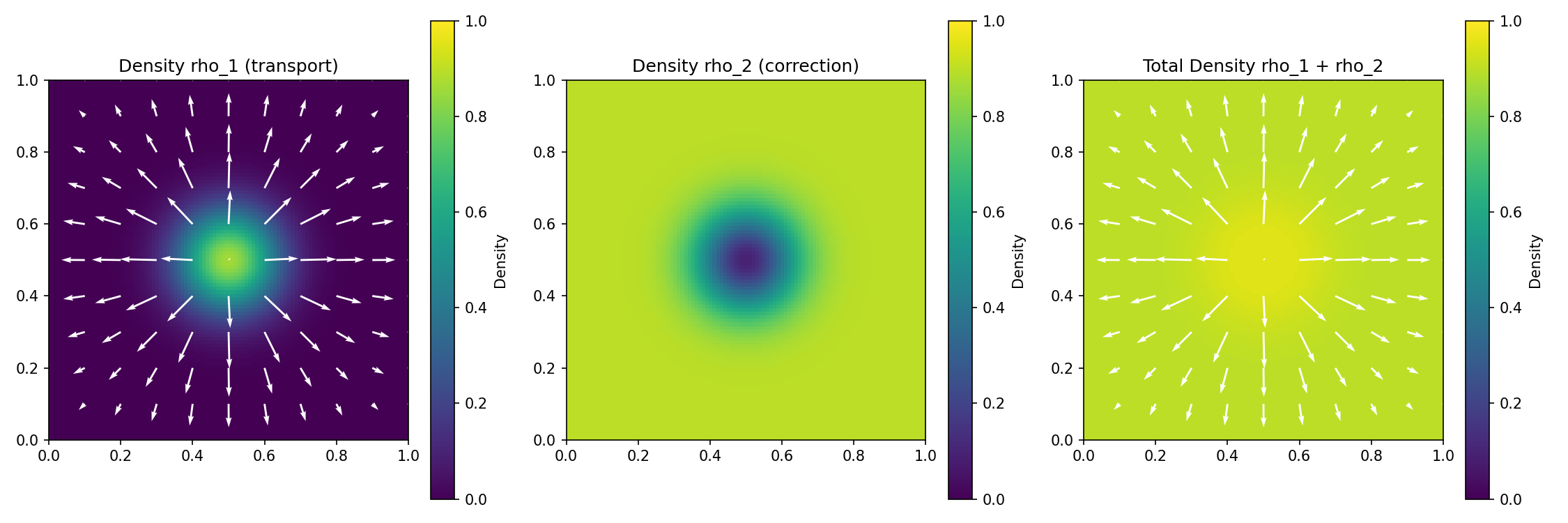}
		\includegraphics[width=0.9\textwidth,height=0.14\textheight,center]{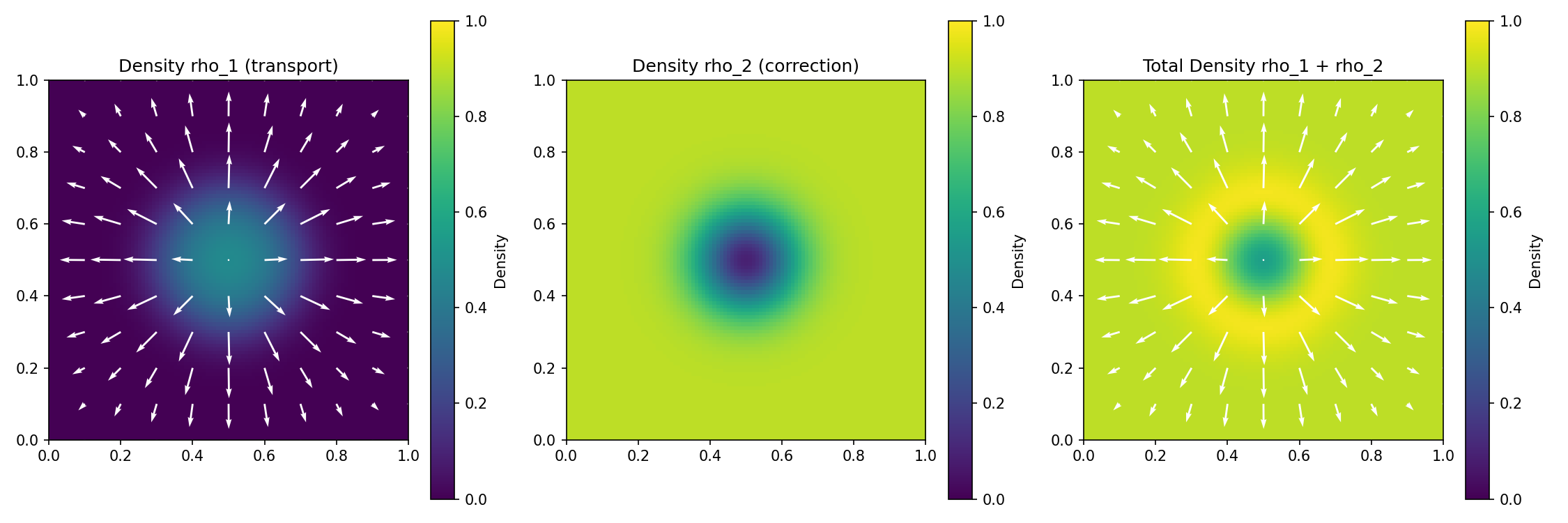}
		\includegraphics[width=0.9\textwidth,height=0.14\textheight,center]{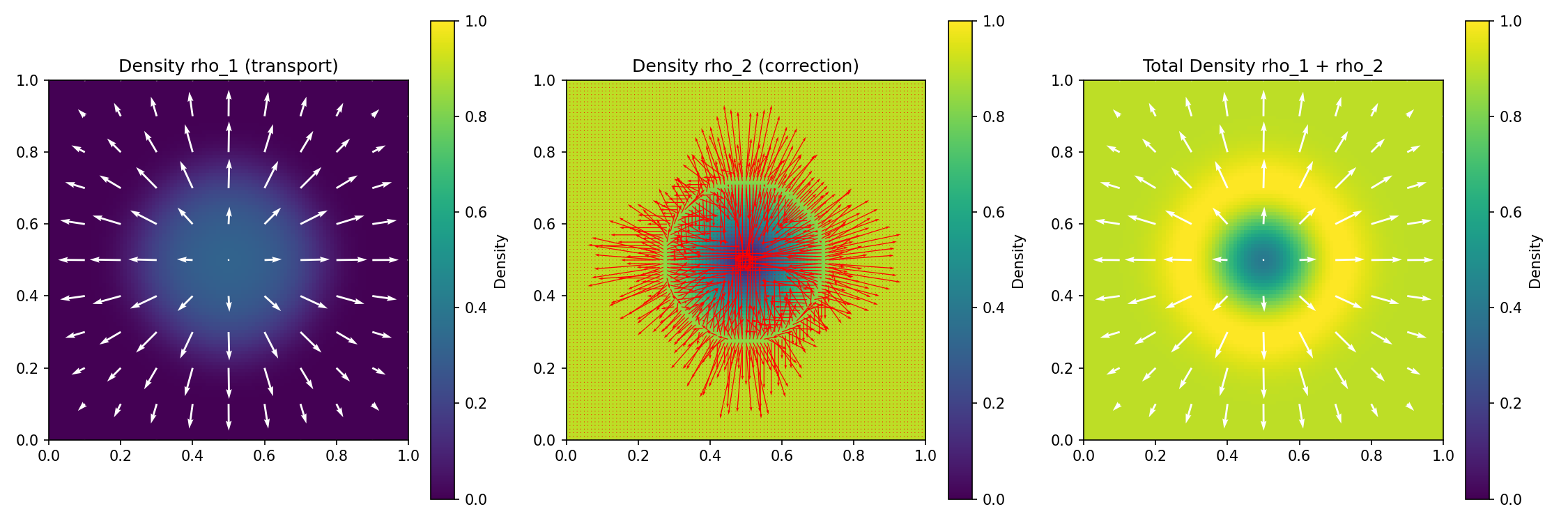}
		\includegraphics[width=0.9\textwidth,height=0.14\textheight,center]{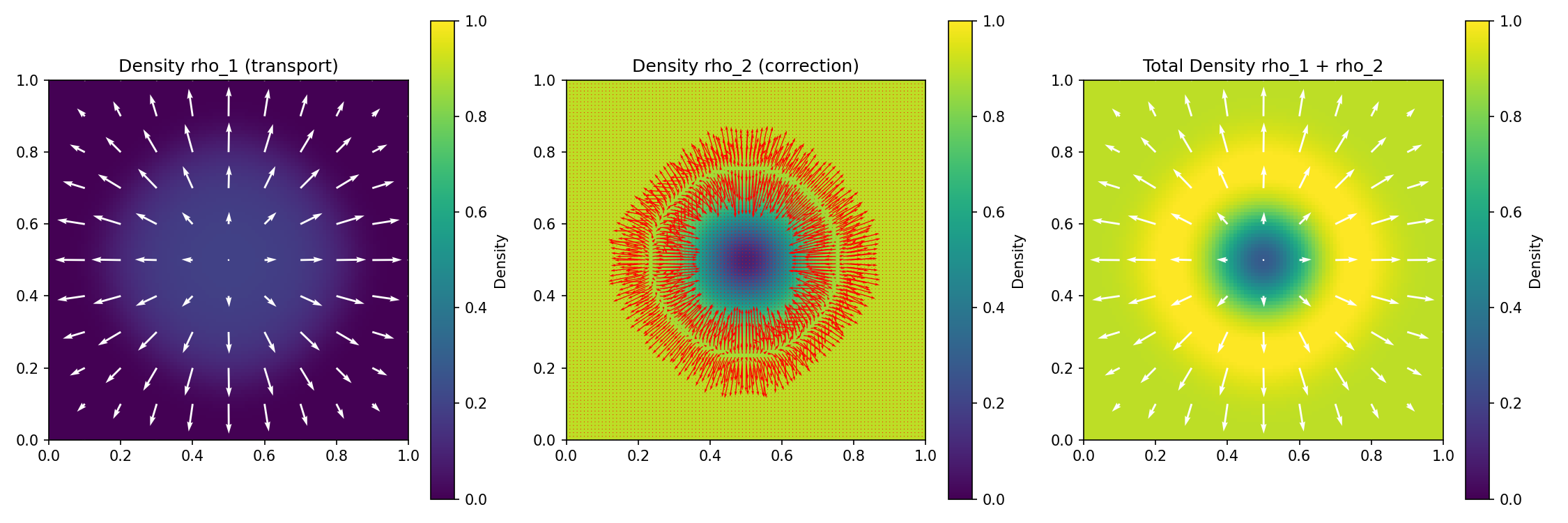}
		\includegraphics[width=0.9\textwidth,height=0.14\textheight,center]{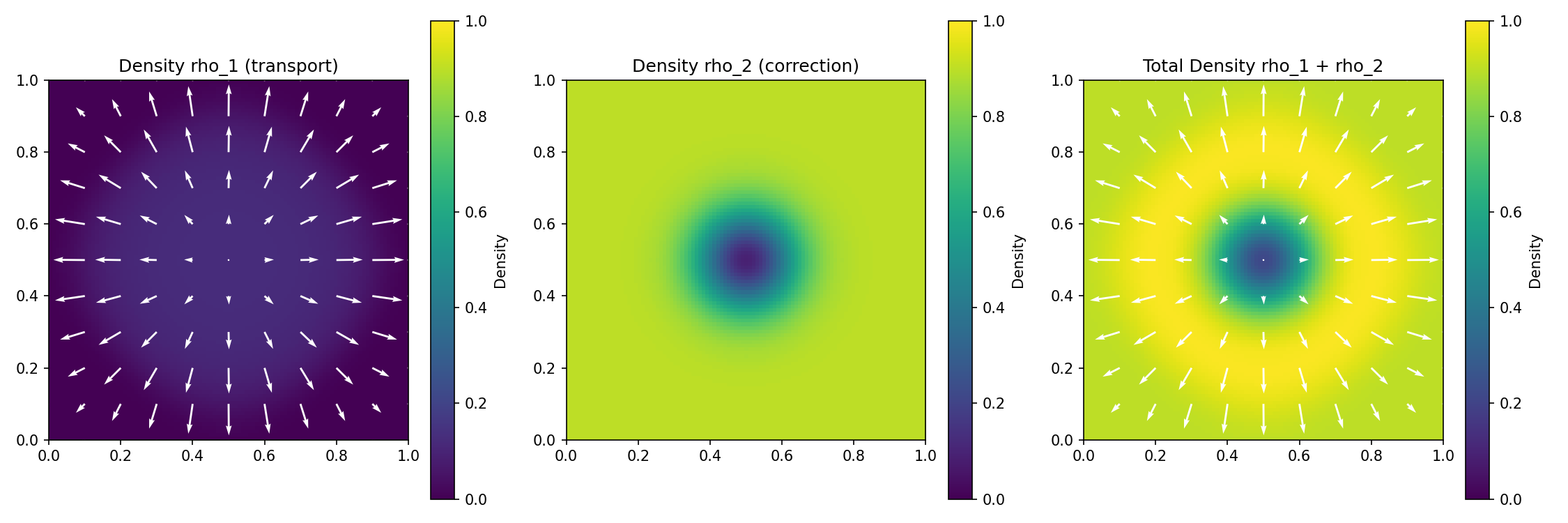}
		\caption{We observe a similar behavior to the Dirichlet case with a non-local potential: an undergo phase of congestion followed by stabilization into a non-congested dynamic.}
		\label{fig:img1-8}
	\end{figure}

	\subsubsection{Example 9:}  A scenario similar to Example $5,$ but employing different initial data distributions for both $\rho_1$ and $\rho_2.$
	\begin{figure}[H]
		\centering
		\includegraphics[width=0.9\textwidth,height=0.14\textheight,center]{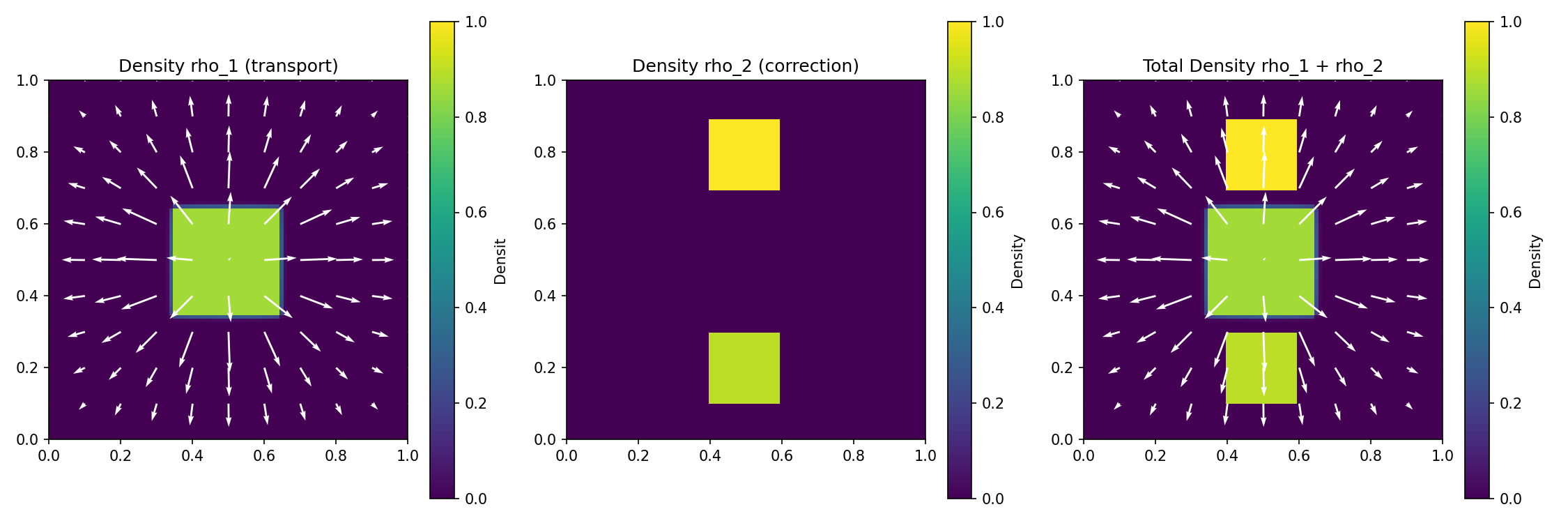}
		\includegraphics[width=0.9\textwidth,height=0.14\textheight,center]{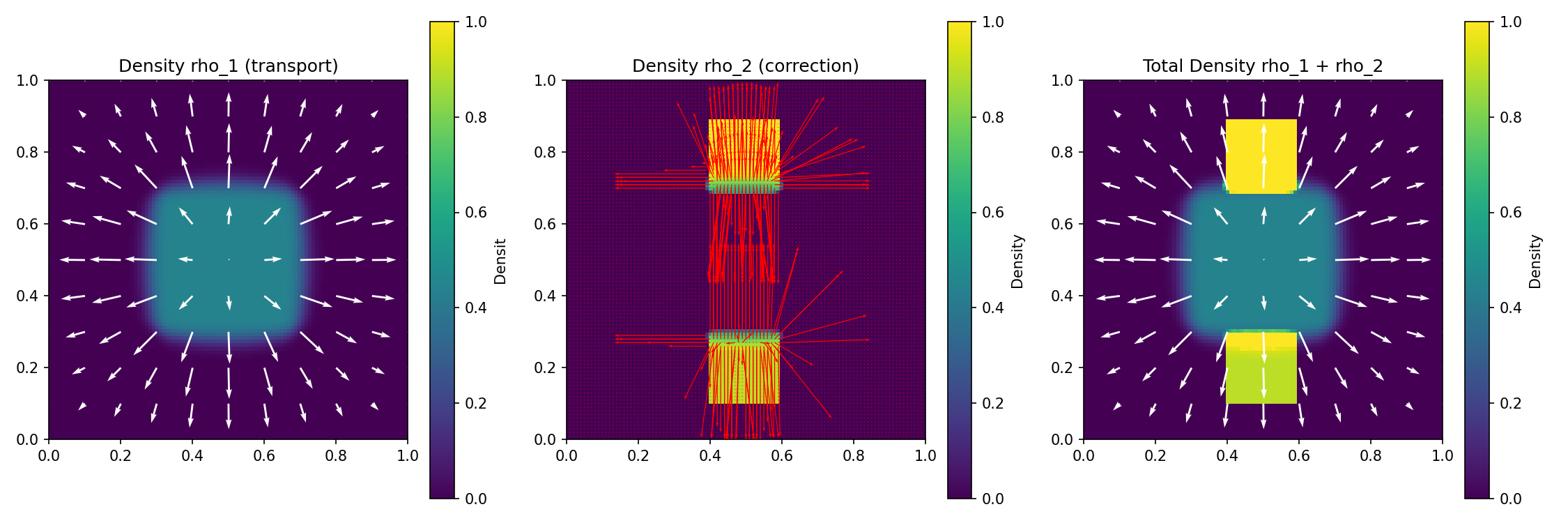}
		\includegraphics[width=0.9\textwidth,height=0.14\textheight,center]{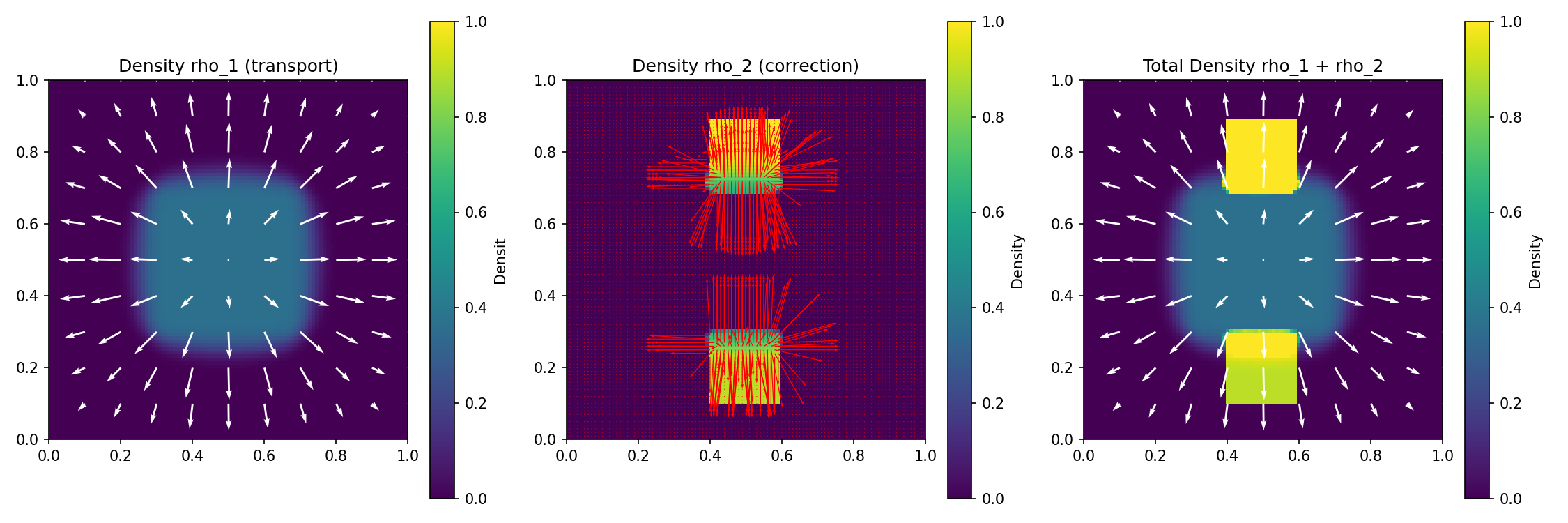}
		\includegraphics[width=0.9\textwidth,height=0.14\textheight,center]{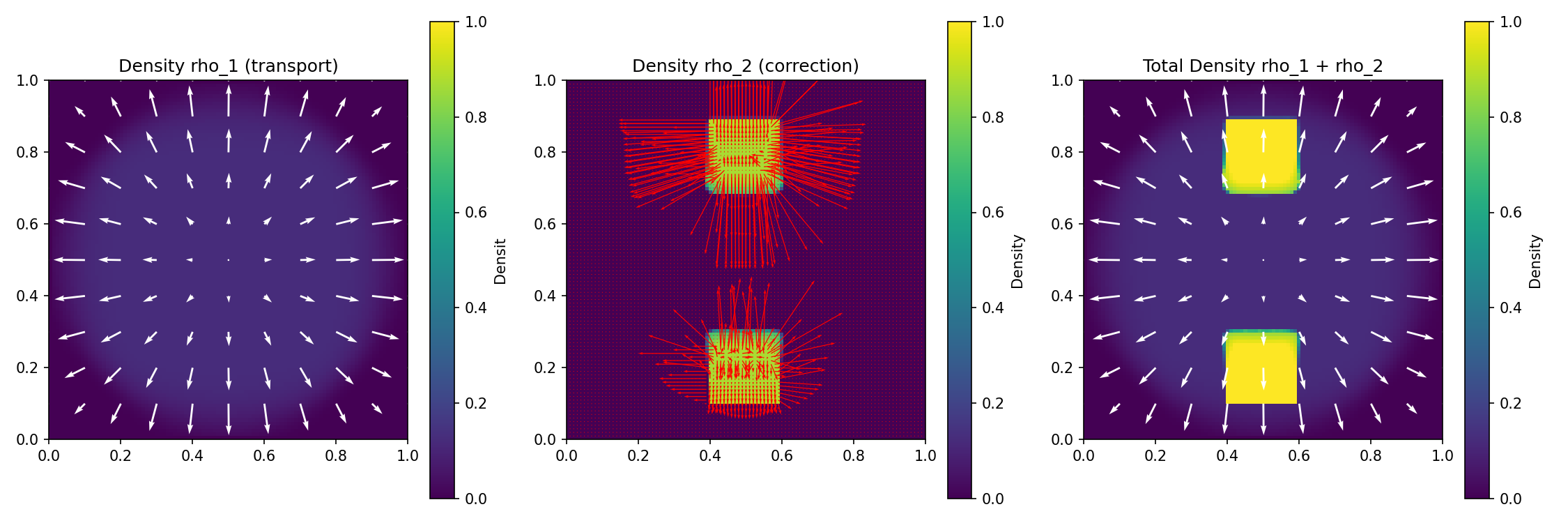}
		\includegraphics[width=0.9\textwidth,height=0.14\textheight,center]{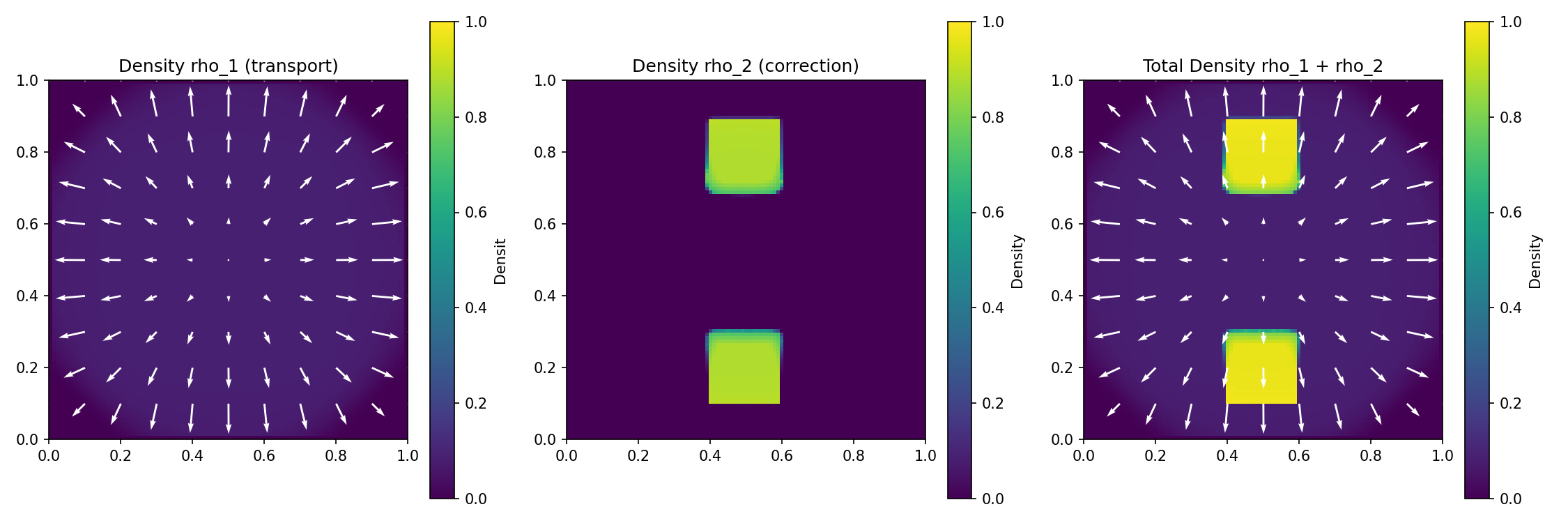}
		\caption{Similar outcomes were obtained when applying the method to different datasets, consistent with the findings in Example $8.$}
		\label{fig:img1-9}
	\end{figure}
	
	\section*{Acknowledgments } NI was partially supported by the CNRST of Morocco under the FINCOM program. He is also grateful to the EST of Essaouira for its hospitality.  Some numerical results were obtained during SG's visit to XLIM at the University of Limoges. SG thanks XLIM for their hospitality and acknowledges the support of Cadi Ayyad University.

\end{document}